\documentclass[a4paper,11pt]{article}
\pdfoutput=1

\usepackage[a4paper,tmargin=3truecm,bmargin=3truecm,rmargin=2.5truecm,
lmargin=2.5truecm,twoside,verbose=true]{geometry}
\usepackage{cancel,graphicx}

\usepackage{amsmath,amssymb}
\usepackage[amsmath, hyperref, thmmarks]{ntheorem}
\usepackage[all]{xypic}
\usepackage[pdftex]{hyperref}
\usepackage[english]{babel}

\numberwithin{equation}{section}

\allowdisplaybreaks[1]

\def\nota{}
\def\question#1{}

\newcommand\caB{{\mathcal B}}
\newcommand\caC{{\mathcal C}}
\newcommand\caD{{\mathcal D}}
\newcommand\caE{{\mathcal E}}
\newcommand\caF{{\mathcal F}}

\newcommand\caI{{\mathcal I}}
\newcommand\caL{{\mathcal L}}

\newcommand\caM{{\mathcal M_\theta}}
\newcommand\caS{{\mathcal S}}
\newcommand\caO{{\mathcal O}}
\newcommand\caU{{\mathcal U}}
\newcommand\caZ{{\mathcal Z}}
\newcommand\wx{{\widetilde x}}
\newcommand\gone{{ \mathchoice {1\mskip-4mu\mathrm{l} } {1\mskip-4mu\mathrm{l} }{1\mskip-4.5mu\mathrm{l} } {1\mskip-5mu\mathrm{l}} }}
\newcommand\gR{{\mathbb R}}
\newcommand\gK{{\mathbb K}}
\newcommand\gT{{\mathbb T}}
\newcommand\gC{{\mathbb C}}
\newcommand\gP{{\mathbb P}}
\newcommand\gB{{\mathbb B}}
\newcommand\gN{{\mathbb N}}
\newcommand\gZ{{\mathbb Z}}

\newcommand\algzero{{\mathsf 0}}
\newcommand\superA{{\mathcal A}}
\newcommand\algA{{\mathbf A}}

\newcommand\ehH{\mathcal H}

\newcommand\modE{{\boldsymbol E}}
\newcommand\modF{{\boldsymbol F}}
\newcommand\modG{{\boldsymbol G}}
\newcommand\modW{\boldsymbol W}

\newcommand\kb{{\mathfrak b}}

\newcommand\kq{{\mathfrak q}}
\newcommand\kg{{\mathfrak g}}

\newcommand\eps{{\varepsilon}}

\newcommand\Ad{{\text{\textup{Ad}}}}

\newcommand\orth{{\text{\textup{orth}}}}

\newcommand\fois{\mathord{\cdot}}
\DeclareMathOperator{\tr}{sTr} 

\DeclareMathOperator{\Dom}{\mathsf{Dom}}

\newcommand\Ker{{\text{\textup{Ker}}}}

\newcommand\dd{{\text{\textup{d}}}}

\newcommand\Supp{{\text{\textup{Supp}}}}
\newcommand\norm{\mathord{\parallel}}

\theoremsymbol{}
\theorembodyfont{\slshape}
\theoremheaderfont{\normalfont\bfseries}
\theoremseparator{}
\newtheorem{Theorem}{Theorem}[section]
\newtheorem{theorem}[Theorem]{Theorem}

\newtheorem{proposition}[Theorem]{Proposition}

\newtheorem{lemma}[Theorem]{Lemma}

\newtheorem{corollary}[Theorem]{Corollary}

\theorembodyfont{\upshape}
\theoremsymbol{\ensuremath{\blacklozenge}}

\newtheorem{example}[Theorem]{Example}

\newtheorem{remark}[Theorem]{Remark}

\newtheorem{definition}[Theorem]{Definition}

\theoremstyle{nonumberplain}
\theoremheaderfont{\scshape}
\theorembodyfont{\normalfont}
\theoremsymbol{\ensuremath{\blacksquare}}

\newtheorem{proof}{Proof}
\qedsymbol{\ensuremath{_\blacksquare}}
\theoremclass{LaTeX}

\renewenvironment{thebibliography}[1]
         {\section*{References}\frenchspacing\small
          \begin{list}{[\arabic{enumi}]}
         {\usecounter{enumi}\parsep=2pt\topsep 0pt
         \settowidth{\labelwidth}{[#1]}
         \leftmargin=\labelwidth\advance\leftmargin\labelsep
         \rightmargin=0pt\itemsep=1pt\sloppy}}{\end{list}}

\title{Deformation Quantization for Heisenberg Supergroup\footnote{Work
supported by the Belgian Interuniversity Attraction Pole (IAP) within the framework ``Nonlinear systems, stochastic processes, and statistical mechanics'' (NOSY).}}
\author{}
\date{}
\author{Pierre Bieliavsky$^a$, Axel de Goursac$^{a,}$\footnote{Corresponding author. Tel: +32 10 47 31 64. Fax: +32 10 47 25 30.} {} and Gijs Tuynman$^b$}

\begin{document}

\maketitle
\vspace*{-1cm}
\begin{center}
\textit{$^a$D\'epartement de Math\'ematiques, Universit\'e Catholique de Louvain,\\ Chemin du Cyclotron 2, 1348 Louvain-la-Neuve, Belgium\\
e-mail: \texttt{Pierre.Bieliavsky@uclouvain.be, axelmg@melix.net}}\\
\textit{$^b$Laboratoire Paul Painlev\'e, U.M.R. CNRS 8524 et UFR de Math\'ematiques, \\
Universit\'e de Lille I, 59655 Villeneuve d'Ascq Cedex, France\\
    e-mail: \texttt{Gijs.Tuynman@univ-lille1.fr}}\\
\end{center}%

\vskip 2cm

\begin{abstract}
We construct a non-formal deformation machinery for the actions of the Heisenberg supergroup analogue to the one developed by M. Rieffel for the actions of $\gR^d$. However, the method used here differs from Rieffel's one: we obtain a Universal Deformation Formula for the actions of $\gR^{m|n}$ as a byproduct of Weyl ordered Kirillov's orbit method adapted to the graded setting. To do so, we have to introduce the notion of C*-superalgebra, which is compatible with the deformation, and which can be seen as corresponding to noncommutative superspaces. We also use this construction to interpret the renormalizability of a noncommutative Quantum Field Theory.
\end{abstract}

\vskip 0.5cm

{\it Keywords:} deformation quantization; Heisenberg supergroup; harmonic analysis; Fr\'echet functional spaces; graded operator algebras; supermanifolds; renormalization

{\it Mathematics Subject Classification:} 53D55; 47L15; 46E10; 42B20; 58A50; 81T75
\vfill

\pagebreak
\tableofcontents
\pagebreak

\section{Introduction}

\subsection{Motivations}

Inspired by algebraic geometry, noncommutative geometry is a domain of mathematics which finds its origin in the correspondence between geometrical spaces and commutative algebras. More precisely, Gelfand's Theorem establishes an equivalence between the category of locally compact Hausdorff spaces and the category of commutative C${}^\ast$-algebras. In this way one can interpret noncommutative C${}^\ast$-algebras as the defining data of noncommutative topological spaces. From this point of view, the noncommutative extension of measure theory corresponds to the theory of von Neumann algebras, while the extension of Riemannian differential geometry corresponds to the theory of spectral triples. Noncommutative geometry, with this very rich way of thinking, has many applications in various areas of Mathematics and Physics \cite{Connes:1994}.

\medskip

A special class of noncommutative algebras, closely related to geometry, is provided by deformation quantization of Poisson manifolds. Deformation quantization, initiated by Bayen, Flato, Fronsdal, Lichnerowicz and Sternheimer \cite{Bayen:1978}, consists in introducing a noncommutative product on the algebra of smooth functions $C^\infty(M)$ on a Poisson manifold $M$ which is a deformation of the standard commutative product and which depends on a formal deformation parameter $\theta$ (see \cite{Kontsevich:2003} for the existence of such deformations). In the case of a symplectic Lie group $G$, the deformed product can be given by a Drinfeld twist \cite{Drinfeld:1989}. Since the algebra $C^\infty(G)$ carries a Hopf algebra structure, such a Drinfeld twist implements a Universal Deformation Formula: it deforms the whole category of module-algebras of $C^\infty(G)[[\theta]]$, the algebra of formal series of functions on $G$.

At the non-formal level (i.e. if the deformation parameter $\theta$ takes real values), one speaks about strict deformation quantization. In the case of an abelian group $G$, Rieffel has exhibited a universal twist on $C^\infty(G)$ \cite{Rieffel:1989}, which deforms also algebras on which $G$ is acting. This procedure has been extended to non-abelian K\"ahler Lie groups \cite{Bieliavsky:2002,Bieliavsky:2007,Bieliavsky:2010kg}.

\medskip

Quantum Field Theories (QFT) on noncommutative deformed spaces are a very interesting area of study, as they could exhibit new physical properties at high energy level (for instance the Planck scale \cite{Doplicher:1994tu}). On the Euclidean Moyal space (a deformation of $\gR^m$), a new type of divergence appears in the real $\phi^4$ QFT, called Ultraviolet-Infrared mixing. It is responsible for the non-renormalizability of the theory, which is very problematic in a physical context. It means that the Moyal deformation (which is universal for the action of $\gR^m$ on C${}^\ast$-algebras, as shown by Rieffel) is not universal for the $\phi^4$-action: renormalizability is not preserved when deforming.

However, Grosse and Wulkenhaar, by adding a new harmonic term to the scalar action, have solved this problem of Ultraviolet-Infrared mixing: the resulting action is renormalizable to all orders \cite{Grosse:2004yu}, in the two and four-dimensional case. Furthermore, this theory with a harmonic term has been mathematically interpreted within a superalgebraic formalism \cite{deGoursac:2008bd}. As we will see in this article, the superalgebra involved in this interpretation corresponds to a deformation of the Heisenberg supergroup.

\medskip

Coming from another direction, Supergeometry is a mathematical theory in which the objects are spaces involving, besides the usual commuting variables, also anticommuting variables. Supermanifolds, which are generalizations of usual manifolds to this anticommutative setting, can be constructed in two different but equivalent ways: the algebro-geometric approach developped by Berezin, Kostant, Leites \cite{Berezin:1976,Kostant:1977}, and the concrete approach of DeWitt \cite{DeWitt:1984} (see \cite{Tuynman:2005,Rogers:2007}). In both approaches, the algebra of functions on a supermanifold is a $\gZ_2$-graded (super)-commutative algebra. Note also that the development of this theory has been motivated by and can be applied to theoretical Physics; it suffices to think of Supersymmetry or of BRST quantization. Representations of the Heisenberg supergroup have been studied in the perspective of geometric quantization \cite{Tuynman:2010,Tuynman:2009}.

\medskip

From the noncommutative point of view, since supergeometry leads to graded commutative algebras, we could interpret noncommutative $\gZ_2$-graded algebras as corresponding to ``noncommutative superspaces''. Such an interpretation has been used in \cite{deGoursac:2008bd} for more general  gradings, but in a purely algebraic setting. Geometric tools like differential calculus and connections have indeed been introduced for graded associative algebras, as well as Hochschild cohomology in \cite{deGoursac:2009gh}. However, an analytical characterization in terms of operator algebras of the objects of Noncommutative Supergeometry was still missing. Note that $\gZ_2$-graded C${}^\ast$-algebras have been extensively studied some time ago \cite{Donovan:1970,Parker:1988}, but we will see in this paper that we have to introduce slightly different structures.

\medskip

In this article we construct a non-formal deformation quantization of the Heisenberg supergroup and establish a Universal Deformation Formula within this non-formal setting. The structure of C${}^\ast$-algebras will be shown not to be adapted to this deformation, forcing us to introduce the notions of C${}^\ast$-superalgebras and Hilbert superspaces, inspired both by operator algebras and geometrical examples. The structure of C${}^\ast$-superalgebra turns out to be compatible with the deformation, and can be seen as a definition of a ``noncommutative topological superspace'' (from the noncommutative geometry point of view). We apply the deformation to a certain class of compact supermanifolds and in particular to the supertorus. Finally, within the context of QFT, we prove that the above deformation (with an odd dimension) is universal also for the $\phi^4$-action.

\subsection{What is done in this paper}

\noindent In section 2 we start with reviewing the basic notions of supergeometry. We then go on introducing the notion of a Hilbert superspace adapted to our needs and an associated C${}^\ast$-superalgebra.
The latter (C${}^\ast$-superalgebra) is up to our knowledge not in the literature, while there already exist notions of Hilbert superspaces \cite{Rudolph:2000} but incompatible with our framework. We end by focusing on the Heisenberg supergroup.

\vspace{2mm}

\noindent In section 3 we start with developing  Kirillov's theory for Heisenberg {\sl super}groups. We pass to quantization within the setting of the Weyl ordering.
We then introduce the functional symbol spaces on which our oscillatory integral (within the $\gZ_2$-graded context) will be defined. These correspond to weighted
super-versions of Fr\'echet-valued Laurent Schwartz B-spaces defined on coadjoint orbits of the Heisenberg supergroup. Next, we use the oscillatory integral to prove that  Weyl's correspondence extends to the required symbol spaces.
Intertwining the operator composition under a super-version of the Berezin transformation, we then end by defining our $\gZ_2$-graded symbolic composition product \eqref{eq-prod-moy}.

\vspace{2mm}

\noindent In section 4 we prove that the symbol regularity established in section 3 allows us to compose smooth vectors of any given strongly continuous subisometric linear action
of the Heisenberg supergroup on a Fr\'echet algebra. The result of such a composition being again a smooth vector, one gets an associative deformed product on the smooth vectors. The latter being valid  for {\sl every} Fr\'echet algebra on which the Heisenberg supergroup acts, we call this abstract composition product formula a {\sl universal}
deformation formula. Starting with a C${}^\ast$-superalgebra, we then construct a compatible pre-C${}^\ast$-superalgebra structure on the deformed algebra of smooth vectors.

\vspace{2mm}

\noindent In section 5 we use what we have done in section 3 and 4 to define a deformation theory of compact trivial Heisenberg supermanifolds and focus on the case of the supertori. It is important to note that the natural notion of supercommutative operator algebra corresponding to compact trivial supermanifolds is indeed our notion of C${}^\ast$-superalgebra.

\vspace{2mm}

\noindent In section 6 we apply our construction to noncommutative renormalizable theories by showing that our universal deformation formula produces the Grosse-Wulkenhaar model when applied to the $\phi^4$-setting, which is not the case for Rieffel's (non-graded) construction. In other words, replacing in the commutative $\phi^4$-model the multiplication product by our graded deformed product yields directly the Grosse-Wulkenhaar model without adding a harmonic term. This provides an interpretation of the harmonic term responsible of the renormalization of this quantum field theory.

\section{Notions of Supergeometry}

The Heisenberg supergroup is a supermanifold with a smooth group law depending on an even symplectic form. In this section, we start giving, in subsection \ref{subsec-linsuper}, some basics about linear superalgebra, and in \ref{subsec-symplsuper} some properties of even symplectic forms on a superspace. The definition of a supermanifold and smooth superfunctions is recalled in \ref{subsec-superman}.

We then proceed to introduce some structures adapted to the space $L^2(M)$, for a ``trivial'' supermanifold $M$ (a notion that will be defined) such as a scalar product and a Hodge operation. Inspired by the space $L^2(M)$ equipped with these structures, we give a new definition of a Hilbert superspace in \ref{subsec-superhilbert} and show some of its properties. From this definition of a Hilbert superspace in mind, we introduce in \ref{subsec-cstsuper} the notion of a C${}^\ast$-superalgebra, which is adapted to describe operators on a Hilbert superspace as well as the space $L^\infty(M)$ for a trivial supermanifold $M$. This notion of a  C${}^\ast$-superalgebra will be also compatible with the deformation. 
Finally, we recall the definition of the Heisenberg supergroup, its Lie superalgebra, and its coadjoint orbits in \ref{subsec-heisenberg}.

\subsection{Linear superalgebra}
\label{subsec-linsuper}

We recall here some basic notions of linear superalgebra (see for example \cite{Tuynman:2005,Rogers:2007} for more details). We use in this article the concrete approach to Supergeometry, whose essence consists to replace the field $\gR$ of real numbers by a real supercommutative algebra.

Let $\superA=\superA_0\oplus\superA_1$ be a real supercommutative superalgebra such that $\superA/\mathcal{N}_{\superA}\simeq \gR$, where $\mathcal{N}_{\superA}$ is the set of all nilpotent elements of $\superA$. For instance, we can consider $\superA=\bigwedge V$, where $V$ is an real infinite-dimensional vector space. We choose such a superalgebra $\superA$ for the rest of the paper and we denote by $\gB:\superA\to\gR$ the quotient map, which is called the body map.
\begin{definition}
\begin{itemize}
\item $\modE$ is called a graded $\superA$-module if it is an $\superA$-module with decomposition $\modE=\modE_0\oplus\modE_1$, and such that $\forall i,j\in\gZ_2$, $\superA_i \modE_j\subset \modE_{i+j}$.
\item $\modF$ is a graded submodule of $\modE$ if it is a $\superA$-submodule of $\modE$, $\modF=\modF_0\oplus\modF_1$, and $\forall i\in\gZ_2$, $\modF_i=(\modF\cap \modE_i)$.
\end{itemize}
\end{definition}

\begin{proposition}
Let $\modE$ be a graded $\superA$-module and $\modF$ a graded submodule of $\modE$. Then the quotient $\modG=\modE/\modF$ has a natural structure of a graded $\superA$-module: if $\pi:\modE\to\modG$ denotes the canonical projection, then the grading is given by $\pi(x)\in \modG_i\Leftrightarrow \exists y\in\modF :\, x-y\in\modE_i$.
\end{proposition}
We recall that a graded $\superA$-module $\modE$ is free if and only if it admits a homogeneous $\superA$-basis; it is called of finite dimension if such a basis is finite. The number $n$ of the elements in a basis does not depend on the choice of a homogeneous basis, nor does the number $p$ of its even elements. We call $\dim(\modE)=p|(n-p)$ the graded dimension of $\modE$; it totally characterizes a free graded $\superA$-module.
\begin{remark}
A graded submodule of a free graded $\superA$-module does not have to be free. However, if it is free, it will be called a graded subspace.
\end{remark}

\begin{proposition}
\label{prop-linsuper-suppl}
Let $\modE$ be a free graded $\superA$-module of finite dimension and let $\modF$ be a graded subspace of $\modE$. Then $\modF$ admits a supplement in $\modE$:
\begin{equation*}
\exists \hat\modF\text{ graded subspace of }\modE,\quad \modE=\modF\oplus\hat\modF.
\end{equation*}
Moreover, $\modE/\modF$ is a free graded $\superA$-module isomorphic to $\modF$.
\end{proposition}

\subsection{Symplectic superalgebra}
\label{subsec-symplsuper}

In this subsection $\modE$ denotes a free graded $\superA$-module of finite dimension $\dim\modE=m|n$.
\begin{definition}
An even map $\omega:\modE\times\modE\to\superA$ is said to be a symplectic form if it is:
\begin{itemize}
\item bilinear: $\forall a\in\superA$, $\forall x,y\in\modE$, $\omega(ax,y)=a\omega(x,y)$ and $\omega(xa,y)=\omega(x,ay)$.
\item superskewsymmetric: $\forall x,y\in\modE$, $\omega(x,y)=-(-1)^{|x||y|}\omega(y,x)$.
\item non degenerate: ($\forall y\in\modE$ $\omega(x,y)=0$)$\Rightarrow$ $x=0$.
\end{itemize}
\end{definition}

\begin{proposition}
\label{prop-sympl-basis}
Let $\omega$ be a symplectic form on $\modE$. Then necessarily $m$ is even and there exists a homogeneous basis $(e_i,f_j,\theta_k,\eta_l)$ of $\modE$ (with $1\le i,j\le \frac m2$, $1\le k\le n_+$, and $1\le l\le n_-$, the elements $e_i$ and $f_j$ even and the elements $\theta_k$ and $\eta_l$ odd) satisfying
\begin{equation*}
\omega(e_i,f_j)=\delta_{ij},\qquad \omega(\theta_k,\theta_{k'})=2\delta_{kk'},\qquad \omega(\eta_l,\eta_{l'})=-2\delta_{ll'},
\end{equation*}
and the other relations vanishing. The numbers $n_+$ and $n_-$ depend only upon the symplectic form and not on the particular choice of the basis. Note that we have in particular $n=n_++n_-$.
\end{proposition}
In this particular basis, the form $\omega$ can be represented by the matrix $\begin{pmatrix} 0 & I & 0 & 0 \\ -I & 0 & 0 & 0 \\ 0 & 0 & 2I & 0 \\ 0 & 0 & 0 & -2I\end{pmatrix}$, where $I$ denotes the unit matrix. In the sequel we will denote by $\omega_0$ the restriction of $\omega$ to the graded subspace generated by the even basis vectors and by $\omega_1$ the restriction of $\omega$ to the graded subspace generated by the odd basis vectors.

\begin{definition}
Let $\modF$ be a subset of $\modE$. We define the symplectic orthogonal of $\modF$ as
\begin{equation*}
\orth(\modF)=\{x\in\modE,\quad\forall y\in\modF,\,\omega(x,y)=0\}.
\end{equation*}
\end{definition}

\begin{proposition}
Let $\modF$ be a subset of $\modE$. The symplectic orthogonal of $\modF$ has the following properties.
\begin{enumerate}
\item $\orth(\modF)$ is a submodule of $\modE$.
\item If $\modF$ is a graded submodule of $\modE$, then $\orth(\modF)$ is one too.
\item If $\modF$ is a graded subspace of $\modE$ then the map $\varphi:\orth(\modF)\to(\modE/\modF)^\ast$, defined by $\varphi(x)=\omega(x,\fois)$, is an isomorphism.
\end{enumerate}
\end{proposition}
\begin{proof}
\begin{enumerate}
\item Let $x_1,x_2\in\orth(\modF)$. Then $\forall y\in\modF$, $\omega(x_1+x_2,y)=\omega(x_1,y)+\omega(x_2,y)=0$, and $\forall a\in\superA$, $\omega(ax_1,y)=a\omega(x_1,y)=0$.
\item Let $x\in\orth(\modF)$, with $x=x_0+x_1$, where $x_0\in\modE_0$ and $x_1\in\modE_1$. We want to show that $x_0$ and $x_1$ are in $\orth(\modF)$.

$\forall y_0\in\modF_0$, $0=\omega(x,y_0)=\omega(x_0,y_0)+\omega(x_1,y_0)$. Since $\omega$ is even, we deduce that $\omega(x_0,y_0)=\omega(x_1,y_0)=0$. In the same way, one has: $\forall y_1\in\modF_1$, $\omega(x_0,y_1)=\omega(x_1,y_1)=0$. It then follows by linearity that $\forall y\in\modF$, $\omega(x_0,y)=\omega(x_1,y)=0$.
\item By Proposition \ref{prop-linsuper-suppl}, if $\modF$ is a graded subspace of $\modE$, then there exists a graded subspace $\hat\modF$ which is a supplement to $\modF$ and $\modE/\modF$ is a free graded $\superA$-module isomorphic to $\hat\modF$.

$\forall y\in\modE$, $\forall z\in\modF$, $\forall x\in\orth(\modF)$, $\omega(x,y+z)=\omega(x,y)$, so that $\varphi$ is well-defined.

If $x\in\text{Ker}\varphi$, then $x\in\orth(\modF)$ and $\forall y\in\modE$, $\omega(x,y+\modF)=0$, so that $x=0$.

For $\psi\in(\modE/\modF)^\ast$ we define $\tilde\psi\in\modE^\ast$ by $\forall z\in\modF$, $\tilde\psi(z)=0$ and $\forall y\in\hat\modF$, $\tilde\psi(y)=\psi(y+F)$. Let $x_\psi={}^\sharp\omega(\tilde\psi)$, where ${}^\sharp\omega:\modE^\ast\to\modE$ exists since $\omega$ is non-degenerate. Then, $\forall y\in\modE$, $\tilde\psi(y)=\omega(x_\psi,y)$, and $\forall z\in\modF$, $\omega(x_\psi,z)=0$. We conclude that $x_\psi\in\orth(\modF)$ and $\psi=\varphi(x_\psi)$. Hence $\varphi$ is an isomorphism and $\orth(\modF)$ is a graded subspace of $\modE$ with the same graded dimension as $\hat\modF$.
\end{enumerate}
\end{proof}

\begin{definition}
Let $\modF$ be a graded subspace of $\modE$.
\begin{itemize}
\item $\modF$ is said to be isotropic if $\modF\subset\orth(\modF)$.
\item $\modF$ is said to be coisotropic if $\orth(\modF)\subset\modF$.
\item $\modF$ is said to be lagrangian if $\modF=\orth(\modF)$.
\item $\modF$ is said to be symplectic if $\modF\cap\orth(\modF)=\algzero$.
\end{itemize}
\end{definition}

\begin{proposition}
\label{prop-sympl-maxiso}
Let $\omega$ be a symplectic form on $\modE$ and let $\modF$ be a maximal isotropic graded subspace of $\modE$. Then there exists a homogeneous basis $(e_i,f_j,\theta_k,\eta_l)$ of $\modE$ satisfying Proposition \ref{prop-sympl-basis} and such that the vectors
\begin{equation*}
e_i\ ,\ \theta_k+\eta_k,\quad 1\le i\le {\textstyle\frac m2}, 1\le k\le \min(n_+,n_-)
\end{equation*}
form a basis for $\modF$. The graded dimension of $\modF$ thus is $\frac m2|\min(n_+,n_-)$; $\modF$ is Lagrangian if and only if $n_+=n_-$.
\end{proposition}
\begin{proof}
For $n=0$, it was already known in the non-graded case. For the part $\omega_1$ of $\omega$ corresponding to the odd generators, we use the decomposition
\begin{equation*}
\begin{pmatrix} 2 & 0  \\ 0 & -2 \end{pmatrix}=\begin{pmatrix} 1 & 1  \\ 1 & -1 \end{pmatrix} \begin{pmatrix} 0 & 1  \\ 1 & 0 \end{pmatrix} \begin{pmatrix} 1 & 1  \\ 1 & -1 \end{pmatrix}
\end{equation*}
to obtain the result.
\end{proof}

\subsection{Supermanifolds}
\label{subsec-superman}

We refer the reader to \cite{DeWitt:1984,Berezin:1976,Tuynman:2005,Rogers:2007} for a complete exposition on supermanifolds. We just recall here some basic definitions.

We define the superspace $\gR^{m|n}=(\superA_0)^m\times(\superA_1)^n$ by using the $\gZ_2$-decomposition $\superA=\superA_0\oplus\superA_1$ of the superalgebra $\superA$. Then $\gR^{m|n}\simeq\modE_0$ for each free graded $\superA$-module $\modE$ with $m$ even generators and $n$ odd generators. The body map can be trivially extended to $\gB:\gR^{m|n}\to\gR^m$.

\begin{definition}
A subset $U$ of $\gR^{m|n}$ is called open if $\gB U$ is an open subset of $\gR^m$ and $U=\gB^{-1}(\gB U)$, namely $U$ is saturated with nilpotent elements. The topology associated to these open subsets is called the DeWitt topology. It is the coarsest topology on $\gR^{m|n}$ such that the body map is continuous. Endowed with this topology, a superspace is locally connected but not Hausdorff.
\end{definition}

The following Lemma allows us to make a connection between $\gR^{m|0}=(\superA_0)^m$ and $\gR^m$ at the level of smooth functions.
\begin{lemma}
\label{lem-superextofsmoothfunc}
To any smooth function $f\in C^\infty(\gR^m)$ one can associate the function $\tilde f:\gR^{m|0}\to\superA_0$ defined by the following prescription: $\forall x\in\gR^{m|0}=(\superA_0)^m$, with $x=x_0+n$, $x_0=\gB(x)\in\gR^m$ and $n\in\gR^{m|0}$ a nilpotent element, we have
\begin{equation*}
\tilde f(x)=\sum_{\alpha\in\gN^m}\frac{1}{\alpha !}\partial^\alpha f(x_0) n^\alpha,
\end{equation*}
with the usual notations for the multi-index $\alpha$. Note that the sum over $\alpha$ is finite due to the nilpotency of $n$.
\end{lemma}

We can now give a characterization of smooth functions on superspaces that could be used as a definition of such superfunctions.

\begin{definition}
Let $U$ be an open subset of $\gR^{m|n}$. A map $f:U\to\superA$ is said to be smooth on $U$, which we will write as $f\in C^\infty(U)$, if there exist unique functions $f_I\in C^\infty(\gB U)$ for all ordered subsets $I$ of $\{1,\dots,n\}$, such that $\forall (x,\xi)\in\gR^{m|n}$ ($x\in\gR^{m|0}$ and $\xi\in\gR^{0|n}$),
\begin{equation*}
f(x,\xi)=\sum_{I}\tilde{f_I}(x)\xi^I,
\end{equation*}
where $\xi^I$ denotes the ordered product of the corresponding coefficients. More precisely, if $I=\{i_1, \dots, i_k\}$ with $1 < i_1 < \cdots < i_k\le n$, then $\xi^I:=\prod_{i\in I}\xi^i \equiv \xi_{i_1} \cdot \xi_{i_2} \cdots \xi_{i_k}$. As a special case we define $\xi^\emptyset = 1$.
We extend this definition in the usual way to functions with values in a superspace.
\end{definition}

\begin{definition}
Let $M$ be a topological space.
\begin{itemize}
\item A chart of $M$ is a homeomorphism $\varphi:U\to W$, with $U$ an open subset of $M$ and $W$ an open subset of $\gR^{m|n}$, for $m,n\in\gN$.
\item An atlas of $M$ is a collection of charts $\caS=\{\varphi_i:U_i\to W_i,\, i\in I\}$ where $\bigcup_{i\in I}U_i=M$ and $\forall i,j\in I$, $\varphi_i\circ\varphi_j^{-1}\in C^\infty(\varphi_j(U_i\cap U_j),W_i)_0$.
\item If $M$ is endowed with an atlas, we define its body as:
\begin{equation*}
\gB M=\{y\in M,\, \exists \varphi_i\text{ with } y\in U_i\text{ and }\varphi_i(y)\in\gB W_i\},
\end{equation*}
and the body map $\gB:M\to\gB M$ on each subset $U_i$ by: $\gB_{|U_i}=\varphi_i^{-1}\circ \gB\circ \varphi_i$.
\item Endowed with an atlas such that $\gB M$ is a real manifold, $M$ is called a supermanifold. The condition on $\gB M$ means that this space is in particular a second countable (or paracompact) Hausdorff topological space, a condition that could not be imposed on $M$ itself as the topology is highly non-Hausdorff. 
\item Let $M$ be a supermanifold. A function $f$ on $M$ is called smooth, and we write $f\in C^\infty(M)$, if and only if for any chart $\varphi_i$ in an atlas for $M$, $f\circ\varphi_i^{-1}\in C^\infty(W_i)$.
\end{itemize}
\end{definition}

\begin{definition}
A Lie supergroup is a supermanifold $G$ which has a group structure for which the multiplication is a smooth map. As a consequence, the identity element of the supergroup has real coordinates (lies in $\gB G$), and the inverse map is also smooth.
\end{definition}

Batchelor's theorem allows us to get a better understanding of a supermanifold. It says that a supermanifold can be seen as a vector bundle over an ordinary (real) manifold.
\begin{theorem}[Batchelor]
\label{thm-superman-batchelor}
Let $M$ be a supermanifold of dimension $m|n$. Then, there exists an atlas $\caS=\{\varphi_i:U_i\to W_i,\, i\in \mathcal{I}\}$ of $M$ such that transition functions $(x_j,\xi_j)=(\varphi_j\circ\varphi_i^{-1})(x_i,\xi_i)$, for $i,j\in \mathcal{I}$, $x_{i,j}\in\gR^{m|0}$, $\xi_{i,j}\in\gR^{0|n}$, are of the form:
\begin{equation*}
x_j^a=f^a(x_i)\quad\text{and}\quad \xi_j^b=\sum_{c=1}^n g^b_c(x_i)\xi_i^c,
\end{equation*}
for each $a\in\{1,\dots,m\}$, $b\in\{1,\dots,n\}$, and for $f^a$ and $g^b_c$ real smooth functions.
\end{theorem}
One can interpret the functions $f^a$ as transitions functions between charts of the manifold $\gB M$ and the functions $g^b_c$ as the components of a matrix-valued function of $\gB(U_i\cap U_j)$, which in turn can be interpreted as the transition functions of a vector bundle $\caE$ of rank $n$ on $\gB M$. Different atlases satisfying Batchelor's theorem give rise to equivalent vector bundles: the equivalence class of the vector bundle $\caE$ is completely determined by the supermanifold $M$. Conversely, to any vector bundle over an ordinary manifold one can associate a supermanifold in this way. Moreover, the algebra of smooth functions on $M$ is isomorphic to the algebra of smooth sections of the exterior bundle of the vector bundle $\caE \to \gB M$:
\begin{equation*}
C^\infty(M)\simeq \Gamma^\infty(\bigwedge\caE).
\end{equation*}

\begin{remark}
\label{rmk-complex-batchelor}
As the proof of Batchelor's theorem relies heavily on a partition of unity argument, it is not generally valid for complex supermanifolds. If it applies, one speaks of a split complex supermanifold, but there do exist non-split complex supermanifolds.
\end{remark}

From now on, instead of looking at $\superA$ valued functions, we will consider ``complex'' valued smooth superfunctions, i.e., with values in $\superA_\gC=\superA\otimes\gC = \superA \oplus i\,\superA$. However, we will denote the space of all complex smooth super functions still by $C^\infty(M)$. Note that the complex conjugation on $\superA_\gC$ is given by: $\forall a\in\superA$, $\forall\lambda\in\gC$,
\begin{equation*}
\overline{a\otimes \lambda}= a\otimes\overline{\lambda},
\end{equation*}
and satisfies: $\forall a,b\in\superA_\gC$, $\overline{a\fois b}=\overline{a}\fois\overline{b}=(-1)^{|a||b|}\overline{b}\fois\overline{a}$.
\begin{lemma}
\label{lem-superman-frechet}
Let $M$ be a supermanifold of dimension $m|n$ and $U$ be an open subset of $M$. Then, the smooth superfunctions $C^\infty(U)$ form a complex $\gZ_2$-graded Fr\'echet algebra.
\end{lemma}
\begin{proof}
By Batchelor's theorem there exists a vector bundle $\caE_U\to\gB U$ of rank $n$ such that $C^\infty(U)\simeq \Gamma^\infty(\bigwedge\caE_U)$. Then, the property $(\ast\ast)$ of p.230 of reference \cite{Dieudonne3:1970} applied to the bundle $\bigwedge\caE_U\to\gB U$ ensures that there exists a frechetic Hausdorff topology on $C^\infty(U)$. The $\gZ_2$-grading of $C^\infty(U)$ corresponds to the one defined by the exterior algebra in $\Gamma^\infty(\bigwedge\caE_U)$ and is therefore compatible with the product.

An explicit formula for the semi-norms used to construct the Frechet structure, equivalent to formula (17.1.1) on page 227 of \cite{Dieudonne3:1970}, is given by
\begin{equation*}
p_{s,K}(f) = \sup_{x\in \gB K, \vert \nu\vert\le s} \Vert (D^\nu f)(x)\Vert
\ ,
\end{equation*}
where $K$ is a compact subset of a coordinate chart and where $D^\nu$ is a multi-derivation including derivatives with respect to the odd coordinates. Taking derivatives with respect to the odd coordinates and then restricting to the body of $K$ implies that the odd coordinates are set to zero. In this way one recovers the components $f_I$ of the function $f$ as given in \ref{lem-superextofsmoothfunc}, which are the (local) components of the vector bundle $\bigwedge\caE_U\to\gB U$. This explicit formula shows at the same time that the Frechet structure does not depend on the particular choice for the vector bundle  $\caE_U\to\gB U$.
\end{proof}

\begin{definition}
\label{def-superman-trivial}
In the sequel we will say that a supermanifold $M$ of dimension $m\vert n$ is trivial if the vector bundle $\caE\to \gB M$ associated to $M$ is (isomorphic to) the trivial bundle $\caE \cong \gB M \times \gR^n$. This is equivalent to saying that $M$ is isomorphic as a supermanifold to the direct product $M_o \times \gR^{0\vert n}$, where $M_o$ is the supermanifold of dimension $m\vert 0$, which is completely determined by $\gB M$.
\end{definition}

\medskip

In the rest of this subsection $M$ will denote a trivial supermanifold of dimension $m|n$. We also assume that $\gB M$ is endowed with a volume form, which we will use for integration. Then, as for a superspace, we have an identification:
\begin{equation*}
C^\infty(M)\simeq C^\infty(\gB M)\otimes\bigwedge \gR^n,
\end{equation*}
and we define $L^2(M)= L^2(\gB M)\otimes\bigwedge \gR^n$, using the above identification.
\begin{remark}
\label{rmk-superman-triv}
In fact, the above definition of $L^2(M)$ is appropriate in the sheaf approach \cite{Berezin:1976,Kostant:1977}, but not in the concrete approach: elements of these spaces are not defined as superfunctions on $M$ (Lemma \ref{lem-superextofsmoothfunc} does not apply to non-smooth functions). 
\end{remark}

We now recall the definition of Berezin integral for odd variables: if $f\in C^\infty(M)$, then we have
\begin{equation*}
\int\dd\xi\, f(x,\xi)=f_{\{1,..,n\}}(x).
\end{equation*}
In this paper we will mean by integration on $M$ (a trivial supermanifold) the process of Berezin integration over the odd variables end usual integration (with respect to the volume form) on $\gB M$. Note that it is the Berezinian and not the Jacobian which appears in the change of variables formula in (super)integration (see \cite{Rogers:2007}).

In order to relate (Berezin) integration on $M$ with the space $L^2(M)$ defined above, we need some structure.
Let $\{\theta^i\}$ be a basis of $\gR^n$. For any (ordered) subset $I=\{i_1, \dots, i_k\}\subset \{1, \dots, n\}$ (thus with $1\le i_1< \cdots <i_k\le n$) we define $\vert I\vert=k$ to be the cardinal of $I$ and we define $\theta^I=\bigwedge_{i\in I}\theta^i\equiv \theta^{i_1} \wedge\cdots \wedge \theta^{i_k} \in \bigwedge \gR^n$. For $I=\emptyset$ we define $\theta^\emptyset = 1\in \bigwedge^0 \gR^n \cong \gR$. For any two (ordered) subsets $I = \{i_1, \dots, i_l\}$ and $J = \{j_1, \dots, j_\ell\}$ of $\{1,\dots,n\}$ we define $\eps(I,J)$ to be zero if $I\cap J\neq\emptyset$ and $(-1)$ to the power the number of transpositions needed to put $i_1, \dots, i_k, j_1, \dots, j_\ell$ into order if $I\cap J=\emptyset$. The quantity $\eps$ verifies the relations
\begin{equation*}
\eps(I,J)=(-1)^{|I||J|}\eps(J,I),\qquad \eps(I,J\cup K)=\eps(I,J)\eps(I,K)\text{ if }J\cap K=\emptyset.
\end{equation*}
It follows that the product in $\bigwedge \gR^n$ is given by: $\theta^I\fois\theta^J=\eps(I,J)\theta^{I\cup J}$. Since the ``products'' $\theta^I$ form a basis of $\bigwedge \gR^n$ when $I$ runs through all (ordered) subsets of $\{1, \dots, n\}$, we can define a supersymmetric non-degenerate scalar product by the formula
\begin{equation}
\langle\theta^I,\theta^J\rangle=\eps(I,J)\delta_{J,\complement I},\label{eq-superman-superprod}
\end{equation}
which satisfies: $\langle\theta^I,\theta^J\rangle=(-1)^{|I||J|}\langle\theta^J,\theta^I\rangle$.

We also introduce the Hodge operation: $\ast\theta^I=\eps(I,\complement I)\theta^{\complement I}$, which allows us to deduce a symmetric positive definite scalar product from the supersymmetric one:
\begin{equation}
\left(\theta^I,\theta^J\right)=\langle\theta^I,\ast\theta^J\rangle=\delta_{I,J}.\label{eq-superman-posprod}
\end{equation}

These scalar products can be extended in a natural way to $L^2(M)\simeq L^2(\gB M)\otimes\bigwedge \gR^n$. The supersymmetric one corresponds to integration with the Lebesgue measure (with respect to the volume form) and Berezin integration. And indeed, in the given identification we have $L^2(\gB M) \otimes \bigwedge \gR^n \ni f=\sum_I f_I\otimes\theta^I \cong f(x,\xi)=\sum_I f_I(x)\xi^I \in L^2(M)$. Then for $f,g\in L^2(M)$, we have
\begin{align}
&\langle f,g\rangle=\int\dd x\,\dd\xi\, \overline{f(x,\xi)}g(x,\xi)=\sum_I\eps(I,\complement I)\int\dd x\overline{f_I(x)}g_{\complement I}(x),\nonumber\\
&\left(f,g\right)=\int\dd x\,\dd\xi\, \overline{f(x,\xi)}(\ast g)(x,\xi)=\sum_I\int\dd x\overline{f_I(x)}g_{I}(x).\label{eq-superman-scalprod}
\end{align}
The second scalar product, which is hermitian positive definite, allows us to define the $L^2$-norm: $\norm f\norm=\sqrt{\left( f,f\right)}$.

\begin{lemma}
\label{lem-superman-exp}
For $\xi,\xi_0\in\gR^{0|n}$ and $\alpha$ a complex parameter, $i\alpha\,\xi\fois\xi_0\in\superA_\gC\simeq \gR^{1|0}\otimes\gC$. So the extension of the exponential function to this element is defined (see Lemma \ref{lem-superextofsmoothfunc}) and we have:
\begin{equation*}
e^{i\alpha\,\xi\fois\xi_0}=\sum_J (i\alpha)^{|J|}(-1)^{\frac{|J|(|J|-1)}{2}}\xi^J\xi_0^J,
\end{equation*}
where the sum is over all (ordered) subsets $J$ of $\{1,\dots,n\}$.
\end{lemma}

\subsection{Hilbert superspaces}
\label{subsec-superhilbert}

There are many possibilities to define the notion of a Hilbert superspace \cite{Rudolph:2000}. We will use one which is compatible with the description of the space $L^2(M)$ given in subsection \ref{subsec-superman} for a trivial supermanifold $M$ (see Definition \ref{def-superman-trivial}). We will see in the sequel that it is also compatible with deformation quantization.

Let $\ehH=\ehH_0\oplus\ehH_1$ be a complex $\gZ_2$-graded vector space, endowed with a scalar product\footnote{this scalar product is chosen left-antilinear and right-linear by convention.} $\left(\fois,\fois\right)$ such that $\ehH$ is a Hilbert space\footnote{the scalar product $\left(\fois,\fois\right)$ is therefore assumed to be hermitian positive definite.} and $\left(\ehH_0,\ehH_1\right)=\algzero$.
Let $J\in\caB(\ehH)$ be a homogeneous operator of degree $n\in\gZ_2$ satisfying, for all homogeneous $x\in\ehH$,
\begin{equation}
J^2(x)=(-1)^{(n+1)|x|}x,\qquad J^\ast(x)=(-1)^{(n+1)|x|}J(x),\label{eq-superhilbert-defj}
\end{equation}
where $J^\ast$ is the adjoint operator of $J$ with respect to the scalar product $\left(\fois,\fois\right)$.
Then the sesquilinear product $\langle\cdot,\cdot\rangle$ on $\ehH$ defined by:
\begin{equation}
\langle x,y\rangle=\left( J(x),y\right),\label{eq-superhilbert-superscal}
\end{equation}
is non-degenerate and superhermitian: $\langle x,y\rangle=(-1)^{|x||y|}\overline{\langle y,x\rangle}$ for homogeneous $x,y$. Note that the operator $J$ is unitary for both scalar products.
\begin{definition}
\label{def-superhilbert}
A $\gZ_2$-graded Hilbert space $\ehH$, endowed with an operator $J$ of degree $n$ satisfying \eqref{eq-superhilbert-defj}, will be called a Hilbert superspace of parity $n$. We will denote it by $(\ehH,J,n)$.
\end{definition}

\begin{example}
\label{ex-superhilbert}
\begin{itemize}
\item In subsection \ref{subsec-superman} we have defined the space $\ehH=\bigwedge \gR^n$. If we endow it with the Hodge operation $J=\ast$, and with both scalar products \eqref{eq-superman-superprod} and \eqref{eq-superman-posprod}, it becomes a Hilbert superspace of parity $n$ mod $2$; an orthonormal basis is given by $(\theta^I)$.
\item For $M$ a trivial supermanifold of dimension $m|n$, the space $L^2(M)=L^2(\gB M)\otimes\bigwedge \gR^n$, endowed with the operator $J=\ast_\xi$ and its two scalar products \eqref{eq-superman-scalprod}, is a Hilbert superspace of parity $n$ mod $2$.
\end{itemize}
\end{example}

\begin{remark}
Any Hilbert superspace of parity $1$ is a Krein space (see \cite{Bognar:1974}).
\end{remark}
\begin{proof}
For $n=1$ we have $J^2=\text{id}$ and $J^\ast=J$, but also $\langle\ehH_0,\ehH_0\rangle=\langle\ehH_1,\ehH_1\rangle=\algzero$. Let us define the spaces $\ehH_+=\Ker(J-\text{id})$ and $\ehH_-=\Ker(J+\text{id})$. Then any $x\in\ehH$ can be written as $x=x_++x_-$ with $x_+=\frac12(x+J(x))\in\ehH_+$ and $x_-=\frac 12(x-J(x))\in\ehH_-$. For this decomposition we have the equalities
\begin{align*}
\langle x_+,x_+\rangle &={\textstyle\frac 12}((x,x)+(x,J(x)),& \langle x_+,x_-\rangle &=0,\\
\langle x_-,x_-\rangle &={\textstyle\frac 12}(-(x,x)+(x,J(x)),& \langle x_-,x_+\rangle &=0,
\end{align*}
which shows that $\ehH$ is a Krein space.
\end{proof}

\begin{remark}
Let $\ehH=\ehH_0\oplus\ehH_1$ be a Hilbert superspace of parity $0$. Then, $\ehH_0$ is a Krein space and $\langle \ehH_0,\ehH_1\rangle=\algzero$.
\end{remark}
\begin{proof}
For $n=0$ we define $J_0=J_{|\ehH_0}$ and $J_1=J_{|\ehH_1}$, which  satisfy, $\forall x_0\in\ehH_0$, $\forall x_1\in\ehH_1$,
\begin{equation*}
J_0^2(x_0)=x_0,\quad J_0^\ast(x_0)=J(x_0),\qquad J_1^2(x_1)=-x_1,\quad J_1^\ast(x_1)=-J_1(x_1).
\end{equation*}
It follows that $J_0$ and $J_1$ are unitary, that $J_0$ is selfadjoint and that $J_1$ is antiselfadjoint. Hence both are diagonalizable: $+1$ and $-1$ are the eigenvalues of $J_0$ corresponding to the decomposition $\ehH_0=\ehH_0^+\oplus\ehH_0^-$, and $+i$ and $-i$ are the eigenvalues of $J_1$ corresponding to the decomposition $\ehH_1=\ehH_1^+\oplus\ehH_1^-$.

More precisely,
$\forall x_0\in\ehH_0$, $\forall x_1\in\ehH_1$, we have $x_0=x_0^++x_0^-$ and $x_1=x_1^++x_1^-$, with $x_0^+=\frac12(x_0+J(x_0))\in\ehH_0^+$, $x_0^-=\frac12(x_0-J(x_0))\in\ehH_0^-$, $x_1^+=\frac12(x_1-iJ(x_1))\in\ehH_1^+$ and $x_1^-=\frac12(x_1+iJ(x_1))\in\ehH_1^-$. With these decompositions we can compute:
\begin{align*}
\langle x_0,x_1\rangle&=0 & &\\
\langle x_0^+,x_0^+\rangle &={\textstyle\frac 12}((x_0,x_0)+(x_0,J(x_0)),& \langle x_0^+,x_0^-\rangle &=0,\\
\langle x_0^-,x_0^-\rangle &={\textstyle\frac 12}(-(x_0,x_0)+(x_0,J(x_0)),& \langle x_0^-,x_0^+\rangle &=0, \\
\langle x_1^+,x_1^+\rangle &={\textstyle\frac i2}(-(x_1,x_1)+i(x_1,J(x_1)),& \langle x_1^+,x_1^-\rangle &=0, \\
\langle x_1^-,x_1^-\rangle &={\textstyle\frac i2}((x_1,x_1)+i(x_1,J(x_1)),& \langle x_1^-,x_1^+\rangle &=0,
\end{align*}
which shows the result.
\end{proof}

\begin{remark}
\label{rmk-superhilbert-op}
One can endow the space $\caB(\ehH)$ of bounded operators (continuous linear maps) on a Hilbert superspace $\ehH$ with the following $\gZ_2$-grading: $f\in\caB(\ehH)$ is homogeneous of degree $i\in\gZ_2$ if $\forall j\in\gZ_2$ $f(\ehH_j)\subset\ehH_{i+j}$. Then $\caB(\ehH)$ is a $\gZ_2$-graded algebra.
\end{remark}
\begin{proof}
Let us introduce the parity operator $\gP:\ehH\to\ehH$ by $\gP = \pi_0 - \pi_1$ (with $\pi_i$ the canonical projection $\pi_i:\ehH \to \ehH_i$) or equivalently $\gP x=(-1)^{|x|}x$ for all homogeneous $x\in\ehH$.  It is a bounded operator since its operator norm $\norm \gP\norm$ can be shown to be $1$. Then, for any $f\in\caB(\ehH)$ we define the maps $f_i:\ehH\to\ehH$, $i=0,1$ by 
\begin{equation*}
f_0(x)={\textstyle\frac12}(f(x)+\gP f(\gP x)),\qquad f_1(x)={\textstyle\frac12}(f(x)-\gP f(\gP x)).
\end{equation*}
They satisfy $f(x)=f_0(x)+f_1(x)$ and $f_i$ maps $\ehH_j$ into $\ehH_{i+j}$. It follows that $f_0$ and $f_1$ are bounded operators, so that $\caB(\ehH)$ is a $\gZ_2$-graded vector space. It is straightforward to see that the grading is compatible with the composition of operators.
\end{proof}

\begin{proposition}
\label{prop-superhilbert-adjoint}
Let $(\ehH,J,n)$ be a Hilbert superspace of parity $n$ and let $\caB(\ehH)$ be the space of its bounded (continuous linear) operators (with respect to the positive definite scalar product). For any $T\in\caB(\ehH)$, there exists a superadjoint $T^\dag\in\caB(\ehH)$ (with respect to the scalar product $\langle\fois,\fois\rangle$), i.e., $\forall x,y\in\ehH$,
\begin{equation*}
\langle T^\dag(x),y\rangle=(-1)^{|T||x|}\langle x,T(y)\rangle.
\end{equation*}
An explicit expression is given by
\begin{equation}
T^\dag(x)=(-1)^{(n+1)(|T|+|x|)+|T||x|}JT^\ast J(x).\label{eq-superhilbert-dagast}
\end{equation}
\end{proposition}
Moreover, its operator norm satisfies: $\norm T^\dag\norm=\norm T^\ast\norm=\norm T\norm$ (where $T^\ast$ denotes the adjoint operator with respect to the positive definite scalar product).

\begin{proof}
For any $T\in\caB(\ehH)$ we have the equality of degree $|T^\ast|=|T|$, and $\forall x,y\in\ehH$, $(x,Ty)=(T^\ast x,y)$. It follows that we have $(-1)^{(n+1)|x|}(J^2x,Ty)=(-1)^{(n+1)(|T|+|x|)}(J^2T^\ast x,y)$, which means that we have $\langle Jx,Ty\rangle=(-1)^{(n+1)|T|}$ $\langle JT^\ast x,y\rangle$. By writing $x'=J(x)$ (i.e. $x=(-1)^{(n+1)|x|}J(x')$), we obtain the result. The unitarity of $J$ gives the property on the operator norm.
\end{proof}

\begin{definition}
Let $(\ehH^{(1)}, J_1,n_1)$ and $(\ehH^{(2)},J_2,n_2)$ be two Hilbert superspaces. By a morphism of Hilbert superspaces between $\ehH^{(1)}$ and $\ehH^{(2)}$ we will mean a continuous linear map $\Phi:\ehH^{(1)}\to\ehH^{(2)}$ of degree 0 satisfying: $\forall x,y\in\ehH^{(1)}$,
\begin{equation*}
\forall x,y\in\ehH^{(1)}\quad:\quad
\left(\Phi(x),\Phi(y)\right)=\left(x,y\right)\quad\text{ and }\quad \Phi\circ J_1=J_2\circ \Phi.
\end{equation*}
If a morphism $\Phi$ exists, it is necessarily injective and the parities of $\ehH^{(1)}$ and $\ehH^{(2)}$ must be equal. Moreover, $\Phi$ also is unitary with respect to the associated superhermitian scalar products \eqref{eq-superhilbert-superscal} of $\ehH^{(1)}$ and $\ehH^{(2)}$:
\begin{equation*}
\forall x,y\in\ehH^{(1)}\quad:\quad
\langle\Phi(x),\Phi(y)\rangle=\langle x,y\rangle.
\end{equation*}
\end{definition}

\begin{proposition}
\label{prop-superhilbert-tensprod}
Let $(\ehH^{(1)}, J_1,n_1)$ and $(\ehH^{(2)},J_2,n_2)$ be two Hilbert superspaces.
\begin{itemize}
\item (construction of the direct sum) If the parities are equal, $n_1=n_2=n$, then the direct sum: $\ehH=\ehH^{(1)}\oplus\ehH^{(2)}$, endowed with the homogeneous operator $J=J_1+J_2$ of degree $n$ is a Hilbert superspace of parity $n$.
\item  (construction of the tensor product)
We will denote by $\ehH=\ehH^{(1)}\otimes\ehH^{(2)}$ the completion of the algebraic tensor product with respect to the scalar product given by 
\begin{equation*}
\forall x_1,y_1\in\ehH^{(1)} \ \forall x_2,y_2\in\ehH^{(2)}\quad:\quad
\left(x_1\otimes x_2,y_1\otimes y_2\right)=(x_1,y_1)(x_2,y_2).
\end{equation*}
It is a $\gZ_2$-graded Hilbert space with respect to the total degree. We can endow it with the structure of a Hilbert superspace of parity $n=n_1+n_2$ by defining the operator $J\in\caB(\ehH)$ of degree $n$ by:
\begin{equation*}
J(x_1\otimes x_2)=(-1)^{(n_1+|x_1|)|x_2|}J_1(x_1)\otimes J_2(x_2).
\end{equation*}
This operator satisfies the conditions \eqref{eq-superhilbert-defj}, and the associated superhermitian scalar product has the following natural property:
\begin{equation*}
\langle x_1\otimes x_2,y_1\otimes y_2\rangle=(-1)^{|x_2||y_1|}\langle x_1,y_1\rangle\langle x_2,y_2\rangle.
\end{equation*}
\end{itemize}
\end{proposition}
\begin{proof}
A direct consequence of the conditions in Definition \ref{def-superhilbert}
\end{proof}

\subsection{\texorpdfstring{C${}^\ast$-superalgebras}{C* superalgebras}}
\label{subsec-cstsuper}

In this subsection we define the notion of C${}^\ast$-superalgebras, appropriate for deformation quantization. To do so, we will in fact follow the example of the function space $L^\infty(M)$ acting by (usual) multiplication on the Hilbert superspace $L^2(M)$.

We start by recalling that, in the non-graded case, if $\gB M$ is a smooth manifold endowed with a volume form, then the map $\mu:L^\infty(\gB M)\to\caB(L^2(\gB M))$, defined by
\begin{equation}
\forall f\in L^\infty(\gB M)\ \forall \varphi\in L^2(\gB M)\quad:\quad
\mu_f(\varphi)=f\fois\varphi,\label{eq-cstsuper-mu}
\end{equation}
is an isometric C${}^\ast$-algebras morphism for the essential-sup norm $\norm\fois\norm_\infty$ of $L^\infty(\gB M)$.

\begin{definition}
\label{def-supercst}
\begin{itemize}
\item Let $\algA$ be a complex $\gZ_2$-graded algebra. By a superinvolution on $\algA$ we will mean a homogeneous antilinear map of degree 0 on $\algA$, denoted by ${}^\dag$, satisfying
\begin{equation*}
\forall a,b\in\algA \quad:\quad
(a^\dag)^\dag=a\quad\text{ and }\quad (a\fois b)^\dag=(-1)^{|a||b|}b^\dag\fois a^\dag.
\end{equation*}
If $\algA$ is a $\gZ_2$-graded Banach algebra, we will also require that $\norm a^\dag\norm=\norm a\norm$ for all $a\in\algA$.
\item A C${}^\ast$-superalgebra $\algA$ is a superinvolutive $\gZ_2$-graded Banach algebra which can be isometrically represented on a Hilbert superspace $(\ehH,J,n)$ by a map $\rho:\algA\to\caB(\ehH)$ of degree 0, satisfying\footnote{where the ${}^\dag$ on the RHS is the superadjoint defined in Proposition \ref{prop-superhilbert-adjoint}.} $\rho(a^\dag)=\rho(a)^\dag$ for all $a\in\algA$.
\item A morphism of C${}^\ast$-superalgebras is an isometric superinvolutive algebra-morphism of degree 0 between two C${}^\ast$-superalgebras.
\end{itemize}
\end{definition}

Note that Remark \ref{rmk-superhilbert-op} and Proposition \ref{prop-superhilbert-adjoint} tell us that the algebra of bounded operators on a Hilbert superspace satisfies the axioms of a C${}^\ast$-superalgebra.

A C${}^\ast$-superalgebra can be seen as a complete, closed for the superinvolution, graded subalgebra of a $\gZ_2$-graded C${}^\ast$-algebra endowed with a continuous superinvolution, where the relation between involution and superinvolution is of the type \eqref{eq-superhilbert-dagast}. Note that the subalgebra is not required to be closed under the involution of the C${}^\ast$-algebra.

\begin{remark}
Using the notations of Definition \ref{def-supercst}, any C${}^\ast$-superalgebra $\algA$ can be seen as a subalgebra of the C${}^\ast$-subalgebra $\tilde\algA$ of $\caB(\ehH)$ generated by $\rho(\algA)$, $J$ and the parity operator $\gP$ introduced in the proof of Remark \ref{rmk-superhilbert-op}. Moreover, we have the following relations:
\begin{align*}
&\forall T\in\tilde\algA,\quad \gP T=(-1)^{|T|}T\gP,\\
&\gP^2=\gone,\qquad \gP^\dag=(-1)^n\gP,\qquad J^2=\gP^{n+1},\qquad J^\dag=J\gP.
\end{align*}
\end{remark}

\begin{example}
\label{ex-supercst-linf}
\begin{itemize}
\item Any C${}^\ast$-algebra $\algA$ is a C${}^\ast$-superalgebra, whose even part is $\algA$ and whose odd part is $\{0\}$.
\item The algebra $\algA=\bigwedge \gR^n$ acting on $\ehH=\bigwedge \gR^n$ by multiplication: $\theta^I\fois\theta^J=\eps(I,J)\theta^{I\cup J}$ (see subsection \ref{subsec-superman} and Example \ref{ex-superhilbert}) is a C${}^\ast$-superalgebra. Here we have the relations
\begin{align*}
&(\theta^I)^\dag=\theta^I,\qquad (\theta^I)^\ast(\theta^J)=\eps(I,J\setminus I)\theta^{J\setminus I}\text{ for }I\subset J,\\
& \norm\theta^I\norm=1,\qquad (\theta^I)^\dag\fois\theta^I=0,\qquad \norm(\theta^I)^\ast\fois\theta^I\norm=\norm\theta^I\norm^2.
\end{align*}
It follows in particular that the operator $(\theta^I)^\ast$ is not in the algebra $\algA\subset\caB(\ehH)$, only in $\caB(\ehH)$.
\item Copying \eqref{eq-cstsuper-mu} to the case of a trivial supermanifold $M$ of dimension $m|n$ (see Definition \ref{def-superman-trivial}) gives us a map $\mu:L^\infty(M)\to\caB(\ehH)$, where $\ehH=L^2(M)$ (see Example \ref{ex-superhilbert}). This map $\mu$ is an isometric representation of the supercommutative C${}^\ast$-superalgebra $\algA=L^\infty(M)$. One can show that the operator norm of $\mu(f)$ (for $f\in L^\infty(M)$) on $\caB(\ehH)$ is given by
\begin{equation*}
\norm \mu(f)\norm=\sum_I \norm f_I\norm_\infty \norm\theta^I\norm=\sum_I\norm f_I\norm_\infty.
\end{equation*}
One can also show that $\mu(f)^\dag=\mu(\overline f)$ and $\norm \overline{f}\norm=\norm f\norm$.
\end{itemize}
\end{example}

\begin{lemma}
If $\Phi:\ehH_1\to\ehH_2$ is a Hilbert superspace isomorphism, then the map $\widetilde\Phi:\caB(\ehH_1)\to\caB(\ehH_2)$ defined by $\caB(\ehH_1)\ni T\mapsto \widetilde\Phi(T)=\Phi\circ T\circ\Phi^{-1}$is a C${}^\ast$-superalgebra morphism .
\end{lemma}

\subsection{Heisenberg supergroup}
\label{subsec-heisenberg}

Let $\modE$ be a free graded $\superA$-module of finite dimension $m|n$ endowed with an even symplectic form $\omega$. As in Proposition  \ref{prop-sympl-basis}, we denote by $\omega_0$ and $\omega_1$ the restriction of $\omega$ to the subspace generated by the even, respectively odd, basis vectors.
\begin{definition}
The Heisenberg superalgebra associated to $(\modE,\omega)$ is given by the $\superA$-module $\kg=\modE\oplus\superA Z$, where $Z$ is an even generator, and by the relations:
\begin{equation*}
\forall x,y\in\modE\ \forall a,b\in\superA\quad:\quad [x+aZ,y+bZ]=\omega(x,y)Z.
\end{equation*}
It is a $\superA$-Lie algebra of dimension $m+1|n$. Its center is given by  $\caZ(\kg)=\superA Z=\kg'$, with $\kg'=[\kg,\kg]$.
\end{definition}

As in the non-graded case, the Heisenberg supergroup $G$ is homeomorphic to the even part of the Heisenberg algebra $\kg_0$. Its group law can be computed from the Baker-Campbell-Hausdorff formula:
\begin{equation*}
\forall x,y\in\modE_0 \ \forall a,b\in\superA_0
\quad:\quad
(x+aZ)\fois(y+bZ)=x+y+(a+b+\frac 12\omega(x,y))Z.
\end{equation*}
$G$ is a non-abelian group with neutral element $0+0Z$ and inverse given by: $(x+aZ)^{-1}=-x-aZ$. Note that the Heisenberg supergroup as defined above is a particular case of ``Heisenberg like groups'' treated in \cite{Tuynman:2010}.

We define a linear isomorphism from $\kg$ to its (right) dual $\kg^\ast$, denoted as $x+aZ\mapsto {}^b(x+aZ)$, by the formula: 
\begin{equation*}
\forall x,y\in\modE \ \forall a,b\in\superA
\quad:\quad
{}^b(x+aZ)(y+bZ)=\omega(x,y)+ab.
\end{equation*}
Using this isomorphism, the adjoint and coadjoint actions are given by the formula (with $y\in\modE$, $b\in\superA$ and  $x\in\modE_0$, $a\in\superA_0$ so $x+aZ\in \kg_0=G$)
\begin{align*}
& \Ad_{(x+aZ)}(y+bZ)=y+bZ+[x+aZ,y+bZ]=y+(b+\omega(x,y))Z,\\
& \Ad^\ast_{(x+aZ)}{}^b(y+bZ)={}^b(y-bx+bZ).
\end{align*}
It follows that the coadjoint orbit $\caO_\zeta=\{\Ad^\ast_g\zeta,\,g\in G\}$ of $\zeta={}^b(y+bZ)\in\gB\kg^\ast$ is in bijection (via the isomorphism ${}^b$) with the even part $\modE_0$ of the module $\modE$, provided $b\neq 0$. For $b=0$, the coadjoint orbit $\caO_\zeta$ is a single point.

\section{Deformation quantization}
\label{sec-defqu}

To perform the deformation quantization of the Heisenberg supergroup, we adapt the general method of \cite{Bieliavsky:2010kg} to the graded setting.
First, in subsection \ref{subsec-induc}, by using Kirillov's orbits method, we associate to each coadjoint orbit $\caO_\zeta$ (with $\zeta\in \gB\kg^\ast$) a unitary induced representation of the Heisenberg supergroup. Note that the unitarity of the representation refers to the supersymmetric scalar product, not the positive definite one.
Using this induced representation, we define in \ref{subsec-qu} a first quantization map on the supergroup which is operator valued. The fact that we need the odd Fourier transform in this definition is an important difference with the non-graded case \cite{Bieliavsky:2010kg}. We then define a second quantization map, which associates bounded operators to functions in $L^1(M)$ (with $M$ a quotient of the supergroup).
In \ref{subsec-prel} we give some details of the functional spaces used in this theory; in particular we prove the existence of a resolution of the identity and we give the definition of a supertrace.

Next we want to enlarge the quantization map on $L^1(M)$ to be defined also on $\caB^1(M)$ (smooth functions all of whose derivatives are bounded), and to allow these functions to take their values in a Fr\'echet space $E$. That is why we introduce, in Appendix \ref{sec-funcspaces}, the symbol calculus $\caB^\mu_E(M)$ and, in \ref{subsec-integ}, the notion of an oscillating integral which allows us to give a meaning to the integral of non Lebesgue-integrable functions (using a phase and integrations by parts).
With this oscillating integral we extend in \ref{subsec-ext} the quantization map to $\caB^\mu_E(M)$, but the images are now in general unbounded operators. However, for $\mu=1$ and a Fr\'echet algebra $\algA$, this quantization map on $\caB^1_\algA(M)$ maps to bounded operators.
Finally, in \ref{subsec-product}, we define the deformed product on $\caB^1_\algA(M)$ which corresponds to the composition of operators via the quantization map, and we give some properties of this product.

\subsection{Unitary induced representation}
\label{subsec-induc}

In order to construct a unitary induced representation of the Heisenberg supergroup using Kirillov's orbits method \cite{Kirillov:1976}, we let $\zeta_0=a_0{}^bZ$ be a non-zero fixed element in $\kg^\ast$, where we assume that $a_0$ is real.
\begin{definition}
A polarization of $\zeta_0$ is a maximal free graded Lie subalgebra $\kb$ of $\kg$ such that $\delta_{\zeta_0}(\kb\times\kb)=\algzero$, where $\delta_{\zeta_0} : \kg\times \kg \to \superA$ is defined by
\begin{equation*}
\forall x,y\in\modE \ \forall a,b\in\superA
\quad:\quad
\delta_{\zeta_0}(x+aZ,y+bZ)=\langle\zeta_0,[x+aZ,y+bZ]\rangle=a_0\omega(x,y).
\end{equation*}
\end{definition}
In our case, any polarization $\kb$ of $\zeta_0$ is of the following form:
\begin{equation*}
\kb=\modW\oplus\superA Z,
\end{equation*}
where $\modW$ is a maximal isotropic subspace of $\modE$ (see Proposition \ref{prop-sympl-maxiso}). $\kb$ is an abelian $\superA$-Lie subalgebra of $\kg$ and an ideal. Since it admits a (non unique) supplement $\kq$ in $\kg$ (which we now fix once and for all), the short exact sequence of graded $\superA$-modules
\begin{gather*}
\xymatrix@1@C=25pt{{\algzero} \ar[r] & {\kb} \ar[r] & {\kg} \ar[r] & {\kg/\kb} \ar[r] & {\algzero}}
\end{gather*}
is split. Note that it is also a short exact sequence of $\superA$-Lie algebras. However, as such it is not (necessarily) split (unlike the non-graded case), because, even though $\kq$ is a graded subspace of $\kg$, it is in general not a lagrangian subspace nor a Lie subalgebra.

We now define $B=\exp(\kb_0)\simeq\kb_0$ and $Q=\exp(\kq_0)\simeq G/B\simeq \kq_0$ and we note that $Q$ is not a subgroup of $G$. Let $\chi:B\to \superA_\gC=\superA\otimes\gC$ be defined by 
\begin{equation*}
\forall w\in\modW_0 \ \forall a\in\superA_0
\quad:\quad
\chi(w+aZ)=e^{i\zeta_0(w+aZ)}=e^{ia_0a}.
\end{equation*}
It is a unitary character of $B$ ($\gB\text{Im}\chi\subset U(1)$), canonically associated to $\zeta_0$. Let us construct the representation of $G$ induced by $\chi$. 

From now on we will assume that the odd dimension of $B$ is zero, which means that $\omega_1$ is positive definite or negative definite on the body of the space generated by the odd generators of $\modE$ (see Proposition \ref{prop-sympl-maxiso}). Thus, an element of $B$ is of the form $(w,0,a)$ with $w\in\modW_0$ and $a\in\superA_0$ while an element of $Q$ is of the form $(x,\xi,0)$ with $x$ an even linear combination of even generators of $\kq$, and $\xi$ an odd linear combination of odd generators of $\kq$. Moreover, the map $Q\times B\to G$ given by group multiplication is a global diffeomorphism.

The group $G$ acts on itself by left translations, and we look at the left regular action $\lambda$ on the $B$-equivariant functions on $G$:
\begin{align*}
&C^\infty(G)^{B}=\{\hat\varphi:G\to\superA_\gC\ C^\infty,\quad \forall g\in G,\,\forall b\in B\ :\  \hat\varphi(gb)=\chi(b^{-1})\hat\varphi(g)\},
\\
&\forall g\in G,\,\forall \hat\varphi\in C^\infty(G)^{B}\quad:\quad \lambda_g\hat\varphi=L^\ast_{g^{-1}}\hat\varphi=\hat\varphi(g^{-1}\fois).
\end{align*}
Note that the space $C^\infty(G)^{B}$ itself is not invariant under the left regular action, because nilpotent constants (coming from elements of $G$) are not $C^\infty$ functions. However, tensoring with $\superA$ gives us the space $C^\infty(G)^{B}\otimes\superA$, which is invariant under the left regular action.
\begin{proposition}
The map $\lambda:G\times (C^\infty(G)^{B}\otimes\superA)\to C^\infty(G)^{B}\otimes\superA$ is an action. When we restrict the function $\lambda_g\hat\varphi$ to $Q\subset G$, we get the following formula: $\forall g\in G$, $\forall\hat\varphi\in C^\infty(G)^{B}$, $\forall q_0\in Q$,
\begin{equation*}
(\lambda_g\hat\varphi)(q_0)\nota=e^{ia_0(a+\omega_0(x-x_0,w)+\frac 12\omega_1(\xi,\xi_0))}\hat\varphi(x_0-x,\xi_0-\xi,0)\nota
\end{equation*}
where $g=q\fois b$, $q=(x,\xi,0)\in Q$, $b=(w,0,a)\in B$ and $q_0=(x_0,\xi_0,0)$. Remember that $\omega_0$ and $\omega_1$ are the diagonal parts of $\omega$ (see Proposition \ref{prop-sympl-basis}).
\end{proposition}

\begin{proof}
If $q=(x,\xi,0)$, $b=(w,0,a)$ and $q_0=(x_0,\xi_0,0)$,
\begin{equation*}
\lambda_{qb}\hat\varphi(q_0)\nota=\hat\varphi(b^{-1}q^{-1}q_0)=\hat\varphi(q^{-1}q_0C_{q_0^{-1}q}(b^{-1})),
\end{equation*}
where $C_{q_0^{-1}q}(b^{-1})=q_0^{-1}qb^{-1}q^{-1}q_0\in B$ because $B$ is a normal subgroup of $G$. However, unlike to the non-graded case, $q^{-1}q_0\notin Q$. In the above notation,
\begin{equation*}
q^{-1}q_0=(-x,-\xi,0)\fois(x_0,\xi_0,0)=\left(-x+x_0,-\xi+\xi_0,\frac12\omega_0(-x,x_0)+\frac12\omega_1(-\xi,\xi_0)\right)=(q_0-q)\fois\beta,
\end{equation*}
where $\beta=(0,0,\frac12\omega_1(-\xi,\xi_0))$. Indeed, $\omega_0(x,x_0)=0$ since the part involving only the even generators of $\kq$ is lagrangian (with respect to the even generated part of $\kg$), see Proposition \ref{prop-sympl-maxiso}. But there is no reason that $\omega_1(\xi,\xi_0)=0$, as $\kq$ is not lagrangian. By taking this into account, we obtain:
\begin{equation*}
\lambda_{qb}\hat\varphi(q_0)=\chi(C_{q_0^{-1}q}(b)\beta^{-1})\hat\varphi(q_0-q),
\end{equation*}
and $\chi(C_{q_0^{-1}q}(b)\beta^{-1})=e^{ia_0(a+\omega_0(x-x_0,w)+\frac 12\omega_1(\xi,\xi_0))}$.
\end{proof}

We already said that the group multiplication $Q\times B\to G$ is a global diffeomorphism. It follows that there is an isomorphism between $C^\infty(Q)$ and $C^\infty(G)^{B}$, which we will denote as $\tilde\varphi\mapsto\hat\varphi$, and which is given by $\hat\varphi(q\fois b)=\chi(b^{-1})\tilde\varphi(q)$ (with $q\in Q$ and $b\in B$). Transferring the action $\lambda$ to $C^\infty(Q)$ we obtain an action $\tilde U:G\times (C^\infty(Q)\otimes\superA)\to C^\infty(Q)\otimes\superA$, which has, using notation as above, the following explicit form:
\begin{equation}
(\tilde U(g)\tilde\varphi)(q_0)=e^{ia_0(a+\omega_0(x-x_0,w)+\frac 12\omega_1(\xi,\xi_0))}\tilde\varphi(q_0-q).
\label{eq-induc-expr}
\end{equation}

Now define $\caD(Q)\simeq\caD(\gB Q)\otimes \bigwedge \gR^n$, where $n$ is the odd dimension of $Q$, and $\caD(\gB Q)$ is the space of complex-valued functions on $\gB Q$ with compact support. Using the expression \eqref{eq-induc-expr}, it is not hard to see that $\caD(Q)\otimes\superA$ is an invariant subspace of $C^\infty(Q)\otimes\superA$ under the action $\tilde U$. In this way we obtain the following induced representation of $G$:
\begin{equation*}
\tilde U:G\to\caL_\superA(\caD(Q)\otimes\superA),
\end{equation*}
where $\caL_\superA$ is the space of $\superA$-linear maps. This induced representation is unitary for the supersymmetric scalar product.
\begin{proposition}
For $\tilde\varphi,\tilde\psi\in\caD(Q)$, one has: $\forall g\in G$,
\begin{equation*}
\langle\tilde U(g)\tilde\varphi,\tilde U(g)\tilde\psi\rangle=\langle\tilde\varphi,\tilde\psi\rangle.
\end{equation*}
\end{proposition}
\begin{proof}
\begin{multline*}
\langle\tilde U(g)\tilde\varphi,\tilde U(g)\tilde\psi\rangle=\int\dd x_0\dd\xi_0\, e^{-ia_0(a+\omega_0(x-x_0,w)+\frac 12\omega_1(\xi,\xi_0))}\overline{\tilde\varphi(x_0-x,\xi_0-\xi)}\\
e^{ia_0(a+\omega_0(x-x_0,w)+\frac 12\omega_1(\xi,\xi_0))}\tilde\psi(x_0-x,\xi_0-\xi)
\end{multline*}
The quantity $\omega_1(\xi,\xi_0)$ is even and thus commutes with all other terms. By a change of variables $x_0\mapsto x_0+x$, $\xi_0\mapsto \xi_0+\xi$ the result is obtained.
\end{proof}
We see that the induced representation $\tilde U$ of the Heisenberg supergroup constructed with the Kirillov's orbits method is naturally unitary with respect to the supersymmetric scalar product $\langle\fois,\fois\rangle$. This structure, and thus the notion of Hilbert superspace (see subsection \ref{subsec-superhilbert}) appears to be adapted for the harmonic analysis on the Heisenberg supergroup.

\subsection{Quantization}
\label{subsec-qu}

In the sequel of this paper we will assume that we have chosen the particular homogeneous basis of $\modE$ described in Proposition \ref{prop-sympl-basis}. This means in particular that we have $\omega_0(x,w)=x\fois w$ and $\omega_1(\xi_1,\xi_2)=2\xi_1\fois\xi_2$ for all $ (x_i,\xi_i)\in Q$ and all $ w\in\modW_0$. Adding a free parameter $\alpha$ to the (even and odd) Fourier transform, we obtain the following expressions for the Fourier Transform of the function constant $1$:
\begin{align}
&\int\dd w\ e^{2ia_0\alpha\omega_0(x,w)}=r_0\alpha^{-\frac m2}\delta(x)\label{eq-qu-fourier}\\
&\caF_\alpha 1(\xi_0)=\int\dd\xi\ e^{-\frac{ia_0}{2}\alpha\omega_1(\xi,\xi_0)}=r_1\alpha^n\xi_0^{\{1,\dots, n\}},\nonumber
\end{align}
where $r_0=\left(\frac{\pi}{a_0}\right)^{\frac m2}$ and $r_1=(ia_0)^n(-1)^{\frac{n(n+1)}{2}}$.

Let us now consider the $\superA$-linear map $\sigma:\kg\to\kg$ defined by 
\begin{equation*}
\forall x\in\modE\ \forall a\in\superA
\quad:\quad
\sigma(x+aZ)=-x+aZ
\ . 
\end{equation*}
It is an involutive automorphism of $\kg$ (and of $G\simeq\kg_0$). The pullback $\sigma^\ast:C^\infty(G)^{B}\to C^\infty(G)^{B}$ is defined by $(\sigma^\ast\hat\varphi)(g)=\hat\varphi(\sigma(g))$ (for $g\in G$ and $\hat\varphi\in C^\infty(G)^{B}$). Using the isomorphism $C^\infty(G)^{B}\cong C^\infty(Q)$ we obtain an involution, denoted by the same symbol, $\sigma^\ast:C^\infty(Q)\to C^\infty(Q)$. We then introduce the operator $\Sigma=\caF_\alpha\sigma^\ast$ as $\sigma$ followed by an odd Fourier transform given by
\begin{equation*}
\forall \tilde\varphi\in C^\infty(Q)\ \forall (x_0,\xi_0)\in Q
\quad:\quad
(\Sigma\tilde\varphi)(x_0,\xi_0)=\gamma\int\dd\xi_1\ e^{-\frac{ia_0\alpha}{2}\omega_1(\xi_1,\xi_0)}\tilde\varphi(-x_0,\xi_1),
\end{equation*}
where we interpret $\alpha$ as a free parameter and $\gamma$ as a fixed complex constant to be determined later. This operator $\Sigma$ satisfies the relations
\begin{align*}
&\Sigma^2=\gamma^2 r_1\alpha^n \text{id}
\\
\tilde\varphi,\tilde\psi\in C^\infty(Q)
\quad:\quad
&
\langle\Sigma\tilde\varphi,\Sigma\tilde\psi\rangle=|\gamma|^2r_1\alpha^n(-1)^{n(|\tilde\varphi|+1)}\langle\tilde\varphi,\tilde\psi\rangle,
\end{align*}
where $n$ is the odd dimension of $Q$ (or of $\modE$ since we assumed $B$ has no odd dimension). By choosing  $\gamma=\frac{(-1)^n}{r_0r_1(1+\alpha)^n}$ (this will be explained in Equation \eqref{eq-ext-condgamma}), we have in particular $\overline\gamma=(-1)^n\gamma$ (see Equation \eqref{eq-qu-fourier} for the definition of $r_0$,$r_1$). Then, we denote $r=\gamma^2 r_1\alpha^n$, and we obtain:
\begin{equation*}
\Sigma^2=r\gone,\qquad \Sigma^\dag=r\Sigma.
\end{equation*}

If we now define $K\subset G$ as the subgroup invariant under $\sigma$, $K=G^\sigma=\{g\in G,\,\sigma(g)=g\}$, then, $K=\superA_0 Z$, and $M:=G/K\simeq\modE_0$ is a supermanifold. With these preparations we can finally define the quantization map $\Omega:G\to\caL_\superA(\caD(Q)\otimes\superA)$ by 
\begin{equation*}
\forall g\in G
\quad:\quad
\Omega(g)=\tilde U(g)\Sigma\ \tilde U(g^{-1}).
\end{equation*}
\begin{proposition}
\label{prop-qu-omega}
The explicit formula for the quantization map is given by:
\begin{equation*}
\bigl(\Omega(q\fois b)\tilde\varphi\bigr)(q_0)= \gamma\int\dd\xi_1\ e^{ia_0(2\omega_0(x-x_0,w)+\frac12\omega_1(\xi,\xi_0)-\frac{\alpha}{2}\omega_1(\xi_1,\xi_0)-\frac12(\alpha+1)\omega_1(\xi,\xi_1))} \tilde\varphi(2x-x_0,\xi+\xi_1),
\end{equation*}
for any $\tilde\varphi\in\caD(Q)$, $q=(x,\xi,0)\in Q$, $q_0=(x_0,\xi_0,0)\in Q$ and $b=(w,0,a)\in B$. It has the following properties:
\begin{enumerate}
\item $\Omega$ is constant on the left classes of $K$.
\item For any $g\in G$ we have $\Omega(g)^\dag=r\,\Omega(g)$.
\end{enumerate}
\end{proposition}

\begin{proof}
\begin{enumerate}
\item This is immediate, as the explicit expression for $\Omega(q\fois b)$ does not depend on $a$.
\item This is a direct consequence of the analogous property for $\Sigma$ and the unitarity of $\tilde U$.
\end{enumerate}
\end{proof}

Let us now introduce the set $M$ of all left classes of $K$: $M=G/K$, which is a trivial supermanifold of dimension $m|n$ (see Definition \ref{def-superman-trivial}). By the point 1. of Proposition \ref{prop-qu-omega} the map $\Omega$ descends to this quotient, giving a map $\Omega:M\to \caL_\superA(\caD(Q)\otimes\superA)$, which we will still denote by the same symbol. If, for any $x\in M$, we introduce the symmetry $s_x:M\to M$ by $s_x(y)=x\sigma(x^{-1}\fois y)$ for all $y\in M$, then, $(M,s)$ becomes a symmetric space, namely:
\begin{equation*}
s_x^2=\text{id}_M\text{ and } s_x s_y s_x=s_{s_x(y)}.
\end{equation*}

\begin{proposition}
For any $\alpha\notin\{-1,0\}$, the map $\Omega:M\to\caL_\superA(\caD(Q)\otimes\superA)$ is a (twisted by $r$) representation of the symmetric space $(M,s)$, meaning that we have, for all $x,y\in M$,
\begin{enumerate}
\item $\Omega(x)^2=r\,\text{id}$.
\item $\Omega(x)\Omega(y)\Omega(x)=r\,\Omega(s_x(y))$.
\end{enumerate}
\end{proposition}

\begin{proof}
To show 1, it suffices to note that for $g\in x\cdot K$, we have 
\begin{equation*}
\Omega(g)^2=\tilde U(g)\Sigma\tilde U(g^{-1})\tilde U(g)\Sigma\tilde U(g^{-1})=r\,\text{id}
\end{equation*}
The identity 2 is a direct computation using the expression of Proposition \ref{prop-qu-omega}.
\end{proof}

Let us now ``extend'' the quantization map $\Omega$ to functions on $M$, by defining $\Omega:\caD(M)\to\caL_\superA(\caD(Q)\otimes\superA)$ by
\begin{equation*}
\forall f\in\caD(M)
\quad:\quad
\Omega(f)=\int_M\dd z\, f(z)\Omega(z).
\end{equation*}
For any $\tilde\varphi\in\caD(Q)$ and $q_0\in Q$, the explicit expression is given by 
\begin{multline}
\bigl(\Omega(f)\tilde\varphi\bigr)(q_0)=\gamma\int_M\dd x\dd\xi\dd w f(x,\xi,w)
\\
\int\dd\xi_1\ e^{ia_0(2\omega_0(x-x_0,w)+\frac12\omega_1(\xi,\xi_0)-\frac{\alpha}{2}\omega_1(\xi_1,\xi_0)-\frac12(\alpha+1)\omega_1(\xi,\xi_1))}
\tilde\varphi(2x-x_0,\xi+\xi_1).\label{eq-qu-omf}
\end{multline}
Since $a_0\in\gR^\ast$ and since $f$ has a compact support with respect to the variable $x\in\gB Q$, it follows that $\Omega(f)\tilde\varphi\in\caD(Q)$. Contrary to the case of the map $\Omega$ on $M$, where $\Omega(x)\tilde\varphi$ does not lie in $\caD(Q)$ but only in $\caD(Q)\otimes\superA$, here there is no need to take the tensor product by $\superA$. This means in particular that we can consider the quantization map as a map $\Omega:\caD(M)\to\caL(\caD(Q))$. Note that the degree of $\Omega(f)$ (with respect to the $\gZ_2$ grading) is the same as the degree of $f$.

\begin{lemma}
\label{lem-qu-cont}
For all $f\in\caD(M)$ there exists a constant $C_f\in\gR^\ast_+$ such that \begin{equation*}
\forall\tilde\varphi\in\caD(Q)
\quad:\quad
\left(\Omega(f)\tilde\varphi,\Omega(f)\tilde\varphi\right)\leq C_f\left(\tilde\varphi,\tilde\varphi\right).
\end{equation*}
\end{lemma}

\begin{proof}
Using formula \eqref{eq-qu-omf} we compute:
\begin{multline*}
\left(\Omega(f)\tilde\varphi,\Omega(f)\tilde\varphi\right) =
\gamma^2\int\dd x_0\dd\xi_0\dd x\dd\xi\dd w \overline{f(x,\xi,w)}\dd\xi_1
\\
e^{-ia_0(2\omega_0(x-x_0,w)+\frac12\omega_1(\xi,\xi_0)-\frac{\alpha}{2}\omega_1(\xi_1,\xi_0)-\frac12(\alpha+1)\omega_1(\xi,\xi_1))}\overline{\tilde\varphi(2x-x_0,\xi+\xi_1)}
\\
\ast_{\xi_0}\int\dd x'\dd\xi'\dd w' f(x',\xi',w')\dd\xi_2 
\\
e^{ia_0(2\omega_0(x'-x_0,w')+\frac12\omega_1(\xi',\xi_0)-\frac{\alpha}{2}\omega_1(\xi_2,\xi_0)-\frac12(\alpha+1)\omega_1(\xi',\xi_2))} \tilde\varphi(2x'-x_0,\xi'+\xi_2),
\end{multline*}
where $\ast_{\xi_0}$ is the Hodge operation with respect to the variable $\xi_0$. Expanding this formula and integrating over the odd variables $\xi$, $\xi'$, $\xi_1$, $\xi_2$ and $\xi_0$ gives us the formula
\begin{multline*}
\left(\Omega(f)\tilde\varphi,\Omega(f)\tilde\varphi\right) =\sum_{I_1,I_2,J_1,J_2}\alpha(I_1,I_2,J_1,J_2) \int\dd x_0\dd x\dd w\dd x'\dd w' \overline{f_{I_1}(x,w)}f_{I_2}(x',w')\\
e^{2ia_0(\omega_0(x'-x_0,w')-\omega_0(x-x_0,w))}\overline{\tilde\varphi_{J_1}(2x-x_0)}\tilde\varphi_{J_2}(2x'-x_0).
\end{multline*}
where $I_1,I_2,J_1,J_2$ are subsets of $\{1,\dots,n\}$ with some constraints between them that we do not write explicitly, and $\alpha(I_1,I_2,J_1,J_2)$ is a real number independent of $f$ and $\tilde\varphi$. Let us now denote by $|\alpha|$ the supremum (maximum) of all $|\alpha(I_1,I_2,J_1,J_2)|$. By using the triangular and the Cauchy-Schwartz inequalities, we obtain
\begin{align*}
\left(\Omega(f)\tilde\varphi,\Omega(f)\tilde\varphi\right) &\leq\sum_{I_1,I_2,J_1,J_2}|\alpha| \int\dd x\dd w\dd x'\dd w' |f_{I_1}(x,w)f_{I_2}(x',w')| \norm\tilde\varphi_{J_1}\norm_2\norm\tilde\varphi_{J_2}\norm_2\\
&\leq\sum_{I_1,I_2,J_1,J_2}|\alpha| \norm f_{I_1}\norm_1 \norm f_{I_2}\norm_1 \norm\tilde\varphi_{J_1}\norm_2\norm\tilde\varphi_{J_2}\norm_2
\end{align*}
But $\left(\tilde\varphi,\tilde\varphi\right)=\sum_I\norm\tilde\varphi_I\norm_2^2\geq \sup_I \norm\tilde\varphi_I\norm_2^2$, hence
\begin{equation*}
\left(\Omega(f)\tilde\varphi,\Omega(f)\tilde\varphi\right)\leq C_f\left(\tilde\varphi,\tilde\varphi\right),
\end{equation*}
with $C_f=\sum_{I_1,I_2,J_1,J_2}|\alpha| \norm f_{I_1}\norm_1 \norm f_{I_2}\norm_1$.
\end{proof}

\begin{corollary}
The quantization map $\Omega:\caD(M)\to\caL(\caD(Q))$ can be continuously extended in a unique way to a map
\begin{equation*}
\Omega:L^1(M)\to\caB(L^2(Q)),
\end{equation*}
where we recall that $L^2(Q)$ is a Hilbert superspace (see subsection \ref{subsec-superhilbert}).
\end{corollary}
\begin{proof}
We have shown in Lemma \ref{lem-qu-cont} that for any $f\in\caD(M)$, $\Omega(f)$ is a continuous operator on $\caD(Q)$. Since $\caD(Q)$ is dense in $L^2(Q)$, we can extend $\Omega(f)$ in a unique way to a bounded operator on $L^2(Q)$. Moreover, if we denote by $\norm f\norm_1$ the norm of $f\in L^1(M)=L^1(\gB M)\otimes\bigwedge \gR^n$:
\begin{equation*}
\norm f\norm_1=\sum_I\int_{\gB M}\dd x|f_I(x)|=\sum_I\norm f_I\norm_1
\end{equation*}
then it is immediate from the proof of Lemma \ref{lem-qu-cont} that the constant $C_f$ verifies $C_f\leq C^2\norm f\norm_1^2$, with $C^2=\sum_{J_1,J_2}|\alpha|$. It follows that we have, for any $\tilde\varphi\in L^2(Q)$,
\begin{equation*}
\sqrt{\left(\Omega(f)\tilde\varphi,\Omega(f)\tilde\varphi\right)}\leq C\sqrt{\left(\tilde\varphi,\tilde\varphi\right)}\norm f\norm_1,
\end{equation*}
which means that $\Omega$ is a continuous map on $\caD(M)$ for the topology of $L^1(M)$ and for the operator topology of $\caB(L^2(Q))$. Since $\caD(M)$ is dense in $L^1(M)$, we can extend $\Omega$ in a unique way to a continuous map: $\Omega:L^1(M)\to\caB(L^2(Q))$.
\end{proof}

\subsection{Preliminaries concerning functional spaces}
\label{subsec-prel}

We start this subsection with a result concerning Schwartz functions on $Q$, where the space $\caS(Q)$ of these functions is defined (as usual) by $\caS(Q) = \caS(\gB Q) \otimes \bigwedge \gR^n$ with $n$ still the odd dimension of $Q$ and $\caS(\gB Q)$ the standard space of Schwartz functions (recall that $\gB Q$ is a (real) vector space).

In the rest of this section, we will use the notation for the variables:
\begin{equation}
x\in Q_o,\quad \xi\in\gR^{0|n},\quad q=(x,\xi)\in Q,\quad w\in\modW_0,\quad z=(x,\xi,w)\in M,\quad y=(\gB x,\gB w)\in \gB M,\label{eq-prel-var}
\end{equation}
where we recall that $M\simeq M_o\times\gR^{0|n}\simeq Q\times\modW_0$, $Q\simeq Q_o\times\gR^{0|n}$ and that the odd dimension of $\modW_0$ is zero.

\begin{lemma}
\label{lem-prel-schw}
For any $\tilde\varphi\in\caS(Q)$ and any $(x,\xi,w)\in M$, we define the function $\tilde\varphi_{(x,\xi,w)}=\tilde U((x,\xi,0)\fois(w,0,0))\tilde\varphi$. Then, for any $\tilde\psi\in L^2(Q)$, the map $(x,\xi,w)\mapsto \langle\tilde\varphi_{(x,\xi,w)},\tilde\psi\rangle$ belongs to $\caS(M)$.
\end{lemma}

\begin{proof}
We start with the computation
\begin{align*}
\langle\tilde\varphi_{(x,\xi,w)},\tilde\psi\rangle &=\int\dd x_0\dd\xi_0 e^{-ia_0(\omega_0(x-x_0,w)+\frac 12\omega_1(\xi,\xi_0))} \overline{\tilde\varphi(x_0-x,\xi_0-\xi)}\tilde\psi(x_0,\xi_0)\\
&=\sum_{I,J,K}\alpha(I,J,K)\int\dd x_0 e^{-ia_0\omega_0(x-x_0,w)}\overline{\tilde\varphi_I(x_0-x)}\tilde\psi_J(x_0)\xi^K,
\end{align*}
where, using notation as in the proof of Lemma \ref{lem-qu-cont}, we integrated over $\xi_0$ (and just as in the proof of Lemma  \ref{lem-qu-cont}, there are some constraints among $I$, $J$ and $K$).

Let us now consider the quantity $\caI$ defined by
\begin{equation*}
\caI=\int_{\gB M}\dd x\dd w \left(D_{(x,w)}^\beta \int\dd x_0 e^{-ia_0\omega_0(x-x_0,w)}\overline{\tilde\varphi_I(x_0-x)}\tilde\psi_J(x_0)\right)P(x,w),
\end{equation*}
where $P$ is an arbitrary polynomial function and $\beta$ a multi-index. If we can show that $\vert \caI\vert<\infty$, then we will have proved the lemma, as it shows that the map $(x,\xi,w)\mapsto \langle\tilde\varphi_{(x,\xi,w)},\tilde\psi\rangle$ is Schwartz. To do so, we start with the obvious observation that we have
\begin{equation*}
\caI=\int\dd x\dd w\dd x_0 \left(D_{(x,w)}^\beta e^{ia_0\omega_0(x_0,w)}\overline{\tilde\varphi_I(x_0)}\tilde\psi_J(x_0+x)\right)P(x,w).
\end{equation*}
We then introduce the operator $O_{x_0}=\frac{1}{1+w^2}(1-\frac{1}{a_0^2}\Delta_{x_0})$, which has the property
$O_{x_0}(e^{ia_0\omega_0(x_0,w)})=e^{ia_0\omega_0(x_0,w)}
$, simply because $\partial_{x_0} e^{ia_0\omega_0(x_0,w)}=ia_0 we^{ia_0\omega_0(x_0,w)}$. Inserting the $k$-th power of this operator on the exponential and integrating by parts gives us:
\begin{equation*}
\caI=\int\dd x\dd w\dd x_0 \left(D_{(x,w)}^\beta e^{ia_0\omega_0(x_0,w)}O^k_{x_0}(\overline{\tilde\varphi_I(x_0)}\tilde\psi_J(x_0+x))\right)P(x,w).
\end{equation*}
We then note that we have the equality
\begin{multline*}
D_{(x,w)}^\beta \left(e^{ia_0\omega_0(x_0,w)}O^k_{x_0}(\overline{\tilde\varphi_I(x_0)}\tilde\psi_J(x_0+x))\right)=\\
\frac{1}{(1+w^2)^k}\sum_{\gamma_i}b_{\gamma_i}(x_0,w) e^{ia_0\omega_0(x_0,w)} \partial_{x_0}^{\gamma_1}\overline{\tilde\varphi_I(x_0)}\partial_{x_0}^{\gamma_2} \partial_{x}^{\gamma_3}\tilde\psi_J(x_0+x),
\end{multline*}
for multi-indices $\gamma_i$ (a finite sum) and functions $b_{\gamma_i}(x_0,w)$ which are bounded by a polynomial in $x_0$ and bounded in $w$. We thus can make the estimation
\begin{equation*}
|\caI|\leq\sum_{\gamma_i}\int\dd x\dd w\dd x_0\frac{|b_{\gamma_i}(x_0,w)P(x,w)|}{(1+w^2)^k}|\partial_{x_0}^{\gamma_1}\overline{\tilde\varphi_I(x_0)}\partial_{x_0}^{\gamma_2} \partial_{x}^{\gamma_3}\tilde\psi_J(x_0+x)|
\ .
\end{equation*}
This will be finite whenever $k$ is greater than $\dim(\gB M)+1+d$ ($d$ the degree of $P$ in $w$) because with such a $k$ the integrand will be integrable in $w$ and because the integrand is Schwartz in $x$ and $x_0$.
\end{proof}

Using the same notation and the same arguments, one can show that the function $(x,\xi,w)\mapsto(\tilde\varphi_{(x,\xi,w)},\tilde\psi)$ also belongs to $\caS(M)$. See also \cite{Bieliavsky:2010kg} for the analog in the non-graded case. Let us now show that there exists a resolution of the identity in this context.

\begin{theorem}[Resolution of the identity]
\label{thm-prel-resol}
There exists a constant $C\in\gC^\ast$ such that 
\begin{equation}
\forall \tilde\varphi\in\caS(\gB Q)\ ,\ \forall \tilde\psi\in\caS(Q)\ ,\ \forall q_1\in Q
\quad:\quad
\int_M\dd z\ \langle\tilde\varphi_{z},\tilde\psi\rangle \tilde\varphi_{z}(q_1)=C\norm\tilde\varphi\norm_2^2\tilde\psi(q_1).\label{eq-prel-resol}
\end{equation}
\end{theorem}

\begin{proof}
By Lemma \ref{lem-prel-schw} the integrand is Schwartz on $M$, so the integral is well defined. Denoting $z=(x,\xi,w)\in M$ and $q_1=(x_1,\xi_1)\in Q$, we then compute: 
\begin{multline*}
\int_M\dd x\dd\xi\dd w\ \langle\tilde\varphi_{(x,\xi,w)},\tilde\psi\rangle \tilde\varphi_{(x,\xi,w)}(x_1,\xi_1)
\\
=\int\dd x\dd\xi\dd w\dd x_0\dd \xi_0 e^{-ia_0 (\omega_0(x_1-x_0,w)+\frac 12\omega_1(\xi,\xi_0-\xi_1))}
\overline{\tilde\varphi(x_0-x)}\tilde\psi(x_0,\xi_0)\tilde\varphi(x_1-x).
\end{multline*}
Using Equations \eqref{eq-qu-fourier} we then obtain
\begin{align*}
\int_M\dd x\dd\xi\dd w\ 
&
\langle\tilde\varphi_{(x,\xi,w)},\tilde\psi\rangle \tilde\varphi_{(x,\xi,w)}(x_1,\xi_1)
\\
&
=r_1(-1)^n\int\dd x\dd w\dd x_0\dd \xi_0 (\xi_0-\xi_1)^{\{1,\dots,n\}}e^{-ia_0\omega_0(x_1-x_0,w)}\overline{\tilde\varphi(x_0-x)}
\\
&
\qquad\kern8cm\tilde\psi(x_0,\xi_0)\tilde\varphi(x_1-x)\\
&=r_1(-1)^n\int\dd x\dd w\dd x_0 e^{-ia_0\omega_0(x_1-x_0,w)}\overline{\tilde\varphi(x_0-x)}\tilde\psi(x_0,\xi_1)\tilde\varphi(x_1-x).\\
\end{align*}
Since $\int \dd we^{-ia_0\omega_0(x_1-x_0,w)}=r_0 2^{\frac m2}\delta(x_1-x_0)$, we get the announced result:
\begin{equation*}
\int_M\dd x\dd\xi\dd w\ \langle\tilde\varphi_{(x,\xi,w)},\tilde\psi\rangle \tilde\varphi_{(x,\xi,w)}(x_1,\xi_1)=C\int\dd x\overline{\tilde\varphi(x)}\tilde\varphi(x)\tilde\psi(x_1,\xi_1),
\end{equation*}
with $C=r_0r_1 2^{\frac m2}(-1)^n$.
\end{proof}

For any $\tilde\varphi\in\caS(\gB Q)$ the map $\caS(Q)\to L^2(Q)$ defined by
\begin{equation*}
\tilde\psi\mapsto \int_M\dd z\, \tilde\varphi_z \langle \tilde\varphi_z,\tilde\psi\rangle= C\norm\tilde\varphi\norm_2^2 \tilde\psi,
\end{equation*}
is a dilation, hence continuous. It thus can be extended to $L^2(Q)$.

Note that the resolution of the identity only uses functions $\tilde\varphi$ defined on the body of $Q$. Note also that for $\tilde\varphi\in\caS(\gB Q)$ and $\forall z\in M$, the function $\tilde\varphi_z$ is even but that it may contain nilpotent elements, and thus in general belongs to $\caS(Q)\otimes\superA$, not to $\caS(Q)$. Note however that we have 
\begin{equation}
\forall T\in\caB(L^2(Q))\ ,\ \forall\tilde\varphi\in\caS(\gB Q)\ ,\ \forall (x,w)\in\gB M)
\quad:\quad
\int\dd\xi \langle\tilde\varphi_{(x,\xi,w)},T\tilde\varphi_{(x,\xi,w)}\rangle\in\gC.\label{eq-prel-deftr}
\end{equation}

\begin{definition}[Supertrace]
\label{def-prel-supertrace}
Let $\tilde\varphi\in\caS(\gB Q)$ have a non-vanishing norm. We define the trace-class operators as those (unbounded) operators of $L^2(Q)$, $T\in\caO(L^2(Q))$, satisfying
\begin{itemize}
\item $\caS(Q)\subset \Dom(T)$. Since $\caS(Q)$ is dense in $L^2(Q)$, the adjoint $T^\ast$ of $T$ with respect to the positive definite scalar product exists.
\item $\caS(Q)\subset \Dom(T^\ast)$.
\item $\int\dd x\dd w\ |\int\dd\xi\,\langle\tilde\varphi_{(x,\xi,w)},T\tilde\varphi_{(x,\xi,w)}\rangle|$ is finite. Note that it makes sense due to Equation \eqref{eq-prel-deftr}.
\end{itemize}
For $T$ a trace-class operator, we then define its supertrace by
\begin{equation}
\tr(T)=\frac{1}{\norm\tilde\varphi\norm_2^2}\int_M\dd z\ \langle\tilde\varphi_z,T\tilde\varphi_z\rangle.\label{eq-prel-supertr}
\end{equation}
\end{definition}

\begin{proposition}
The supertrace has the following properties:
\begin{enumerate}
\item It is independent of $\tilde\varphi\in\caS(\gB Q)$ used (twice) in its definition.
\item If $T_1$ and $T_2$ are trace-class operators such that $T_1T_2$ and $ T_2T_1$ are also trace-class, then we have
\begin{equation*}
\tr(T_1T_2)=(-1)^{|T_1||T_2|}\tr(T_2T_1).
\end{equation*}
\end{enumerate}
\end{proposition}

\begin{proof}
\begin{enumerate}
\item For $T\in\caB^1(L^2(Q))$ and $\tilde\varphi,\tilde\psi\in\caS(\gB Q)$ we compute, 
\begin{align*}
\tr(T)&=\frac{1}{\norm\tilde\varphi\norm^2_2}\int\dd z\dd q_0\ \overline{\tilde\varphi_z(q_0)}T\tilde\varphi_z(q_0)\\
&=\frac{1}{C\norm\tilde\varphi\norm^2_2\norm\tilde\psi\norm^2_2}\int\dd z\dd q_0\dd z_1\dd q_2\  \overline{\tilde\varphi_z(q_0)}\overline{\tilde\psi_{z_1}(q_2)}T\tilde\varphi_z(q_2)\tilde\psi_{z_1}(q_0)\\
&=\frac{1}{C\norm\tilde\varphi\norm^2_2\norm\tilde\psi\norm^2_2}\int\dd z_1\dd q_2\dd z\dd q_0\  \overline{\tilde\varphi_z(q_0)}\tilde\psi_{z_1}(q_0)\tilde\varphi_z(q_2) T^\dag\overline{\tilde\psi_{z_1}}(q_2)\\
&=\frac{1}{\norm\tilde\psi\norm^2_2}\int\dd z_1\dd q_2\ \tilde\psi_{z_1}(q_2) T^\dag\overline{\tilde\psi_{z_1}}(q_2)\\
&=\frac{1}{\norm\tilde\psi\norm^2_2}\int\dd z_1\dd q_2\ \overline{\tilde\psi_{z_1}(q_2)}T\tilde\psi_{z_1}(q_2),
\end{align*}
where we used Theorem \ref{thm-prel-resol} successively with $\tilde\psi$ and $\tilde\varphi$. We also used that there exists a superadjoint of $T$: the adjoint $T^\ast$ exists and has a domain containing $\caS(Q)$, so as the superadjoint $T^\dag$ (see Proposition \ref{prop-superhilbert-adjoint}). 

\item To prove the second property, one uses exactly the same method as for the first: two resolutions of the identity and the use of both superadjoints $T_1^\dag$ and $T_2^\dag$.
\end{enumerate}
\end{proof}

\begin{lemma}
Let $T\in\caB^1(L^2(Q))$ be an operator admitting a kernel, i.e. there exists a measurable function $K:Q\times Q\to\gC$ such that $\forall \tilde\psi\in L^2(Q)$, $\forall q_0\in Q$, we have $T\tilde\psi(q_0)=\int_Q\dd q\ K(q,q_0)\tilde\psi(q)$. Then its supertrace is given by
\begin{equation*}
\tr(T)=C\int_Q\dd q\ K(q,q),
\end{equation*}
where $C$ is the constant defined in Theorem \ref{thm-prel-resol}.
\end{lemma}

\subsection{Integration}
\label{subsec-integ}

We now consider a complex Fr\'echet space $(E,|\fois|_j)$ with its family of seminorms indexed by $j\in\gN$. We refer the reader to Appendix \ref{sec-funcspaces} for the definitions of a weight $\mu$, of the functional spaces $\caD_E(M)$, $L^1_E(M)$ and $\caB^\mu_E(M)$, and of the $E$-valued integral. With these notions we observe that, since $L^1_E(M)=L^1(M)\widehat{\otimes}E$, the quantization map can be extended by the projective tensor product \cite{Grothendieck:1966}:
\begin{equation*}
\Omega:L^1_E(M)\to\caB(L^2(Q))\widehat{\otimes} E.
\end{equation*}

We then attack the notion of the oscillating integral, which allows to integrate functions which are not Lebesgue-integrable (for example functions in $\caB^\mu_E(M)$ when $\mu\notin L^1(\gB M)$), by using the phase and integration by parts. See also \cite{Bieliavsky:2010kg} in the non-graded case.

\begin{theorem}
\label{thm-reg-oscil}
The functional $\caD_E(M)\to E$, called an oscillating integral and defined by
\begin{equation}
f\mapsto \int_M\dd x\dd\xi\dd w\ e^{2ia_0\omega_0(x,w)}f(x,\xi,w),
\label{eq-reg-oscil}
\end{equation}
can be extended to a continuous linear functional $\caB^\mu_E(M)\to E$, for any weight $\mu$ bounded by a polynomial function.

Moreover, there is only one such extension which is continuous for the topology induced on $\caB^{\mu}_E(M)$ by the topology of $\caB^{\mu'}_E(M)$ for any weight $\mu'>\mu$ bounded by a polynomial function such that $\lim_{|y|\to\infty}\frac{\mu(y)}{\mu'(y)}$ $=0$.
\end{theorem}

\begin{proof}
For this result, we may assume, without loss of generality, that we have $2a_0=1$, which simplifies the notation.
\begin{itemize}
\item Let $f\in\caD_E(M)$. We will use the basis of Proposition \ref{prop-sympl-basis}, so that $\omega_0$ is given by the expression $\omega_0(x,w)=xw$ (usual scalar product). If we define the operator $O$ by 
\begin{equation*}
(O\fois f)(x,\xi,w)=(1-\Delta_{(x,w)}) \Bigl(\frac{1}{1+x^2+w^2}f(x,\xi,w)\Bigr)
\ ,
\end{equation*}
then, as in the proof of Lemma \ref{lem-prel-schw}, an integration by parts gives
\begin{equation}
\int_M\dd x\dd\xi\dd w\ e^{i\omega_0(x,w)}f(x,\xi,w)= \int_M\dd x\dd\xi\dd w\ e^{i\omega_0(x,w)}(O^k\fois f)(x,\xi,w),
\label{eq-reg-oscil1}
\end{equation}
for any $k\in\gN$. Moreover, it is not hard to show that there exists functions $b^\alpha\in\caB^1_\gC(\gB M)$ such that
\begin{equation*}
O^k=\frac{1}{(1+x^2+w^2)^k}\sum_\alpha b^\alpha(x,w) D^\alpha.
\end{equation*}

\item Since $\mu$ is bounded by a polynomial function, there exists $N\in\gN$ and a constant $C>0$ such that $\forall (x,w)\in \gB M$, $\mu(x,w)\leq C(1+x^2+w^2)^N$. Hence for any $f\in\caB^\mu_E(M)$ and $I=\{1,\dots,n\}$, we have
\begin{equation*}
|O^k\fois f_I(x,w)|_j\leq \frac{1}{(1+x^2+w^2)^k}\sum_\alpha|b^\alpha(x,w)| C_{j,\alpha,I}\mu(x,w)\leq C'(1+x^2+w^2)^{N-k},
\end{equation*}
with $C'$ a positive constant and $C_{j,\alpha,I}$ as in Definition \ref{def-mu-bounded-functions}. It follows that $O^k\fois f\in L^1_E(M)$, provided $k-N\geq \dim(\gB M)+1$. As a consequence, the right hand-side of \eqref{eq-reg-oscil} does exist for $f\in\caB^\mu_E(M)$.

\item Since $|D^\alpha f|_{j,\gamma}=\sup_{y\in\gB M}\{\frac{1}{\mu(y)}\sum_I|D^\gamma D^\alpha f_I(y)|_j\}=|f|_{j,\gamma+\alpha}$, it follows that the map $D^\alpha:\caB^\mu_E(M)\to\caB^\mu_E(M)$ is continuous.

\item Let $f\in\caB^1_\gC(M)$ and $F\in\caB^\mu_E(M)$. By using the Leibniz rule and the equality $(f\fois F)_I=\sum_{J\subset I}\eps(J,\complement^I_J)f_JF_{\complement^I_J}$, we can compute:
\begin{align*}
|f\fois F|_{j,\alpha}&=\sup_{y\in\gB M}\{\frac{1}{\mu(y)}\sum_I|D^\alpha(\sum_{J\subset I}\eps(J,\complement^I_J)f_J(y)F_{\complement^I_J}(y))|_j\}\\
&=\sup_{y\in\gB M}\{\frac{1}{\mu(y)}\sum_I|\sum_{(D^\alpha),J\subset I}\eps(J,\complement^I_J)D^\alpha_{(1)}f_J(y)D^\alpha_{(2)}F_{\complement^I_J}(y)|_j\}\\
&\leq\sup_{y\in\gB M}\{\frac{1}{\mu(y)}\sum_{(D^\alpha),I,J\subset I} |D^\alpha_{(1)}f_J(y)||D^\alpha_{(2)}F_{\complement^I_J}(y)|_j\}\\
&\leq\sum_{(D^\alpha)}\sup_{y\in\gB M}\{\frac{1}{\mu(y)}\sum_{J,K} |D^\alpha_{(1)}f_J(y)||D^\alpha_{(2)}F_{K}(y)|_j\}\\
&\leq\sum_{(D^\alpha)}|f|_{D^\alpha_{(1)}}\sup_{y\in\gB M}\{\frac{1}{\mu(y)}\sum_{K} |D^\alpha_{(2)}F_{K}(y)|_j\}\leq \sum_{(D^\alpha)}|f|_{D^\alpha_{(1)}}|F|_{j,D^\alpha_{(2)}},
\end{align*}
where we have used the Hopf algebra approach\slash notation of the Leibniz rule and the fact that $\sum_J|D^\alpha_{(1)}f_J(y)|\leq |f|_{D^\alpha_{(1)}}$. This tells us that the product of functions map $\caB^1_\gC(M)\times\caB^\mu_E(M)\to\caB^\mu_E(M)$ is continuous.

\item In exactly the same way one can show that the map $\caB^{\mu'}_\gC(M)\times\caB^\mu_E(M)\to\caB^{\mu\mu'}_E(M)$ taking the product of two functions is continuous. Moreover, from the above expression of $O^k$, one deduces that the map $O^k:\caB^\mu_E(M)\to\caB^{\frac{\mu(x,w)}{(1+x^2+w^2)^k}}_E(M)$ is continuous as it consists of products and derivations.

\item Let $\mu'\in L^1(\gB M)$ be a weight and $f\in\caB^{\mu'}_E(M)\subset L^1_E(M)$. Then:
\begin{equation*}
|\int_M\dd y\dd\xi f(y,\xi)|_j\leq \int_M\dd y\dd\xi |f(y,\xi)|_j\leq \int_{\gB M}\dd y |f|_{j,1}\mu'(y)=\norm\mu'\norm_1|f|_{j,1}.
\end{equation*}
It follows that the integration map $\caB^{\mu'}_E(M)\to E$ is continuous. By choosing $\mu'(x,w)=\frac{\mu(x,w)}{(1+x^2+w^2)^k}\in L^1_E(M)$ and assuming that $k-N\geq \dim(\gB M)+1$, we thus have shown that the map $\caB^\mu_E(M)\to E$ given by
\begin{equation*}
f\mapsto \int_M\dd x\dd\xi\dd w\ e^{i\omega_0(x,w)}(O^k\fois f)(x,\xi,w)
\end{equation*}
is continuous. Moreover, it is an extension of \eqref{eq-reg-oscil} due to Equation \eqref{eq-reg-oscil1}.

\item We now will prove that $\caB^\mu_E(M)\subset\text{adh}_{\caB^{\mu'}_E(M)}(\caD_E(M))$ when $\mu'>\mu$ is another weight bounded by a polynomial function and satisfying $\lim_{|y|\to\infty}\frac{\mu(y)}{\mu'(y)}=0$. For any $ f\in\caB^{\mu'}_E(M)$ and for any semi-norm $|\cdot |_{j,D^\alpha}$ on $\caB^\mu_E(M)$, we have
\begin{align*}
|\frac{\mu}{\mu'}f|_{j,D^\alpha} &=\sup_{y\in\gB M}\{\frac{1}{\mu(y)}\sum_I|D^\alpha(\frac{\mu(y)}{\mu'(y)}f_I(y))|_j\}\\
&\leq \sup_{y\in\gB M}\{\frac{1}{\mu(y)}\sum_{I,(D^\alpha)}|D^\alpha_{(1)}\mu(y)||D^\alpha_{(2)}\frac{1}{\mu'(y)}||D^\alpha_{(3)}f_I(y)|_j\}\\
&\leq \sum_{(D^\alpha)}\sup_{y\in\gB M}\{\frac{1}{\mu(y)}\sum_{I}C_{D^\alpha_{(1)}}C_{D^\alpha_{(2)}}\frac{\mu(y)}{\mu'(y)}|D^\alpha_{(3)}f_I(y)|_j\}\\
&\leq \sum_{(D^\alpha)}C_{D^\alpha_{(1)}}C_{D^\alpha_{(2)}}|f|_{j,D^\alpha_{(3)}},
\end{align*}
where $|f|_{j,D^\alpha_{(3)}}$ is a semi-norm of $\caB^{\mu'}_E(M)$. Moreover, $|D^\alpha_{(1)}\mu|<C_{D^\alpha_{(1)}}\mu$ because $\mu$ is a weight, and $D^\alpha_{(2)}\frac{1}{\mu'}<C_{D^\alpha_{(2)}}\frac{1}{\mu'}$ because $D^\alpha_{(2)}\frac{1}{\mu'}$ is a term of the form $\frac{1}{\mu'}$ times a sum of products of bounded terms $\frac{D^\beta\mu'}{\mu'}$. We thus have shown the equivalence $f\in\caB^{\mu'}_E(M)\Leftrightarrow \frac{\mu}{\mu'}f\in\caB^\mu_E(M)$.

Now let $\chi_p:\gB M\to[0,1]$ be a smooth function with compact support and taking the value 1 on the ball of radius $p\in\gN^\ast$ in $\gB M$.
Then the limit $\lim_{p\to\infty}(\frac{\mu}{\mu'}\chi_p)=\frac{\mu}{\mu'}$ is uniform since $\lim\frac{\mu}{\mu'}=0$. It follows that the sequence $(\frac{\mu}{\mu'}\chi_p f)$ in $\caD_E(M)$ converges (for $p\to\infty$) to $\frac{\mu}{\mu'}f$.
\end{itemize}
\end{proof}

\subsection{Extension of the quantization map}
\label{subsec-ext}

\begin{lemma}
\label{lem-reg-trans}
If $\mu$ is a weight bounded by a polynomial, then there exists a weight $\tilde\mu>\mu$ also bounded by a polynomial such that
\begin{equation*}
\forall f\in\caB^\mu_E(M),\quad\forall y\in\gB M
\quad:\quad 
L_y^\ast f\in\caB^{\tilde\mu}_E(M),
\end{equation*}
where $L_y$ denotes translation over $y$.
\end{lemma}

\begin{proof}
For any $y,y_0\in\gB M$ we have
\begin{equation*}
|D^\alpha L_y^\ast f_I(y_0)|_j=|L_y^\ast D^\alpha f_I(y_0)|_j=|D^\alpha f_I(y+y_0)|_j\leq C_{j,\alpha,I}\mu(y+y_0).
\end{equation*}
Moreover, there exists a weight $\tilde\mu>\mu$ bounded by a polynomial such that for all $ y,y_0\in\gB M$ we have $\mu(y+y_0)\leq \tilde\mu(y)\tilde\mu(y_0)$. Note that $\tilde\mu$ is independent of $y$ and $y_0$. Then, introducing $C'=C_{j,\alpha,I}\tilde\mu(y)$, we obtain $|D^\alpha L_y^\ast f_I(y_0)|_j\leq C'\tilde\mu(y_0)$. We thus have shown $L_y^\ast f\in\caB^{\tilde\mu}_E(M)$.
\end{proof}

\begin{theorem}
\label{thm-reg-quant}
Let $\mu$ be a weight bounded by a polynomial and $\tilde\eta\in\caD(Q)$. Then:
\begin{enumerate}
\item For any $f\in\caB^\mu_E(M)$ and $z\in M$, the element $\left(\Omega(f)\tilde\eta,\tilde\eta_z\right)$ is well defined in $E$.

\item The map $z\mapsto \left(\Omega(f)\tilde\eta,\tilde\eta_z\right)$ is in $\caS_E(M)$.

\item The linear map $\caB^\mu_E(M)\to\caS_E(M)$ given by $f\mapsto \left(\Omega(f)\tilde\eta,\tilde\eta_{\fois}\right)$ is continuous.

\item For any $y\in\gB M$, the map $\left(\Omega(L_y^\ast f)\tilde\eta,\tilde\eta_{\fois}\right)$ is in $\caS_E(M)$, and the map $y\mapsto \left(\Omega(L_y^\ast f)\tilde\eta,\tilde\eta_{\fois}\right)$ is in $\caB^{\tilde\mu}_{\caS_E(M)}(\gB M)$, where $\tilde\mu$ is the weight associated to $\mu$ by Lemma \ref{lem-reg-trans}.
\end{enumerate}
\end{theorem}

\begin{proof}
Remember that we use the notation \eqref{eq-prel-var} for the variables.
\begin{enumerate}
\item As in the proof of Lemma \ref{lem-qu-cont}, we integrate over the odd variables and obtain\footnote{as before, we do not write the constraints between the subsets $I,J,K,L$.} (using the change of variable $x_1\mapsto x_1+x_0$):
\begin{align*}
\left(\Omega(f)\tilde\eta,\tilde\eta_z\right) =&\int_Q\dd x_0\dd\xi_0\ \overline{(\Omega(f)\tilde\eta)(x_0,\xi_0)}\ast_{\xi_0}\tilde\eta_z(x_0,\xi_0)\\
=&\gamma\int \dd x_0\dd\xi_0\dd x_1\dd\xi_1\dd w_1\overline{f(x_1,\xi_1,w_1)}\dd\xi_2\\
& e^{-ia_0(2\omega_0(x_1-x_0,w_1)+\frac12\omega_1(\xi_1,\xi_0) -\frac{\alpha}{2}\omega_1(\xi_2,\xi_0)-\frac{\alpha+1}{2}\omega_1(\xi_1,\xi_2))}\\
&\overline{\tilde\eta(2x_1-x_0,\xi_1+\xi_2)} \ast_{\xi_0}\left[e^{ia_0(\omega_0(x-x_0,w)+\frac 12\omega_1(\xi,\xi_0))}\tilde\eta(x_0-x,\xi_0-\xi)\right].
\end{align*}
\begin{align}
\left(\Omega(f)\tilde\eta,\tilde\eta_z\right)=&\sum_{I,J,K,L}\alpha(I,J,K,L)\int\dd x_0\dd x_1\dd w_1\ \overline{f_I(x_1,w_1)}e^{-2ia_0\omega_0(x_1-x_0,w_1)} \overline{\tilde\eta_J(2x_1-x_0)}\nonumber\\
&\quad e^{ia_0\omega_0(x-x_0,w)}\tilde\eta_K(x_0-x)\xi^L\nonumber\\
=&\sum_{I,J,K,L}\alpha(I,J,K,L)\int\dd x_0\dd x_1\dd w_1\ \overline{f_I(x_1+x_0,w_1)}e^{-2ia_0\omega_0(x_1,w_1)} \overline{\tilde\eta_J(2x_1+x_0)}\nonumber\\
& \quad e^{ia_0\omega_0(x-x_0,w)}\tilde\eta_K(x_0-x)\xi^L\ .
\label{eq-reg-matr}
\end{align}

But $\overline{\tilde\eta_J}$ is bounded, $x_0\mapsto\tilde\eta_K(x_0-x)$ is a smooth function with compact support, and (using Lemma \ref{lem-reg-trans}) there exists a weight $\tilde\mu$ independent of $x_0$ and bounded by a polynomial such that $(x_1,w_1)\mapsto \overline{f_I(x_1+x_0,w_1)}$ is in $\caB^{\tilde\mu}_E(\gB M)$. Hence the function
\begin{equation*}
(x_1,w_1)\mapsto \int\dd x_0\ \overline{f_I(x_1+x_0,w_1)}\overline{\tilde\eta_J(2x_1+x_0)} e^{ia_0\omega_0(x-x_0,w)}\tilde\eta_K(x_0-x)
\end{equation*}
is in $\caB^{\tilde\mu}_E(\gB M)$. By using Theorem \ref{thm-reg-oscil}, we conclude that the element $\left(\Omega(f)\tilde\eta,\tilde\eta_z\right)$ belongs to $E$ (see Appendix \ref{sec-funcspaces} for the definition of the $E$-valued integral).

\item 
We now recall the definition of the operator $O_{(x_1,w_1)}$ used in the proof of Theorem \ref{thm-reg-oscil}, but in the variables $(x_1,w_1)$: $O_{(x_1,w_1)}f = (1-\frac1{4a_0^2}\Delta_{(x_1,w_1)})(\frac{1}{1+x_1^2+w_1^2}f)$. We insert the $k$-th power of the operator $(\frac{1}{1+x_1^2+w_1^2})(1-\frac1{4a_0^2}\Delta_{(x_1,w_1)})$ on the phase in \eqref{eq-reg-matr} (on which it is the identity) and we integrate by parts. This gives us terms involving factors of the form $O_{(x_1,w_1)}^k(\overline{ f_I \, \tilde \eta} )$. Just as in the proof of Theorem \ref{thm-reg-oscil} one can show that
\begin{multline*}
O^k_{(x_1,w_1)}(\overline{f_I(x_1+x_0,w_1)}\overline{\tilde\eta_J(2x_1+x_0)})=\\
\frac{1}{(1+x_1^2+w_1^2)^k}\sum_{\gamma_i}b_{\gamma_i}(x_1,w_1)\partial_{(x_1,w_1)}^{\gamma_1} \overline{f_I(x_1+x_0,w_1)}\partial_{x_1}^{\gamma_2}\overline{\tilde\eta_J(2x_1+x_0)},
\end{multline*}
for multi-indices $\gamma_i$ (a finite sum) and bounded functions $b_{\gamma_i}(x_0,w)$.
Now remember that, in order to investigate the nature of the function $(x,\xi,w)=z\mapsto \left(\Omega(f)\tilde\eta,\tilde\eta_z\right)$, we have to investigate the coefficients (functions of $x$ and $w$) with respect to an expansion into powers of $\xi$. What we thus have obtained so far, introducing the abbreviations $\hat f(x_1+x_0,w_1)=\partial_{(x_1,w_1)}^{\gamma_1} \overline{f_I(x_1+x_0,w_1)}$ (which is in $\caB^\mu_E(\gB M)$) and $\hat\eta(2x_1+x_0)=\partial_{x_1}^{\gamma_2}\overline{\tilde\eta_J(2x_1+x_0)}$ (which is in $\caD(\gB Q)$), is that these coefficients are finite linear combinations of terms of the form:
\begin{equation*}
\int\dd x_0\dd x_1\dd w_1 \frac{b_{\gamma_i}(x_1,w_1)}{(1+x_1^2+w_1^2)^k} e^{-2ia_0\omega_0(x_1,w_1)}\hat f(x_1+x_0,w_1)\hat\eta(2x_1+x_0)e^{ia_0\omega_0(x-x_0,w)}\tilde\eta_K(x_0-x).
\end{equation*}
Let us now use the operator $O_{x_0}$ for the variable $x_0$ just as in the proof of Lemma \ref{lem-prel-schw}: $O_{x_0}=(1-\frac1{a_0^2}\Delta_{x_0})\frac{1}{1+w^2}$ (note that this is not an exact equivalent of the operator $O_{(x_1,w_1)}$ because here the function $1/(1+w^2)$ commutes with the Laplacian). As it acts as the identity on the phase $e^{ia_0\omega_0(x-x_0,w)}$, we can insert the $l$-th power of it and integrate by parts. We can also prove (just as before) the equality
\begin{multline*}
O^l_{x_0}(\hat f(x_1+x_0,w_1)\hat\eta(2x_1+x_0)\tilde\eta_K(x_0-x))=\\
\frac{1}{(1+w^2)^l}\sum_{\delta_i}c_{\delta_i}\partial_{x_0}^{\delta_1}\hat f(x_1+x_0,w_1) \partial_{x_0}^{\delta_2}\hat\eta(2x_1+x_0) \partial_{x_0}^{\delta_3}\tilde\eta_K(x_0-x),
\end{multline*}
for multi-indices $\delta_i$ (a finite sum) and constants $c_{\delta_i}$. 
We now introduce yet another set of abbreviation by writing $F(x_1+x_0,w_1)=\partial_{x_0}^{\delta_1}\hat f(x_1+x_0,w_1)$ (which is in $\caB_E^\mu(\gB M)$), $\varphi(2x_1+x_0)=\partial_{x_0}^{\delta_2}\hat\eta(2x_1+x_0)$ (which is in $\caD(\gB Q)$), $\psi(x_0-x)=\partial_{x_0}^{\delta_3}\tilde\eta_K(x_0-x)$ (which is also in $\caD(\gB Q)$) and $b(x_1,w_1)=c_{\delta_i}b_{\gamma_i}(x_1,w_1)$ (which is in $\caB^1(\gB M)$). With these abbreviations the coefficients of $\left(\Omega(f)\tilde\eta,\tilde\eta_z\right)$ with respect to the powers of the variables $\xi$ are finite linear combinations of terms of the form:
\begin{multline*}
I(x,w)=\int\dd x_0\dd x_1\dd w_1 \frac{b(x_1,w_1)}{(1+x_1^2+w_1^2)^k(1+w^2)^l} e^{-2ia_0\omega_0(x_1,w_1)}e^{ia_0\omega_0(x-x_0,w)}\\
F(x_1+x_0,w_1)\varphi(2x_1+x_0)\psi(x_0-x).
\end{multline*}
To prove that $\left(\Omega(f)\tilde\eta,\tilde\eta_.\right)$ is a Schwartz function, we will show that the quantities $|I_{P,\beta}|_j$ are finite, where $P$ is any polynomial function on $\gB M$, $\beta$ any multi-index and $I_{P,\beta}$ defined as
\begin{equation*}
I_{P,\beta}=\int\dd x\dd w P(x,w) D^\beta_{(x,w)}I(x,w),
\end{equation*}
simply because a function is bounded if its derivative is integrable. We start by giving the explicit expression of $I_{P,\beta}$:
\begin{multline*}
I_{P,\beta}=\int\dd x\dd w \dd x_0\dd x_1\dd w_1 P(x,w) D^\beta_{(x,w)}\Big(\frac{b(x_1,w_1)}{(1+x_1^2+w_1^2)^k(1+w^2)^l} e^{-2ia_0\omega_0(x_1,w_1)}\\
e^{ia_0\omega_0(x-x_0,w)} F(x_1+x_0,w_1)\varphi(2x_1+x_0)\psi(x_0-x)\Big),
\end{multline*}
Next, as before, we write
\begin{multline*}
D^\beta_{(x,w)}\Big(\frac{1}{(1+w^2)^l} e^{ia_0\omega_0(x-x_0,w)}\psi(x_0-x)\Big)=\\
\frac{1}{(1+w^2)^l}\sum_\lambda d_\lambda(x-x_0,w)\partial_x^\lambda\psi(x_0-x) e^{ia_0\omega_0(x-x_0,w)},
\end{multline*}
with $\lambda$ a multi-index and $d_\lambda$ a function bounded by a polynomial of degree $|\beta|$. Using this expression we obtain
\begin{multline*}
I_{P,\beta}=\sum_\lambda\int\dd x\dd w \dd x_0\dd x_1\dd w_1 \frac{P(x,w)b(x_1,w_1)d_\lambda(x-x_0,w)}{(1+x_1^2+w_1^2)^k(1+w^2)^l} e^{-2ia_0\omega_0(x_1,w_1)}\\
e^{ia_0\omega_0(x-x_0,w)} F(x_1+x_0,w_1)\varphi(2x_1+x_0)\partial_x^\lambda\psi(x_0-x).
\end{multline*}
It follows from Lemma \ref{lem-reg-trans} that $|F(x_1+x_0,w_1)|_j\leq |F|_{j,0}^{(\tilde\mu)}\tilde\mu(x_1,w_1)\tilde\mu(x_0,0)$, hence $|F|_{j,0}^{(\tilde\mu)}\leq |F|_{j,0}^{(\mu)}$. By performing successively the changes of variables: $x_0\mapsto x_0-2x_1$ and $x\mapsto x_0-2x_1-x$, we obtain
\begin{multline*}
|I_{P,\beta}|_j\leq\sum_\lambda\int\dd x\dd w \dd x_0\dd x_1\dd w_1 |F|_{j,0}^{(\mu)} \frac{|\tilde\mu(x_1,w_1)\tilde\mu(x_0,0)P(x_0-2x_1-x,w)|}{(1+x_1^2+w_1^2)^k(1+w^2)^l}\\
|b(x_1,w_1)d_\lambda(-2x_1-x,w)\varphi(x_0)\partial_x^\lambda\psi(x)|.
\end{multline*}
Using appropriate values for $k$ and $l$, and because $\varphi$ and $\psi$ are functions with compact support, it follows that the above expression is finite, so that $\left(\Omega(f)\tilde\eta,\tilde\eta_.\right)$ is indeed a Schwartz function.

\item From 2 we deduce that $|\left(\Omega(f)\tilde\eta,\tilde\eta_{\fois}\right)|_{\nu,j,\beta}$ (which is bounded by terms of the same form but with seminorms $|\fois|'_{\nu,j,\beta}$, see Appendix \ref{sec-funcspaces}) is bounded by terms of the form $|f|_{j,\alpha}$, which proves continuity.

\item If we replace $f$ by $L_y^\ast f$ in 3 and if we use $\mu(y+y_0)\leq \tilde\mu(y)\tilde\mu(y_0)$ (see Lemma \ref{lem-reg-trans}), we obtain: 
\begin{equation*}
\forall\nu,j,\beta
\quad:\quad
|\left(\Omega(L_y^\ast f)\tilde\eta,\tilde\eta_{\fois}\right)|_{\nu,j,\beta}\leq \sum_{\alpha} C_\alpha \tilde\mu(y)|f|^{(\mu)}_{j,\alpha},
\end{equation*}
where the sum over the multi-indices $\alpha$ is finite, and the $C_\alpha$ are positive constants.
\end{enumerate}
\end{proof}

In the rest of this section we will no longer consider an arbitrary Fr\'echet space $E$, but only the particular case of a complex Fr\'echet algebra\footnote{This means in particular that the product is continuous for the topology determined by the seminorms. We do not assume that each seminorm is submultiplicative.} $(\algA,|\fois|_j)$. We also assume that $\algA$ is endowed with a continuous involution denoted by $a\mapsto\overline a$. On the space $L^2(Q)\otimes \algA$ we will use the topology defined by the seminorms $|\cdot|_j^{(h)}$ defined as 
\begin{equation*}
\forall\tilde\phi\in L^2(Q)\otimes \algA
\quad:\quad
|\tilde\phi|_j^{(h)}:=|\int_M\dd q\ \overline{\tilde\phi(q)}\ast_\xi\tilde\phi(q)|_j^{\frac12},
\end{equation*}
where $\ast_\xi$ denotes the Hodge operation with respect to the variable $\xi$ (see subsection \ref{subsec-superman}). Its completion will be denoted by $L^2(Q)\otimes_{(h)} \algA$ where we added $(h)$ to differentiate this completion from other types of topological tensor products (see \cite{Grothendieck:1966}).

A priori, for any $f\in\caB^\mu_\algA(M)$, the map $\Omega(f)$ is an unbounded operator from $L^2(Q)$ to $L^2(Q)\otimes_{(h)} \algA$. Theorem \ref{thm-reg-quant} shows that $\caS(Q)\subset \text{Dom}(\Omega(f))$. In fact, we will prove that for $\mu=1$ ($f\in\caB^1_\algA(M)$), the map $\Omega(f)$ is a bounded operator. We thus consider the space $\caL(L^2(Q),L^2(Q)\otimes_{(h)} \algA)$ of continuous linear applications, endowed with the topology of bounded convergence (see \cite{Bourbaki5:1966}).

\begin{theorem}
\label{thm-reg-contin}
The quantization map
\begin{equation*}
\Omega:\caB^1_\algA(M)\to \caL(L^2(Q),L^2(Q)\otimes_{(h)} \algA)
\end{equation*}
is continuous.
\end{theorem}

\begin{proof}
First, note that $\left(\tilde\varphi,\Omega(f)\tilde\psi\right)=\langle\tilde\varphi,\ast\Omega(f)\tilde\psi\rangle=(-1)^{|\tilde\varphi|(n+1)}\langle\ast\tilde\varphi,\Omega(f)\tilde\psi\rangle$, so that it suffices to consider $\langle\tilde\varphi,\Omega(f)\tilde\psi\rangle$ ($\tilde\varphi\mapsto (-1)^{|\tilde\varphi|(n+1)} \ast\tilde\varphi$ is a bijection). We then write $z_i=(x_i,\xi_i,w_i)\in M$ and $y_i=(x_i,w_i)\in \gB M$ and we denote by $(\fois,\fois)$ the scalar product of functions on $\gB M$. With these notations we use the resolution of the identity (see Theorem \ref{thm-prel-resol}) on $\langle\tilde\varphi,\Omega(f)\tilde\psi\rangle$ giving
\begin{align*}
\langle\tilde\varphi,\Omega(f)\tilde\psi\rangle &=\frac{1}{C^2\norm\tilde\eta\norm_2^4}\int\dd z_1\dd z_2\langle\tilde\varphi,\tilde\eta_{z_1}\rangle\langle\tilde\eta_{z_1},\Omega(f)\tilde\eta_{z_2}\rangle\langle\tilde\eta_{z_2},\tilde\psi\rangle\\
&=\sum_{I,J,K}\frac{\alpha(I,J,K)}{C^2\norm\tilde\eta\norm_2^4}\int\dd y_1\dd y_2 (\tilde\varphi_I,\tilde\eta_{y_1})(\tilde\eta_{y_1},\Omega(f_J)\tilde\eta_{y_2})(\tilde\eta_{y_2},\tilde\psi_K),\\
\end{align*}
where $I,J,K$ are subsets of $\{1,\dots,n\}$ with some unspecified constraints among them, where $\alpha(I,J,K)$ is a real number, and where $\tilde\eta$ is an element of $\caS(\gB M)$ with non-vanishing norm. We then can make the following estimates:
\begin{align}
|\langle\tilde\varphi,\Omega(f)\tilde\psi\rangle|_j \leq &\sum_{I,J,K}\frac{|\alpha(I,J,K)|}{C^2\norm\tilde\eta\norm_2^4}\int\dd y_1\dd y_2 |(\tilde\varphi_I,\tilde\eta_{y_1})||(\tilde\eta_{y_1},\Omega(f_J)\tilde\eta_{y_2})|_j |(\tilde\eta_{y_2},\tilde\psi_K)|\nonumber\\
\leq &\sum_{I,J,K}\frac{|\alpha(I,J,K)|}{C^2\norm\tilde\eta\norm_2^4} \left(\int\dd y_1\dd y_2 |(\tilde\varphi_I,\tilde\eta_{y_1})|^2|(\tilde\eta_{y_1},\Omega(f_J)\tilde\eta_{y_2})|_j\right)^{\frac12}\nonumber\\
&\left(\int\dd y_1\dd y_2 |(\tilde\eta_{y_1},\Omega(f_J)\tilde\eta_{y_2})|_j |(\tilde\eta_{y_2},\tilde\psi_K)|^2\right)^{\frac12}\nonumber\\
\leq &\sum_{I,J,K}\frac{|\alpha(I,J,K)|}{C^2\norm\tilde\eta\norm_2^4} \left(\int\dd y_1 |(\tilde\varphi_I,\tilde\eta_{y_1})|^2\right)^{\frac12} \sup_{y_1}\left(\int\dd y_2|(\tilde\eta_{y_1},\Omega(f_J)\tilde\eta_{y_2})|_j\right)^{\frac12}\nonumber\\
&\sup_{y_2}\left(\int\dd y_1 |(\tilde\eta_{y_1},\Omega(f_J)\tilde\eta_{y_2})|_j\right)^{\frac 12} \left(\int\dd y_2|(\tilde\eta_{y_2},\tilde\psi_K)|^2\right)^{\frac12}\label{eq-cst-inter}
\end{align}
Looking at the factor $\int\dd y_1 |(\tilde\varphi_I,\tilde\eta_{y_1})|^2$ we have
\begin{equation*}
\int\dd y_1 |(\tilde\varphi_I,\tilde\eta_{y_1})|^2=C'\norm\tilde\eta\norm_2^2\norm\tilde\varphi_I\norm^2_2\leq C'\norm\tilde\eta\norm^2_2\norm\tilde\varphi\norm^2_2,
\end{equation*}
where we used the resolution of the identity just on $\gB M$ (with another constant $C'$) and the fact that $\norm\tilde\varphi\norm_2^2=\sum_I\norm\tilde\varphi_I\norm_2^2$. In the same way one can show that $\int\dd y_2|(\tilde\eta_{y_2},\tilde\psi_K)|^2\leq C'\norm\tilde\eta\norm^2_2\norm\tilde\psi\norm^2_2$.

Next, by definition of $\tilde\eta_y$ and $\Omega$, we have
\begin{equation*}
(\tilde\eta_{y_1},\Omega(f_J)\tilde\eta_{y_2})=(\tilde\eta_{y_1},\tilde U(y_2)\Omega(L^\ast_{y_2}f_J)\tilde\eta)=(\tilde\eta_{y_2^{-1}y_1},\Omega(L^\ast_{y_2}f_J)\tilde\eta).
\end{equation*}
Using point 4 of Theorem \ref{thm-reg-quant}, we deduce that, for all $j$ and for all polynomial functions $\nu$ on $\gB M$, there exists constants $C'_\alpha>0$ (where $\alpha$ runs through a finite set of multi-indices) such that
\begin{equation*}
|(\tilde\eta_{y_1},\Omega(f_J)\tilde\eta_{y_2})|_j\leq\sum_\alpha C'_\alpha\frac{1}{|\nu(y_1-y_2)|}|f|_{j,\alpha}^{(1)},
\end{equation*}
simply because $\tilde\mu=1$ when $\mu=1$. When we plug all these inequalities in \eqref{eq-cst-inter} we obtain that, for all $j$, there exist constants $C_\alpha>0$ (and remember, $\alpha$ takes values in a finite set of multi-indices) such that 
\begin{equation*}
\forall f\in\caB_E^1(M)\ ,\ \forall\tilde\varphi,\tilde\psi\in L^2(Q)
\quad:\quad
|\left(\tilde\varphi,\Omega(f)\tilde\psi\right)|_j\leq \sum_\alpha C_\alpha\norm \tilde\varphi\norm_2\norm\tilde\psi\norm_2 |f|_{j,\alpha}.
\end{equation*}
\end{proof}

\subsection{Construction of the deformed product}
\label{subsec-product}

First, we consider the space $\caB^1_\gC(M)$ which contains the constant function $z\mapsto 1$. We assume in the following that the map $\Omega$ defined on this space (see Theorem \ref{thm-reg-contin}) is compatible with the unit, i.e. $\Omega(1)=1$. This corresponds to fix the constant $\gamma$ as
\begin{equation}
\gamma=\frac{(-1)^n}{r_0 r_1(1+\alpha)^n}\label{eq-ext-condgamma}
\end{equation}
in the expression of $\Omega$ \eqref{eq-qu-omf}. Remember that $r_0$ and $r_1$ have been defined in \eqref{eq-qu-fourier}.

Using the notation of the previous subsection with the Fr\'echet algebra $\algA$, we now introduce the deformed noncommutative product.

\begin{proposition}
\label{prop-prod-moy}
Using two independent parameters $\lambda$ and $\kappa$, we define the bilinear map $\star:\caD_\algA(M)\times\caD_\algA(M)\to C^\infty(M,\algA\otimes\superA)$ by 
\begin{multline}
(f_1\star f_2)(z)=\kappa\int\dd z_1\dd z_2\ f_1(z_1)f_2(z_2) e^{2ia_0(\omega_0(x_2-x,w_1)+\omega_0(x-x_1,w_2)+\omega_0(x_1-x_2,w))}\\
e^{2ia_0\lambda(\omega_1(\xi,\xi_1)+\omega_1(\xi_1,\xi_2)+\omega_1(\xi_2,\xi))},\label{eq-prod-moy}
\end{multline}
for any $f_1,f_2\in\caD_\algA(M)$ and any $ z\in M$, using the notation \eqref{eq-prel-var}. This map can be extended:
\begin{itemize}
\item to Schwartz functions $\star:\caS_\algA(M)\times\caS_\algA(M)\to\caS_\algA(M)$.

\item to the symbol calculus $\star:\caB^\mu_\algA(M)\times\caB_\algA^{\mu'}(M)\to\caB_\algA^{\tilde\mu\tilde\mu'}(M)$ by using the oscillating integral with $\mu$ and $\mu'$ weights bounded by a polynomial and $\tilde\mu$ and $\tilde\mu'$ the associated weights according to Lemma \ref{lem-reg-trans}.

\item as $\star:\caB^\mu_\algA(M)\times\caS_\algA(M)\to\caS_\algA(M)$ and $\star:\caS_\algA(M)\times\caB_\algA^{\mu}(M)\to\caS_\algA(M)$.
\end{itemize}
Moreover, in each case the extended bilinear map is continuous for the topologies involved.
\end{proposition}

\begin{proof}
\begin{itemize}
\item By using the continuity of the undeformed product of $\algA$ and operators $O^k_{x_1}$ and $O^l_{w_1}$ on the phase as in Theorem \ref{thm-reg-quant}, one finds that $|\int\dd z P(z) D^\beta_z (f_1\star f_2(z))|_j$ is bounded by a linear combination of seminorms of $f_1$ and $f_2$ in $\caS_\algA(M)$, for any polynomial function $P$ on $\gB M$ and any multi-index $\beta$. This proves the existence of the product $f_1\star f_2$, the fact that it is in $\caS_\algA(M)$, and the continuity of $\star$.

\item By a change of variables, we can express the product as:
\begin{equation*}
(f_1\star f_2)(z)=\kappa\int\dd z_1\dd z_2\ f_1(z_1+z)f_2(z_2+z) e^{2ia_0(\omega_0(x_2,w_1)-\omega_0(x_1,w_2)+\lambda\omega_1(\xi_1,\xi_2))}.
\end{equation*}
As we have the estimate $|\partial^\gamma(f_1)_J(x_1+x,w_1+w)|_{k}\leq |f_1|_{k,\gamma}^{(\mu)}\tilde\mu(x_1,w_1)\tilde\mu(x,w)$, it follows that $|f_1\star f_2|_{j,\beta}^{(\tilde\mu\tilde\mu')}$ is bounded by a linear combination of seminorms of $f_1$ and $f_2$.

\item This follows in the same way as the first extension.
\end{itemize}
\end{proof}

\begin{proposition}
\label{prop-product-idoper}
If we choose the values $\lambda=-\frac{(\alpha+1)^2}{4\alpha}$ and $\kappa=\frac{\gamma\alpha^n}{r_0(1+\alpha)^n}$, then the product defined in Proposition \ref{prop-prod-moy} corresponds to the product of operators via the quantization map $\Omega$: 
\begin{equation*}
\forall f_1\in\caB^\mu_\algA(M)\ \forall f_2\in\caB^{\mu'}_\algA(M)
\quad:\quad
\Omega(f_1\star f_2)=\Omega(f_1)\Omega(f_2),
\end{equation*}
In the sequel we will always assume that $\lambda$ and $\kappa$ have these values and we recall that $\gamma$ is given by \eqref{eq-ext-condgamma}.
\end{proposition}

\begin{proof}
After some elementary computations, by using Equations \eqref{eq-qu-omf} and \eqref{eq-prod-moy}, we obtain, for $\tilde\varphi\in \caS(Q)$,
\begin{multline*}
\Omega(f_1\star f_2)\tilde\varphi(q_0)=\gamma\kappa r_0\int\dd\xi\dd z_1\dd z_2f_1(z_1)f_2(z_2)\dd\xi_5\ e^{ia_0(2\omega_0(x_1-x_0,w_1)+2\omega_0(x_2-2x_1+x_0,w_2))}
\\
e^{ia_0(2\lambda\omega_1(\xi,\xi_1)+2\lambda\omega_1(\xi_1,\xi_2)+2\lambda\omega_1(\xi_2,\xi)+\frac12\omega_1(\xi,\xi_0)-\frac{\alpha}{2}\omega_1(\xi_5,\xi_0)-\frac12(\alpha+1)\omega_1(\xi,\xi_5))}
\\ 
\tilde\varphi(2x_2-2x_1+x_0,\xi+\xi_5),
\end{multline*}
and
\begin{multline*}
\Omega(f_1)\Omega(f_2)\tilde\varphi(q_0)=(-1)^n\gamma^2\int\dd\xi_3\dd z_1\dd z_2 f_1(z_1)f_2(z_2)\dd\xi_4\ e^{ia_0(2\omega_0(x_1-x_0,w_1)+2\omega_0(x_2-2x_1+x_0,w_2))}\\
e^{ia_0(\frac 12\omega_1(\xi_1,\xi_0)-\frac{\alpha}{2}\omega_1(\xi_3,\xi_0)-\frac12(\alpha+1)\omega_1(\xi_1,\xi_3)+\frac12\omega_1(\xi_2,\xi_1+\xi_3)-\frac{\alpha}{2}\omega_1(\xi_4,\xi_1+\xi_3)-\frac12(\alpha+1)\omega_1(\xi_2,\xi_4))}\\
\tilde\varphi(2x_2-2x_1+x_0,\xi_2+\xi_4).
\end{multline*}
In the above expressions, we used the notation $z_i=(x_i,\xi_i,w_i)\in M$ (for $i=1,2$), $\xi,\xi_3,\xi_4,\xi_5$ are odd coordinates of $M$ and $q_0=(x_0,\xi_0)\in Q$. It then suffices to apply the change of variables $(\xi,\xi_5) \mapsto (\xi_3=-\frac{1}{\alpha}\xi+\frac{1}{\alpha}\xi_1+\xi_5,\xi_4=\xi_5+\xi-\xi_2)$, whose Berezinian is $(-1)^n\frac{(1+\alpha)^n}{\alpha^n}$, to obtain the result.
\end{proof}

Note that if $\algA$ is unital, then $1\star 1=1$ (of course for the given values of $\lambda,\kappa,\gamma$). If $\alpha=1$, the product is given by the formula
\begin{equation*}
(f_1\star f_2)(z)=\kappa\int\dd z_1\dd z_2\ f_1(z_1)f_2(z_2) e^{-2ia_0(\omega(z_1,z_2)+\omega(z_2,z)+\omega(z,z_1))},
\end{equation*}
which is a unified expression in the even and odd variables.

\begin{proposition}
\label{prop-product-invol}
If $\algA$ is endowed with an involution (satisfying $\overline{a\fois b}=\overline{b}\fois\overline{a}$), then: 
\begin{equation*}
\forall f_1\in\caB^\mu_\algA(M)\ \forall f_2\in\caB^{\mu'}_\algA(M)
\quad:\quad
\overline{f_1\star f_2}=(-1)^{|f_1||f_2|}\overline{f_2}\star\overline{f_1},\qquad \Omega(f_1)^\dag=\Omega(\overline{f_1}).
\end{equation*}
\end{proposition}

\begin{proof}
This is a direct computation using the given expressions for the product \eqref{eq-prod-moy} and for the quantization map \eqref{eq-qu-omf}.
\end{proof}

\begin{lemma}
\label{lem-product-trcl}
For $f\in\caB_\algA^\mu(M)$, with $\mu$ a weight bounded by a polynomial, and for $z_1\in M$, the operator $\Omega(f)\Omega(z_1)$ is trace-class (see Definition \ref{def-prel-supertrace}).
\end{lemma}

\begin{proof}
For $\tilde\varphi\in\caS(Q)$ and $z=(x,\xi,w)\in M$, an elementary computation gives
\begin{multline*}
\int\dd\xi\langle\tilde\varphi_z,\Omega(f)\Omega(z_1)\tilde\varphi_z\rangle= \sum_{I,J,K}\beta(I,J,K)\int\dd x_0\dd x_2\dd w_2\overline{\tilde\varphi_I(x_0-x)} f_J(x_2,w_2)\\
\tilde\varphi_K(x_0-x+2x_1-2x_2) e^{2ia_0(\omega_0(x_2-x_0,w_2)+\omega_0(x_1-2x_2+x_0,w_1)+\omega_0(x_1-x_2,w))}.
\end{multline*}
We then apply the same techniques as used in the proof of Theorem \ref{thm-reg-quant}: using the operators $O_{x_2}$, $O_{w_2}$ and $O_{x_0}$, and applying the change of variables $x\mapsto x_0-x$ in
\begin{equation*}
\int\dd x\dd w|\int\dd\xi\langle\tilde\varphi_z,\Omega(f)\Omega(z_1)\tilde\varphi_z\rangle|_j,
\end{equation*}
we can insert arbitrary powers of $\frac{1}{1+w^2}$, $\frac{1}{1+x_0^2}$ and $\frac{1}{1+w_2^2}$ into the above integral, allowing us to show that it is finite.
\end{proof}

If one computes the supertrace (see Equation \eqref{eq-prel-supertr}) of the operator $\Omega(f)\Omega(z_1)$ (with $f\in\caB_\algA^\mu(M)$ and $z_1\in M$), one finds zero. Therefore, we will twist the supertrace by an odd Fourier transformation to define the Berezin transformation.

\begin{definition}
We define the Berezin transformation by $\tr(\Omega(f)\Omega(z_1)\caF_\beta)$, for $f\in\caB^\mu_\algA(M)$ and $z_1\in M$. This expression makes sense because of Lemma \ref{lem-product-trcl}. Here $\caF_\beta$ denotes the odd Fourier transformation with parameter $\beta$: 
\begin{equation}
\forall \tilde\varphi\in L^2(Q)
\quad:\quad
\caF_\beta\tilde\varphi(x_0,\xi_0)=\int\dd\xi\ e^{-\frac{ia_0\beta}{2}\omega_1(\xi,\xi_0)}\tilde\varphi(x_0,\xi).\label{eq-product-fourier}
\end{equation}
\end{definition}

\begin{proposition}
\label{prop-product-ber}
The Berezin transformation is equal, up to a numerical factor, to the identity. More precisely, if $\mu$ is a weight bounded by a polynomial, then
\begin{equation*}
\forall f\in\caB^\mu_\algA(M)
\quad:\quad
\tr(\Omega(f)\Omega(z_1)\caF_\beta) =r_1\left(\frac{\alpha-\beta}{1+\alpha}\right)^n (-1)^{n|f|}f(z_1).
\end{equation*}
\end{proposition}

\begin{corollary}
\label{cor-prod-inj}
\begin{itemize}
\item The quantization map $\Omega$ is injective on $\caB^\mu_\algA(M)$, for any weight $\mu$ bounded by a polynomial.

\item The product $\star:\caB^\mu_\algA(M)\times\caB_\algA^{\mu'}(M)\to\caB_\algA^{\tilde\mu\tilde\mu'}(M)$ is associative (needs Proposition \ref{prop-product-idoper} and the injectivity of $\Omega$).

\end{itemize}
\end{corollary}

\begin{example}
We compute here the deformed product at low odd dimensions. On the right hand side, the product $f_i\star g_j$ (for functions $f_i$ and $g_j$ of the even variables only) denotes the usual Moyal deformed product (see \eqref{eq-action-moyalprod}).

\begin{itemize}
\item For $n=1$, we have for any $f\in\caS(M)$ the decomposition $f(x,\xi)=f_0(x)+f_1(x)\xi$. And then the deformed product is given by the formula
\begin{equation}
(f\star g)(x,\xi)=(f_0\star g_0)(x)+\frac{i\alpha}{a_0(1+\alpha)^2}(f_1\star g_1)(x)+\big((f_0\star g_1)(x)+(f_1\star g_0)(x)\big)\xi.\label{eq-product-onedim}
\end{equation}

\item For $n=2$, we have for any $ f\in\caS(M)$ the decomposition $f(x,\xi)=f_0(x)+f_1(x)\xi^1+f_2(x)\xi^2+f_3(x)\xi^1\xi^2$. In this case the deformed product is given by the formula
\begin{multline*}
(f\star g)(x,\xi)=\Big[f_0\star g_0+\frac{i\alpha}{a_0(1+\alpha)^2}(f_1\star g_1 +f_2\star g_2)+\frac{\alpha^2}{a_0^2(1+\alpha)^4}f_3\star g_3+\big(f_0\star g_1+f_1\star g_0\\
-\frac{i\alpha}{a_0(1+\alpha)^2}(f_2\star g_3-f_3\star g_2)\big)\xi^1 +\big(f_0\star g_2+f_2\star g_0-\frac{i\alpha}{a_0(1+\alpha)^2}(f_3\star g_1-f_1\star g_3)\big)\xi^2\\
+\big(f_0\star g_3+f_3\star g_0+f_1\star g_2-f_2\star g_1\big)\xi^1\xi^2\Big](x).
\end{multline*}
Note that the algebra $(\caS(M),\star)$ has a $(\gZ_2)^2$-grading, which is isomorphic to the space of quaternions $\mathbb H$ when the even dimension $m$ is zero. This result can be generalized for arbitrary $n$, as we will see in the next Theorem.
\end{itemize}
\end{example}

\begin{theorem}
\label{thm-product-form}
The product introduced in Proposition \ref{prop-prod-moy} generates a Clifford algebra, when restricted to the odd coordinates of $M$. More precisely, if $\caS(M)$ is endowed with the deformed product, we have the following isomorphism of graded algebras:
\begin{equation*}
\caS(M)\approx Cl(n,\gC)\otimes \caS(\gB M),
\end{equation*}
where the grading of $\caS(M)$ corresponds to the usual $\gZ_2$-grading of $Cl(n,\gC)$, and where $\caS(\gB M)$ is endowed with the Moyal product.
\end{theorem}

\begin{proof}
The proof of this Theorem will be given in Appendix \ref{sec-finealg}.
\end{proof}

Note that this result is already given at the formal deformation level in older papers (see for instance \cite{Musson:2009}, where the deformed algebra is isomorphic to a Clifford-Weyl algebra). If the parameter $\alpha$ goes to 0, the part $\bigwedge\gR^n$ in $\caS(M)\simeq \bigwedge\gR^n\otimes \caS(\gB M)$ remains undeformed while the part $\caS(\gB M)$ is endowed with the Moyal product.

\begin{proposition}[Tracial identity]
\label{prop-product-tracial}
Let $\mu$ and $\mu'$ be weights bounded by a polynomial and let $f\in\caB^\mu_\algA(M)$ and $g\in\caB^{\mu'}_\algA(M)$ be such that $f\star g\in L^1_\algA(M)$. Then we have the equality
\begin{equation*}
\int_M\dd z\, (f\star g)(z)= \int_M\dd z\, f(z)g(z).
\end{equation*}
\end{proposition}

\begin{proof}
This follows from direct computation using the identities \eqref{eq-qu-fourier} and the given value of the coefficient $\kappa$.
\end{proof}

\begin{corollary}
\label{cor-product-trace}
$(\caD_{L^1}(M),\star)$ is an associative Fr\'echet $\gZ_2$-graded algebra, endowed with the integration as a supertrace: $\text{str}(f)=\int_M\dd z\, f(z)$ for $f\in \caD_{L^1}(M)$.
\end{corollary}

\begin{proof}
We refer to Appendix \ref{sec-funcspaces} for the definition of $\caD_{L^1}(M)$. Since $\caD_{L^1}(M)\subset \caD_{L^2}(M)$, $\caD_{L^1}(M)$ is an algebra for the usual supercommutative product. The deformed product is defined on it and the tracial identity of Proposition \ref{prop-product-tracial} shows that this space is stable under the deformed product. Moreover, the product is continuous for the topology of $\caD_{L^1}(M)$.
\end{proof}

\section{Universal Deformation Formula}
\label{sec-univ}

The deformed product introduced in section \ref{sec-defqu} will allow to deform algebras $\algA$ on which the supergroup $M$ acts, and the deformed product on $\algA$ (or rather on the subspace of its smooth vectors) will be called the Universal Deformation Formula (UDF). We first give in subsection \ref{subsec-act} the conditions that an action of the supergroup on a Fr\'echet algebra $\algA$ has to satisfy for the UDF to be applicable, and we show that the smooth vectors of $\algA$ for this action are dense in $\algA$. Next we associate in \ref{subsec-deffrech} to each smooth vector an element of $\caB^1_\algA(M)$. Using the deformed product of section \ref{sec-defqu}, we then construct a deformed product for the space of smooth vectors of $\algA$. Finally, in \ref{subsec-fam}, from a C${}^\ast$-superalgebra $\algA$, we can complete the smooth vectors' space endowed with the deformed product, into a new C${}^\ast$-superalgebra.

\subsection{Action of the Heisenberg supergroup}
\label{subsec-act}

Recall that $M=G/(\superA_0 Z)\simeq \gR^{m|n}$ (see subsection \ref{subsec-qu}). As $\superA_0 Z\subset G$ is normal, $M$ has a group structure which turns out to be abelian: $\forall z,z'\in M$, $z\fois z'= z+z'$.

Now let $\algA$ be a complex Fr\'echet algebra and let $\rho:M\times(\algA\otimes\superA)\to(\algA\otimes\superA)$ be an action of $M$, i.e.:
\begin{equation*}
\rho_{0}=\text{id}
\quad\mbox{and}\quad
\forall z_1,z_2\in M,\quad \rho_{z_1\fois z_2}=\rho_{z_1}\rho_{z_2},
\end{equation*}
such that $\forall z\in M$, $\rho_z:(\algA\otimes\superA)\to\algA\otimes\superA$ is an $\superA$-linear algebra-automorphism. By writing $z=(y,\xi)$ and expanding into powers of $\xi$: $\rho_{(y,\xi)}(a)=\sum_I\rho_y(a)_I\xi^I$ for $a\in\algA$, we assume also that $\rho$ is strongly continuous in the sense that
\begin{align*}
&\forall y\in\gB M,\quad \forall I,\quad\forall a\in\algA,\quad\rho_y(a)_I\in\algA\\
&\forall a\in\algA,\quad \rho^a:\ z\mapsto\rho_z(a)\text{ is continuous (for the DeWitt topology)},
\end{align*}
and that $\rho$ is subisometric:
\begin{equation*}
\exists C>0,\quad\forall a\in\algA,\quad \forall I,\quad\forall j,\quad \exists k,\quad \forall y\in\gB M, \quad |\rho_y(a)_I|_j\leq C|a|_k.
\end{equation*}
Note that one could interpret $\rho$ as an action of the Heisenberg supergroup $G$ on $\algA\otimes\superA$ with a trivial action of $K=\superA_0 Z$. See \cite{Warner:1972} for the non-graded case.

\begin{definition}
The set $\algA^\infty$ of smooth vectors of $\algA$ is defined as
\begin{equation*}
\algA^\infty=\{a\in\algA,\quad z\mapsto \rho_z(a)\in C^\infty(M,\algA\otimes\superA)\}.
\end{equation*}
\end{definition}

\begin{proposition}
The set of smooth vectors $\algA^\infty$ is dense in $\algA$ (for the topology of $\algA$).
\end{proposition}

\begin{proof}
\begin{itemize}
\item For any $a\in\algA$ and $f\in\caD(M)$, we define
\begin{equation*}
\tilde a=\int_M\dd z_0\ f(z_0)\rho_{z_0}(a).
\end{equation*}
As we integrate in particular over the odd variables in $z_0$, $\tilde a$ belongs to $\algA$. And then, for all $ z\in M$,
\begin{align*}
\rho_z(\tilde a)&=\int_M\dd z_0\ f(z_0)(-1)^{|z|(n+|f|)}\rho_z\rho_{z_0}(a)= \int_M\dd z_0\ f(z_0)(-1)^{|z|(n+|f|)}\rho_{z\fois z_0}(a)\\
&=\int_M\dd z_0\ f(z^{-1}\fois z_0)(-1)^{|z|(n+|f|)}\rho_{z_0}(a).
\end{align*}
Since $f\in\caD(M)$, $z\mapsto\rho_z(a)\in C^\infty(M,\algA\otimes\superA)$ and $\tilde a\in\algA^\infty$.

\item Now let $\chi\in\caD(\gB M)$ be such that $\int_{\gB M}\dd y\ \chi(y)=1$. If we define, for $p\in \gN^\ast$, the function $\chi_{\frac 1p}$ by $\chi_{\frac 1p}(y)=p^m\chi(py)$, then it gives us an approximation of the identity: $\chi_{\frac 1p}(y)\to_{p\to\infty}\delta(y)$.

If we define $f_p(y,\xi)=\chi_{\frac 1p}(y)\xi^{\{1,\dots,n\}}$ with $(y,\xi)\in M$, then $f_p\in\caD(M)$ and the elements
\begin{equation*}
\tilde a_p:=\int_M\dd y\dd\xi\ f_p(y,\xi)\rho_{(y,\xi)}(a)=\int_{\gB M}\dd y\ \chi_{\frac 1p}(y)\rho_{(y,0)}(a)
\end{equation*}
have the property $\lim_{p\to\infty}\tilde a_p=\rho_{(0,0)}(a)=a$ (in the topology of $\algA$) because $\rho$ is strongly continuous. For any $a\in\algA$ we thus have constructed a sequence $(\tilde a_p)$ of smooth vectors converging to $a$.
\end{itemize}
\end{proof}

\subsection{Deformation of Fr\'echet algebras}
\label{subsec-deffrech}

With notation as in subsection \ref{subsec-act} and using the action of $M$, we will construct a deformed product on $\algA^\infty$.

\begin{lemma}
For any smooth vector $a\in\algA^\infty$, the map $\rho^a=z\mapsto\rho_z(a)$ belongs to $\caB^1_{\algA^\infty}(M)$.
\end{lemma}

\begin{proof}
We start by recalling the definition of the set of bounded functions on $M$:
\begin{equation*}
\caB^1_\algA(M)=\{f\in C^\infty(M,\algA\otimes\superA),\, \forall D^\beta,\ \forall j,\ \forall I,\ \exists C_{j,\beta,I}>0,\ \forall y\in \gB M,\ |D^\beta f_I(y)|_j<C_{j,\beta,I}\}.
\end{equation*}
The differential operators $D^\beta$ can be realized as tensor products of vector fields on $M$ generated by elements of $\gB\kg_0$. We note $D^\beta=\tilde P$ with $P\in\caU(\gB \kg_0)$. Then, for any $g=(h,\xi)\in G$ and $X\in\gB \kg_0$, we have
\begin{equation*}
\tilde X_g.\rho^a=\frac{\dd}{\dd t}\rho_{g e^{tX}}(a)|_{t=0}=\rho_g(X.a),
\end{equation*}
where $X.a:=\frac{\dd}{\dd t}\rho_{e^{tX}}(a)|_{t=0}$. In the same way, for $P\in\caU(\gB \kg_0)$, we find $\tilde P_g.\rho^a=\rho_g(P.a)$.

Since $P\in\caU(\gB \kg_0)$ does not contain odd variables, we have for any $h\in\gB G$, $\tilde P_h.(\rho^a{}_I)=(\tilde P_h.\rho^a)_I=\rho_h(P.a)_I$. And thus, $\forall y\in\gB M$,
\begin{align*}
|\tilde P_y\rho^a{}_I|_j\leq \sup_{h\in\gB G}\{|\tilde P_h\rho^a{}_I|_j\} = \sup_{h\in\gB G}\{|\rho_h(P.a)_I|_j\}\leq C|P.a|_k
\end{align*}
for $C$ and a $k$ coming from the subisometry of $\rho$. This shows that $\rho^a\in\caB^1_\algA(M)$.

It is immediate to see that $\forall y\in\gB M$, $\forall a\in\algA$, $\forall I$, we have $\rho_y(a)_I\in\algA^\infty$, and thus $\rho^a\in\caB^1_{\algA^\infty}(M)$.
\end{proof}

\begin{corollary}
\label{cor-deffrech-frech}
$(\algA^\infty,\fois)$ is a Fr\'echet algebra for the seminorms
\begin{equation*}
|a|_{P,j}=\sup_{y\in\gB M,I}\{|\tilde P_y.(\rho^a{}_I)|_j\}=|\rho^a|_{j,\beta}^{(1)},
\end{equation*}
with $a\in\algA^\infty$ and $\tilde P=D^\beta$, where $P\in\caU(\gB \kg_0)$.
\end{corollary}

\begin{proof}
It is straightforward to adapt the proof in \cite{Warner:1972} to the graded case.
\end{proof}

\begin{lemma}
\label{lem-deffrech-ident}
We have the following identities (using notation as before):
\begin{itemize}
\item $\forall f_i\in\caB^1_\algA(M)$, $\forall z,z_0\in M$, we have $\rho_{z_0}((f_1\star f_2)(z))=(\rho_{z_0}(f_1)\star \rho_{z_0}(f_2))(z)$, where we define $\rho_{z_0}(f)(z)=\rho_{z_0}(f(z))$.

\item $\forall a,b\in\algA^\infty$, $\forall z\in M$, we have $(\rho^a\star\rho^b)(z)=((\rho_z\rho^a)\star(\rho_z\rho^b))(0)$.

\end{itemize}
\end{lemma}

\begin{proof}
This is a direct computation using the explicit form of the product \eqref{eq-prod-moy}.
\end{proof}

\begin{proposition}
\label{prop-deffrech-prodalg}
Using the product $\star$ \eqref{eq-prod-moy} on $\caB^1_{A^\infty}(M)$, we define a product $\star_\rho$  on $\algA^\infty$ by: 
\begin{equation*}
\forall a,b\in\algA^\infty
\quad:\quad
a\star_\rho b:=(\rho^a\star\rho^b)(0).
\end{equation*}
This product is associative and $(\algA^\infty,\star_\rho)$ is a Fr\'echet algebra. It is called the Universal Deformation Formula (UDF) of the Heisenberg Supergroup for Fr\'echet algebras $\algA$.
\end{proposition}

\begin{proof}
By using successively the two identities of Lemma \ref{lem-deffrech-ident}, one can compute that we have, $\forall a,b\in\algA^\infty$, $\forall z\in M$ ,
\begin{equation}
\rho_z(\rho^a\star\rho^b(0))=((\rho_z\rho^a)\star(\rho_z\rho^b))(0)=(\rho^a\star\rho^b)(z).\label{eq-deffrech-interm}
\end{equation}
We thus have, $\rho^{(\rho^a\star\rho^b)(0)}=\rho^a\star\rho^b$. To prove associativity we write, $\forall a,b,c\in\algA^\infty$,
\begin{equation*}
(a\star_\rho b)\star_\rho c=(\rho^{(\rho^a\star\rho^b)(0)}\star\rho^c)(0)=(\rho^a\star\rho^b\star\rho^c)(0),
\end{equation*}
where we used the associativity of the product \eqref{eq-prod-moy}.
\end{proof}

\subsection{\texorpdfstring{Deformation of C${}^\ast$-superalgebras}{Deformation of C*-superalgebras}}
\label{subsec-fam}

Let $\algA$ be a C${}^\ast$-superalgebra. It is in particular a Fr\'echet algebra so that we can use the results of subsection \ref{subsec-deffrech} for $\algA$.

We will consider the minimal tensor product $\caB(L^2(Q))\widehat{\otimes}\algA$, which means the completion with respect to the operator norm of $\caB(L^2(Q)\otimes\ehH)$ if $\algA$ is embedded in $\caB(\ehH)$. It is a C${}^\ast$-superalgebra since $L^2(Q)\otimes\ehH$ has the structure of Hilbert superspace (see Proposition \ref{prop-superhilbert-tensprod}).

We consider $\algA\otimes\superA$ as a graded vector space for the total degree (sum of the degree of $\algA$ and the one of $\superA$) of this tensor product. Then, we can endow it with a product and a superinvolution (see Definition \ref{def-supercst}): $\forall a,b\in\algA$, $\forall\eta,\eta'\in\superA$,
\begin{equation*}
(a\otimes\eta)\fois (b\otimes\eta')=(-1)^{|\eta||b|}(a\fois b)\otimes(\eta\eta'),\qquad (a\otimes\eta)^\dag=(a^\dag)\otimes\eta.
\end{equation*}

\begin{lemma}
$(\caB^1_\algA(M),\star)$ is a graded associative algebra for the total degree: if $f(y,\xi)=\sum_{I,j}f_{I,j}(y)\xi^I$ with $f_{I,j}(y)\in\algA_j$, then $|f_{I,j}|=|I|+j$. Moreover, the superinvolution on $\algA$ can be extended to a superinvolution for the total degree on $\caB^1_\algA(M)$.
\end{lemma}

\begin{proof}
Due to Corollary \ref{cor-prod-inj}, we already know that $(\caB^1_\algA(M),\star)$ is an associative algebra. It is immediate to see that the product $\star$ given by \eqref{eq-prod-moy} is compatible with the total grading. Since $\algA\otimes\superA$ is endowed with a superinvolution (see the beginning of this subsection), we can extend it to the graded algebra $\caB^1_\algA(M)$ by defining: $f^\dag(y,\xi)=\sum_{I,j}f_{I,j}(y)^\dag\xi^I$. Then, a direct computation shows that $\forall f,g\in\caB^1_\algA(M)$, $(f\star g)^\dag=(-1)^{|f||g|}g^\dag\star f^\dag$, for the total degrees $|f|$ and $|g|$ of the functions $f$ and $g$.
\end{proof}

\begin{proposition}
\label{prop-fam-om}
The quantization map
\begin{equation*}
\Omega:\caB_\algA^1(M)\to\caB(L^2(Q))\widehat{\otimes}\algA
\end{equation*}
is a homogeneous superinvolutive injective continuous morphism of algebras of degree 0.
\end{proposition}

\begin{proof}
The norm on the algebra $\caB(L^2(Q))\widehat{\otimes}\algA$ is given by
\begin{equation*}
\norm T\norm=\sup_{\Phi,\Psi\in (L^2(Q)\otimes\ehH)\smallsetminus\algzero}\frac{|\left(\Phi,T\Psi\right)|}{\norm\Phi\norm\,\norm\Psi\norm}.
\end{equation*}
We proceed here with the same philosophy as in the proof of Theorem \ref{thm-reg-contin}. We first notice that we can deal with the superhermitian scalar product. Then, we can make use two times of the resolution of the identity. For $\Phi,\Psi\in L^2(Q)\otimes\ehH$ and $f\in\caB_\algA^1(M)$, it gives
\begin{align*}
|\langle\Phi,\Omega(f)\Psi\rangle| = &|\sum_{I,J,K}\frac{\alpha(I,J,K)}{C^2\norm\tilde\eta\norm_2^4}\int\dd y_1\dd y_2 \langle(\Phi_I,\tilde\eta_{y_1}), (\tilde\eta_{y_1},\Omega(f_J)\tilde\eta_{y_2})(\tilde\eta_{y_2},\Psi_K)\rangle_{\ehH}|\\
= &\sum_{I,J,K}\frac{|\alpha(I,J,K)|}{C^2\norm\tilde\eta\norm_2^4}\int\dd y_1\dd y_2 \norm(\Phi_I,\tilde\eta_{y_1})\norm_{\ehH} \norm(\tilde\eta_{y_1},\Omega(f_J)\tilde\eta_{y_2})\norm_{\algA}\norm(\tilde\eta_{y_2},\Psi_K)\norm_{\ehH}
\end{align*}
which corresponds to the first line of Equation \eqref{eq-cst-inter}. By using two analogous arguments:
\begin{align*}
&\int\dd y_1 \norm(\Phi_I,\tilde\eta_{y_1})|_{\ehH}^2 =C'\norm\tilde\eta\norm_2^2\norm\Phi_I\norm^2\leq C'\norm\tilde\eta\norm^2_2\norm\Phi\norm^2,\\
&\norm(\tilde\eta_{y_1},\Omega(f_J)\tilde\eta_{y_2})\norm_{\algA}\leq\sum_\alpha C'_\alpha\frac{1}{|\nu(y_1-y_2)|}|f|_{\alpha},
\end{align*}
where $|f|_\alpha=\sup_{y\in\gB M}\sum_I \norm D^\alpha f_I(y)\norm_\algA$, we arrive at the conclusion:
\begin{equation*}
\forall f\in\caB_\algA^1(M)\ ,\ \forall\Phi,\Psi\in L^2(Q)\otimes\ehH
\quad:\quad
|\left(\Phi,\Omega(f)\Psi\right)|\leq \sum_\alpha C_\alpha\norm \Phi\norm \norm\Psi\norm |f|_{\alpha},
\end{equation*}
which shows the continuity of $\Omega$.

Furthermore, the algebra-morphism property is given by Proposition \ref{prop-deffrech-prodalg} and the injectivity by Corollary \ref{cor-prod-inj}. Proposition \ref{prop-product-invol} shows the superinvolutive property, whereas Equation \eqref{eq-qu-omf} expresses the fact that $\Omega$ is of degree 0.
\end{proof}

Let us now make the additional assumption on the action $\rho:M\times(\algA\otimes\superA)\to(\algA\otimes\superA)$ that it is compatible with the superinvolution and the total degree: 
\begin{equation}
\forall a\in\algA
\quad:\quad
(\rho^a)^\dag=\rho^{(a^\dag)},\qquad |\rho^a|=|a|,\label{eq-fam-supassum}
\end{equation}
just as we did for C${}^\ast$-superalgebra morphism in Definition \ref{def-supercst}.

\begin{lemma}
\label{lem-fam-rho}
With this additional assumption, the map $\rho:(\algA^\infty,\star_\rho)\to(\caB^1_\algA(M),\star)$ defined by: $\rho:a\mapsto\rho^a$, is a superinvolutive (injective) isometric morphism of algebras of degree 0.
\end{lemma}

\begin{proof}
Isometry is immediate from the definition of the seminorms on $\algA^\infty$ in Corollary \ref{cor-deffrech-frech}. The algebra-morphism property can be shown in the following way: $\forall a,b\in\algA^\infty$, $\forall z\in M$,
\begin{equation*}
\rho_z(a\star_\rho b)=\rho_z(\rho^a\star\rho^b(0))=(\rho^a\star\rho^b)(z),
\end{equation*}
using Equation \eqref{eq-deffrech-interm}. It follows that we have $\rho^{a\star_\rho b}=\rho^a\star\rho^b$.
\end{proof}

\begin{theorem}
\label{thm-fam-constr}
The map $\Omega\circ\rho:(\algA^\infty,\star_\rho)\to\caB(L^2(Q))\widehat{\otimes}\algA$ is an injective superinvolutive morphism of graded algebras of degree 0, and is continuous with respect to the Fr\'echet topology of $\algA^\infty$. Using this map we can define the norm
\begin{equation*}
\norm a\norm_\rho=\norm\Omega(\rho^a)\norm
\end{equation*}
on $\algA^\infty$. Denoting by $\algA_\rho$ the completion of $\algA^\infty$ with respect to this norm, $(\algA_\rho,\star_\rho,{}^\dag,\norm\fois\norm_\rho)$ is a C${}^\ast$-superalgebra.
\end{theorem}

\begin{proof}
This follows directly from Proposition \ref{prop-fam-om}, Lemma \ref{lem-fam-rho} and the fact that $\caB(L^2(Q))\widehat{\otimes}\algA$ is a C${}^\ast$-superalgebra. The isometry property of the superinvolution also holds: 
\begin{equation*}
\forall a\in\algA
\quad:\quad
\norm a^\dag\norm_\rho=\norm\Omega(\rho^{a^\dag})\norm=\norm\Omega(\rho^a)^\dag\norm= \norm\Omega(\rho^a)\norm=\norm a\norm_\rho.
\end{equation*}
\end{proof}

The expression of the product $\star_\rho$ given by Proposition \ref{prop-deffrech-prodalg}, together with the construction of the deformed C${}^\ast$-superalgebra of Theorem \ref{thm-fam-constr}, is called the Universal Deformation Formula (UDF) of the Heisenberg Supergroup for C${}^\ast$-superalgebras.

\section{The quantum supertorus}
\label{sec-torus}

In this section we apply the UDF on some geometric examples. We first consider in subsection \ref{subsec-cpcttriv} an action of the supergroup $M$ on a compact trivial supermanifold $X$ with certain conditions on this action. This induces in a natural way an action of $M$ on the C${}^\ast$-superalgebra $\caC(X)$, which we then deform into another C${}^\ast$-superalgebra but now non-supercommutative. As a particular example we consider the case where $\gB X$ is the 2-dimensional torus in \ref{subsec-deftorus}.

\subsection{Deformation of compact trivial supermanifolds}
\label{subsec-cpcttriv}

We consider a trivial compact supermanifold $X = X_o \times \gR^{0\vert q}$ of dimension $p|q$ (and thus $X_o$ is a supermanifold of dimension $p\vert 0$ completely determined by $\gB X_o = \gB X$ which is compact). We also consider a smooth action $\tau: M\times X\to X$ of the abelian supergroup $M\simeq\gR^{m|n}$ on $X$. As we have the direct product $X=X_o \times \gR^{0\vert q}$, any element $u\in X$ is a couple $u=(v,\eta)$, with $v\in X_o$ and $\eta\in\gR^{0|q}$. If we decompose an element $z\in M=\gR^{m|n}$ in even and odd coordinates $z=(y,\xi)$ with $y\in \gR^{m|0}$ and $\xi \in \gR^{0|n}$, we can decompose the two entries of $\tau_z(u) \in X_o \times \gR^{0|q}$ with respect to powers of the odd coordinates as follows: 
\begin{equation*}
\forall z=(y,\xi)\in M = \gR^{m|n}
\quad:\quad
\tau_z u=\left((\tau_y v)^0_{IJ}\xi^I\eta^J,(\tau_y v)^1_{IJ}\xi^I\eta^J\right),
\end{equation*}
with $(\tau_y v)^i_{IJ}$ smooth functions on $\gB X \times \gR^m$, such that $(\tau_y v)^0_{IJ}$ and $(\tau_y v)^1_{IJ}$ take values respectively in $\gB X$ and in $\gR^q$. Due to parity, we must have in particular $(\tau_y v)^i_{IJ}=0$ whenever $i+\vert I\vert +\vert J\vert = 1$.

We now make the additional assumption that $\forall (I,J)\neq (\emptyset,\emptyset)$ we have $(\tau_y v)^0_{IJ}=0$. We thus have $(\tau_y v)^0_{IJ}\xi^I\eta^J=(\tau_y v)^0_{\emptyset\emptyset}$, which we will shorten to $(\tau_y v)^0\in\gB X$. This means that we assume that the elements of the form $(0,\xi)\in M$ do not act on $X_o$. We also assume that every component $(\tau_y v)^{1,k}_{IJ}$ ($k\in\{1,\dots,q\}$) of $(\tau_y v)^1_{IJ}$ is uniformly bounded in $y$.

We finally define $\caC(X)\simeq C(\gB X)\otimes\bigwedge\gR^q$ to be the completion of $C^\infty(X)$ (complex super smooth functions) with respect to the norm $\norm f\norm=\sum_I\norm f_I\norm_\infty$ (see Remark \ref{rmk-superman-triv}). Note that $\caC(X)$ does not correspond to the space of continuous functions $X\to\superA$ for the DeWitt topology.

\begin{proposition}
\label{prop-cpcttriv-deform}
The algebra $\algA=\caC(X)$ is a C${}^\ast$-superalgebra. If we denote by $\rho:M\times (\algA\otimes\superA)\to (\algA\otimes\superA)$ the natural action of $M$ on $\algA$: 
\begin{equation*}
\forall f\in\caC(X)\ ,\ \forall z\in M\ ,\ \forall u\in X
\quad:\quad
\rho_z (f)(u)=f(\tau_{z^{-1}}u)=f(\tau_{-z}u)
\ ,
\end{equation*}
then we have the inclusion $C^\infty(X)\subset\algA^\infty$, and Theorem \ref{thm-fam-constr} applies, yielding a deformation $(\algA_\rho,\star_\rho,{}^\dag,\norm\fois\norm_\rho)$ of $\algA^\infty$. The deformed product is given by the following explicit formula: for all $ f,g\in\algA^\infty$,
\begin{equation}
(f\star_\rho g)(u)=\kappa\int\dd z_1\dd z_2\ f(\tau_{-z_1}u)g(\tau_{-z_2}u) e^{2ia_0(\omega_0(x_2,w_1)-\omega_0(x_1,w_2)+\lambda\omega_1(\xi_1,\xi_2))},\label{eq-cpcttriv-product}
\end{equation}
where $\kappa$ and $\lambda$ are given by Proposition \ref{prop-product-idoper}.
\end{proposition}

\begin{proof}
We first note that we have the inclusion $\caC(X)\subset L^\infty(X)$ and that $L^\infty(X)$ is a C${}^\ast$-superalgebra multiplicatively represented on $L^2(X)$ (see Example \ref{ex-supercst-linf}). Since $\caC(X)$ is a complete subalgebra of $L^\infty(X)$ which is closed with respect to the superinvolution (which here is complex conjugation), $\caC(X)$ is also a C${}^\ast$-superalgebra.

Next we check that the action $\rho:M\times(\algA\otimes\superA)\to(\algA\otimes\superA)$ satisfies the conditions given in subsections \ref{subsec-act} and \ref{subsec-fam}. Using notation as before we have: $\forall z=(y,\xi)\in M$, $\forall u=(v,\eta)\in X$, $\forall f\in \algA$,
\begin{multline}
\rho_z(f)(u)=\sum_K f_K\bigl((\tau_{-y}v)^0\bigr) \Bigl(\sum_{I,J}(\tau_{-y}v)^1_{IJ}(-\xi)^I\eta^J\Bigr)^K\\
= \sum_K f_K\bigl((\tau_{-y}v)^0\bigr) \prod_{k\in K}\sum_{I_k,J_k}(\tau_{-y}v)^{1,k}_{I_kJ_k}(-\xi)^{I_k}\eta^{J_k}\label{eq-cpcttriv-action}.
\end{multline}

\begin{itemize}
\item Since $\tau$ is an action, it follows immediately that we have $\forall z_1,z_2\in M$, $\rho_{z_1+z_2}=\rho_{z_1}\rho_{z_2}$ and $\rho_0=\text{id}$.

\item By definition of the product of functions we have:
\begin{equation*}
\forall f,g\in\algA\ ,\ \forall z\in M\ ,\ \forall u\in X
\quad:\quad
\rho_z(f\fois g)(u)=(f\fois g)(\tau_{-z}u)=f(\tau_{-z}u)g(\tau_{-z}u)=\rho_z(f)\rho_z(g).
\end{equation*}

\item Since $\tau$ is smooth, since $f\in\algA$, and because of \eqref{eq-cpcttriv-action}, the coefficient (a function!) of $\xi^I$ in $\rho_z(f)$: $\rho_y(f)_I$ belongs to $\algA$ for all $y\in \gB M$ and all $I$, and the map $z\mapsto\rho_z(f)$ is continuous with respect to the DeWitt topology.

\item $\forall y\in \gB M$, $\forall I$,
\begin{equation*}
\norm\rho_y(f)_I\norm=\sum_J\norm \sum_K f_K((\tau_{-y}\bullet)^0) \prod_{k\in K}\sum_{I_k,J_k}(\tau_{-y}\bullet)^{1,k}_{I_kJ_k}(-1)^\beta\norm_\infty,
\end{equation*}
where $\beta$ is some integer depending on $I_k,J_k$, and where we have, for the summation, the constraints $\bigsqcup_{k\in K} I_k=I$ and $\bigsqcup_{k\in K} J_k=J$. Then,
\begin{equation*}
\norm\rho_y(f)_I\norm\leq \sum_{J,K}\norm f_K\norm_\infty\norm \prod_{k\in K}\sum_{I_k,J_k}(\tau_{-y}\bullet)^{1,k}_{I_kJ_k}(-1)^\beta\norm_\infty.
\end{equation*}
Using that $(\tau_{-y} v)^{1,k}_{IJ}$ is uniformly bounded in $y,v,I,J,k$, we obtain the subisometry property: there exists $C>0$ such that 
\begin{equation*}
\forall y\in\gB M\ ,\ \forall I\ ,\ \forall f\in\algA
\quad:\quad
\norm\rho_y(f)_I\norm\leq C\sum_K\norm f_K\norm_\infty= C\norm f\norm.
\end{equation*}

\item Careful inspection of \eqref{eq-cpcttriv-action} tells us that the non-vanishing terms must satisfy $\sum_{k\in K}(|I_k|+|J_k|)=|K|$, which immediately implies that we have $|\rho^f|=|f|$.

\item As the terms $(\tau_{-y}v)^{1,k}_{IJ}$ are real, we have $\overline{\rho^f}=\rho^{\overline f}$.
\end{itemize}
We thus have shown that all hypotheses needed for Theorem \ref{thm-fam-constr} are satisfied: conditions on $\rho$ at the beginning of subsection \ref{subsec-act} as well as Equation \eqref{eq-fam-supassum}. Hence we can construct a deformed C${}^\ast$-superalgebra $(\algA_\rho,\star_\rho,{}^\dag,\norm\fois\norm_\rho)$. Note that for $f\in C^\infty(X)$, the map $z\mapsto\rho_z(f)$ is smooth, because the action $\tau$ is smooth.
\end{proof}
Note that in the above proof, if each derivative of $y\mapsto(\tau_{-y}v)^0$ is non-vanishing, the theorem of composition of functions implies that $\algA^\infty=C^\infty(X)$.

\subsection{Deformation of the supertorus}
\label{subsec-deftorus}

Let us now consider the special case of the trivial supertorus $X=\gT^{2|n}$, i.e., the (unique) compact trivial supermanifold of dimension $2|n$ such that $\gB X = \gB X_o = \gT^2 \cong \gR^2/\gZ^2$ and $X= X_o\times \gR^{0\vert n}$. We will describe $X$ by the global ``chart'' $(v^1,v^2,\eta^1, \dots,\eta^n)\in\superA_0^2 \times \superA_1^n$ with periodicity conditions
\begin{equation*}
(v^1+1,v^2,\eta)=(v^1,v^2,\eta)=(v^1,v^2+1,\eta)
\ .
\end{equation*}
On $X$ we define an action of $M=\gR^{2|n}$ by translations: 
\begin{equation*}
\forall z=(y,\xi)\in M\ ,\ \forall u=(v,\eta)\in X
\quad:\quad
\tau_z u=u+z=(v+y,\eta+\xi)
\ .
\end{equation*}
It follows that we have
\begin{equation*}
(\tau_y v)^0_{\emptyset\emptyset}=v+y,\qquad (\tau_y v)^{1,k}_{\{i\}\emptyset}=\delta_{ik},\qquad (\tau_yv)^{1,k}_{\emptyset\{j\}}=\delta_{jk},
\end{equation*}
for $i,j,k\in\{1,\dots,n\}$, while all other terms vanish. Hence the conditions of subsection \ref{subsec-cpcttriv} for $\tau$ are satisfied.
We thus can apply Proposition \ref{prop-cpcttriv-deform} to obtain that $\algA=\caC(\gT^{2|n})\simeq C(\gT^2)\otimes\bigwedge\gR^n$ is a C${}^\ast$-superalgebra and that we have the equality $\algA^\infty=C^\infty(\gT^{2|n})$ (note that each derivative of $y\mapsto v+y$ is non-vanishing). By defining the action $\rho$ of $M$ on $\algA$ by $\rho_z(f)(u)=f(\tau_{-z} u)$, we thus can deform $\algA^\infty$ into a noncommutative C${}^\ast$-superalgebra $(\algA_\rho,\star_\rho,{}^\dag,\norm\fois\norm_\rho)$. We will denote the deformed algebra by $\gT^{2|n}_{\theta}$ and call it the quantum supertorus.

Let us now describe, for $n=1,2$, the deformed product of generators of $C^\infty(\gT^{2|n})$. Our first observation is that $C^\infty(\gT^{2})$ is generated by $e^{2i\pi x}$ and $e^{2i\pi y}$ (changing notation to $x=v^1$ and $y=v^2$), because of the periodic conditions satisfied by elements of $C^\infty(\gT^{2})$.

\begin{itemize}
\item The case $n=1$: the generators are the functions $e^{2i\pi x}$, $e^{2i\pi y}$, and $\xi\in\gR^{0|1}$. A direct computation shows that we have (up to a multiplicative factor and introducing $\theta=1/a_0$):
\begin{align*}
e^{2i\pi y}\star e^{2i\pi x}&=e^{2i\pi\theta}e^{2i\pi x}\star e^{2i\pi y},& \xi\star\xi&=1,\\
e^{2i\pi x}\star\xi&=\xi\star e^{2i\pi x},& e^{2i\pi y}\star \xi&=\xi\star e^{2i\pi y}.
\end{align*}

\item The case $n=2$: here the generators are the functions $e^{2i\pi x}$, $e^{2i\pi y}$, and $\xi,\eta\in\gR^{0|2}$. Using the same notation as for the case $n=1$, we obtain (again up to a multiplicative factor):
\begin{align*}
e^{2i\pi y}\star e^{2i\pi x}&=e^{2i\pi\theta}e^{2i\pi x}\star e^{2i\pi y},& \xi\star\xi&=\eta\star\eta=1,\qquad \xi\star\eta=-\eta\star\xi,\\
e^{2i\pi x}\star\xi&=\xi\star e^{2i\pi x},& e^{2i\pi y}\star \xi&=\xi\star e^{2i\pi y},\\
e^{2i\pi x}\star\eta&=\eta\star e^{2i\pi x},& e^{2i\pi y}\star \eta&=\eta\star e^{2i\pi y}.
\end{align*}

\end{itemize}

\section{An application to a noncommutative Quantum Field Theory}

In this section we will re-interpret the renormalizability of a certain model of noncommutative quantum field theory. After having introduced a trace in subsection \ref{subsec-twtr}, which we will need to define an action functional, we interpret in \ref{subsec-action} the renormalizable action with harmonic term \cite{Grosse:2004yu} as a $\phi^4$-action on a deformed superspace. We show that the universal deformation formula is indeed universal for the $\phi^4$-action with respect to renormalizability only if an odd dimension is added to the space which is to be deformed.

\subsection{A twisted trace}
\label{subsec-twtr}

We have seen in Corollary \ref{cor-product-trace} that $\text{str(f)}=\int_M\dd z\, f(z)$ is a supertrace on the algebra $\caD_{L^1}(M)$. Motivated by subsection \ref{subsec-qu} and Proposition \ref{prop-product-ber}, we will twist this supertrace.

\begin{proposition}
The formula $\text{tr}=\frac{1}{r_1}\text{str}\circ\caF_1$ (see \eqref{eq-product-fourier}) defines a (non-graded) trace on the algebra $\caD_{L^1}(M)$. For $f\in\caD_{L^1}(M)$ the explicit expression is given by:
\begin{equation}
\text{tr}(f)=\int_{\gB M}\dd x\dd w\, f(x,0,w)
\ .
\label{eq-trace-tr}
\end{equation}

\end{proposition}

\begin{proof}
For $f,g\in\caD_{L^1}(M)$ we compute, using Lemma \ref{lem-superman-exp}:
\begin{align*}
\text{tr}(f\star g)&= r_0^2\kappa\int \dd z\dd\xi_0\, f(z)g(x,\xi_0,w)e^{2ia_0\lambda\omega_1(\xi,\xi_0)}\\
&=r_0^2 \kappa\sum_I\int\dd x\dd w\, f_I(x,w)g_I(x,w)(4ia_0\lambda)^{n-|I|}(-1)^{\frac{n(n+1)}{2}+\frac{|I|(|I|-1)}{2}},
\end{align*}
It follows immediately that we have $\text{tr}(f\star g)=\text{tr}(g\star f)$.
\end{proof}

Note however that the deformed product does not satisfy a tracial identity as in Proposition \ref{prop-product-tracial} with respect to this non-graded trace.

\subsection{The action of the noncommutative Quantum Field Theory}
\label{subsec-action}

For several years there has been an increasing interest in noncommutative quantum field theories as candidates for new physics beyond the existing theories of particles physics as the Standard Model. In noncommutative quantum field theory one considers an action functional of fields on a noncommutative space. In the case of the Euclidean Moyal space $\gR^m_\theta$ (which, in the context of this paper, corresponds to the case with odd dimension zero), the product $f\star_\theta g$ for $f,g\in\caS(\gR^m)$ is given by:
\begin{equation}
(f\star_\theta g)(x)=\frac{1}{(\pi\theta)^m}\int \dd y\dd z f(y)g(z)e^{\frac{2i}{\theta}(\omega(x,y)+\omega(y,z)+\omega(z,x))},\label{eq-action-moyalprod}
\end{equation}
where $\omega$ denotes the standard symplectic form and where, as before we write $a_0=\frac{1}{\theta}$. Note also that we have $x,y,z\in\gR^m$.

Since all derivatives are inner: $\partial_\mu\phi=[-\frac{i}{2}\wx_\mu,\phi]_\star$, with $\wx=\frac{2}{\theta}\omega(x,\fois)$, the standard action on this noncommutative space given by\footnote{We use the Einstein summation convention.}
\begin{equation*}
S(\phi)=\int\dd x\Big(\frac 12(\partial_\mu\phi)\star_\theta(\partial_\mu\phi)+\frac{M^2}{2}\phi\star_\theta \phi+\lambda\,\phi\star_\theta\phi\star_\theta\phi\star_\theta\phi\Big)
\end{equation*}
can be reexpressed as:
\begin{equation}
S(\phi)=\int\dd x\Big(\frac 12|[-\frac{i}{2}\wx_\mu,\phi]_\star|^2+\frac{M^2}{2}|\phi|^2+\lambda\, |\phi\star_\theta\phi|^2\Big),\label{eq-action-actnaif}
\end{equation}
where $M$ and $\lambda$ are the parameters of the theory. This action is not renormalizable because of a new type of divergence called Ultraviolet-Infrared (UV-IR) mixing. By adding a harmonic term to the action, Grosse and Wulkenhaar solved the problem of UV-IR mixing: the action
\begin{equation}
S(\phi)=\int\dd x\Big(\frac 12(\partial_\mu\phi)\star_\theta(\partial_\mu\phi)+\frac{\Omega^2}{2}(\wx_\mu\phi)\star_\theta(\wx_\mu\phi)+\frac{M^2}{2}\phi\star_\theta \phi+\lambda\,\phi\star_\theta\phi\star_\theta\phi\star_\theta\phi\Big)\label{eq-action-actharm}
\end{equation}
is renormalizable at all orders in perturbation \cite{Grosse:2004yu} in dimension $m=4$. See \cite{deGoursac:2009gh} for an introduction to the subject of renormalization of Euclidean noncommutative quantum field theories. There are several mathematical interpretations of the harmonic term (see \cite{deGoursac:2010zb} for a review), but one of the most promising uses superalgebras \cite{deGoursac:2008bd}. Moreover, it admits a geometrical interpretation of the gauge theory \cite{deGoursac:2007gq,Grosse:2007dm} associated to \eqref{eq-action-actharm}.
\medskip

Let us now provide some details on the interpretation exposed in \cite{deGoursac:2008bd}. We define $\algA=\algA_0\oplus\algA_1$ to be the $\gZ_2$-graded complex algebra given by $\algA_0=\algA_1=\caM$, where $\caM$ is the Moyal-Weyl algebra\footnote{the algebra of polynomials of $\gR^m$ endowed with the Moyal product \eqref{eq-action-moyalprod}} equipped with the following product: 
\begin{equation}
\forall a,b,c,d\in\caM
\quad:\quad
(a,b)\fois(c,d)=(a\star_\theta c+\gamma\, b\star_\theta d,a\star_\theta d+b\star_\theta c),\label{eq-action-superprod}
\end{equation}
where $\gamma$ is a real parameter. In \cite{deGoursac:2008bd} a differential calculus based on the graded derivations of $\algA$ was constructed, and the action \eqref{eq-action-actharm}, as well as its associated gauge action, has been obtained in terms of this differential calculus.

However, there was no explanation to the construction of $\algA$. Here, we will show that it is a deformation of $\gR^{m|1}$. Indeed, it can be shown that $\algA$ is the universal enveloping algebra of the complex Heisenberg superalgebra of dimension $m+1|1$: $\algA=\caU(\kg)$. Then, the deformed product of $\gR^{m|1}$ is given by \eqref{eq-product-onedim} (extended to polynomials), and up to a factor and a redefinition of $\gamma$, it corresponds exactly to \eqref{eq-action-superprod}.

\begin{proposition}
The action \eqref{eq-action-actharm} can be expressed in terms of the deformed product of $\gR^{m|1}$ and of the trace \eqref{eq-trace-tr}:
\begin{equation}
S(\phi)=\text{tr}\Big(\frac12\sum_\mu |[-\frac{i}{2}\wx_\mu\eta,\phi\eta]_\star|^2+\frac{M^2}{2}|\phi\eta|^{2}+\lambda\,|(\phi\eta)\star(\phi\eta)|^2\Big),\label{eq-action-actsuper}
\end{equation}
up to a redefinition of the parameters, and with $\eta=a+b\xi$, $a,b\in\gR$.
\end{proposition}

\begin{proof}
One can easily verify that we have
\begin{align*}
[-\frac{i}{2}\wx_\mu\eta,\phi\eta]_\star&=a^2\partial_\mu\phi+2ab(\partial_\mu\phi)\xi+\frac{\alpha\theta b^2}{(1+\alpha)^2}\wx_\mu\phi\\
(\phi\eta)\star(\phi\eta)&=(a^2+2ab\xi+\frac{i\alpha \theta b^2}{(1+\alpha)^2}) (\phi\star_\theta\phi).
\end{align*}
Since one can show, by integration by parts, that $\int\phi\wx_\mu\partial_\mu \phi=0$ for a real field $\phi$, we obtain that the action \eqref{eq-action-actsuper} is given by:
\begin{multline*}
S(\phi)= a^4\int\dd x\Big(\frac12(\partial_\mu\phi)\star_\theta(\partial_\mu\phi) +\frac{\alpha^2\theta^2 b^4}{2a^4(1+\alpha)^4}(\wx_\mu\phi)\star_\theta(\wx_\mu\phi)+\frac{M^2}{2 a^2}\phi\star_\theta \phi\\
+\lambda\left(1+\frac{\alpha^2\theta^2 b^4}{a^4(1+\alpha)^4}\right)\phi\star_\theta\phi\star_\theta\phi\star_\theta\phi\Big)
\end{multline*}
\end{proof}

Let us analyze the consequences of this result. One can easily see that \eqref{eq-action-actsuper} is exactly a $\phi^4$-type action, as in \eqref{eq-action-actnaif}, but now on the deformed superspace $\gR^{m|1}_\theta$. The parameter $\eta$ is the most general element of $L^\infty(\gR^{m|1})$, which is independent from the even variables. Hence $\eta$ does not add any degrees of freedom: the quantum field, as in the theory on $\gR^m_\theta$, is $\phi\in\caD_{L^1}(\gR^m)$.\question{Je ne comprend rien de la dernire phrase!} To summarize, we can say that to renormalize the scalar theory on $\gR^m_\theta$, one usually changes the $\phi^4$-action \eqref{eq-action-actnaif} by adding a harmonic term (in \eqref{eq-action-actharm}). In our approach, we do just the opposite: the framework is changed by the addition of an odd dimension to the deformed space, but the action is still the $\phi^4$-action. It is a kind of universality for the renormalizable $\phi^4$-action. However, the right transition from the commutative setting to the noncommutative one then should be the transition $\gR^m\to\gR^{m|1}_\theta$.

\section{Conclusion}

In this paper, after some recalls about supergeometry, we studied in section 2 the structure of $L^2(M)$ and $L^\infty(M)$, where $M$ is a trivial supermanifold, which led us to introduce the new notions of Hilbert superspaces and C*-superalgebras. These categories possess good and consistent properties for operator algebras. We then computed the coadjoint orbits of the Heisenberg supergroup $G$ and find that a generic orbit is diffeomorphic to $\gR^{m|n}$.

\medskip

We then built in section 3 an induced representation of the Heisenberg supergroup $G$ by using Kirillov's orbits method, and we saw that the notion of Hilbert superspace appears to be natural in the context of harmonic analysis on the Heisenberg supergroup. This representation allowed us to define a quantization map, which is a graded generalization of the Weyl ordering, and which is valued in operators on a Hilbert superspace. By using the oscillatory integral, we could extend the quantization map to functional symbol spaces defined on the coadjoint orbits of $G$ (i.e. on $\gR^{m|n}$). Consequently, we defined a deformed product on these symbol spaces, which corresponds to the operator product via the quantization map. It turns out that the space of Schwartz superfunctions on $\gR^{m|n}$, endowed with this deformed product, is then isomorphic to the tensor product of the (non-graded) Moyal algebra with the Clifford algebra $Cl(n,\gC)$.

\medskip

This non-formal deformation of the Heisenberg supergroup $G$ provides also a universal deformation formula (UDF). We indeed considered in section 4 a strongly continuous action of $G$ on an arbitrary Fr\'echet algebra $\algA$. Then, the regularity of the product introduced in section 3 allowed us to deform the product of the vectors of $\algA$ which are smooth for the action of $G$. Moreover, the deformation is compatible with the structure of C*-superalgebra, but not with the one of C*-algebra: from a C*-superalgebra $\algA$, one can construct a structure of deformed pre-C*-superalgebra on the smooth vectors of $\algA$.

\medskip

As a first application of our construction, we considered in section 5 a Heisenberg trivial supermanifold $X$. We next defined a space of continuous superfunctions $\caC(X)$, which carries a structure of C*-superalgebra and on which $G$ acts. Thanks to the UDF of section 4, we then deformed the supermanifold $X$ via its C*-superalgebra of functions. In particular, the example of the supertorus is described.

\medskip

As a second application, we reexpressed the renormalizable quantum field theory on the Moyal space with harmonic term as a standard scalar action with our deformed product, on the deformation of the superspace $\gR^{m|1}$. This provides a better understanding of the origin of this model with harmonic term. Moreover, our construction may also be applied on other noncommutative spaces to exhibit renormalizable field theories.

\appendix

\section{Functional spaces}
\label{sec-funcspaces}

In this appendix, we recall the notion of the space $L^1_E(M)$ \cite{Bourbaki6:1959}, where $(E,|\fois|_j)$ is a Fr\'echet space $(E,|\fois|_j)$ with its family of seminorms indexed by $j\in\gN$. We adapt it here to the type of supermanifolds $M$ considered in this article. $\caD_E(M)$ will denote the space of smooth $E$-valued functions on $M$. We also introduce the symbol calculus $\caB^\mu_E(M)$ for $M$.

\begin{definition}
Using notation previously introduced, we define the semi-norms:
\begin{equation*}
|f|_{1,j}=\sum_I \int_{\gB M}\dd x\dd w|f_I(x,w)|_j,
\end{equation*}
for $f\in\caD_E(M)$. Then, the space $L^1_E(M)$ can be defined as the completion of $\caD_E(M)$ with respect to these semi-norms (quotiented by the functions vanishing almost everywhere on $\gB M$).

For $f\in\caD_E(M)$ we define the integral $\int_M\dd zf(z)\in E'^{\ast}$ (the algebraic dual of the topological dual of $E$) by:
\begin{equation*}
\forall\phi\in E',\quad \langle \int_M\dd zf(z),\phi\rangle=\int_M\dd z\langle f(z),\phi\rangle,
\end{equation*}
where $\langle\fois,\fois\rangle$ denotes in this context the duality bracket. As $E$ is complete, this integral actually belongs to $E$: $\int_M\dd zf(z)\in E$.\question{C'est vrai que'on a l'galit $E'^* = E$ si $E$ est complet???}
\end{definition}

\begin{proposition}[see \cite{Bourbaki6:1959}]
The integration map $\int:\caD_E(M)\to E$ is continuous for the topology of $L^1_E(M)$. Hence it can be extended in a unique way to $L^1_E(M)$. Since $E$ is a Fr\'echet space, this extension also takes values in $E$:  $\forall f\in L^1_E(M)$~: $\int_M\dd zf(z)\in E$.\question{la dernire partie me semble trivial, car $E$ est complet. Est ce vraiment la peine de le rpter? C'est implicite dans le fait qu'on prolonge l'application $\int$ je pense.}
\end{proposition}

\begin{definition}
\label{def-mu-bounded-functions}
Let $\mu\in C^{\infty}(\gB M,\gR^\ast_+)$ be arbitrary.
\begin{itemize}
\item The space of $\mu$-bounded functions on $M$ is defined as:
\begin{equation*}
\caB^\mu_E(M)=\{f\in C^\infty(M,E\otimes\superA),\, \forall D^\alpha,\ \forall j,\ \forall I,\ \exists C_{j,\alpha,I}>0,\ \forall y\in \gB M,\ |D^\alpha f_I(y)|_j<C_{j,\alpha,I}\mu(y)\}.
\end{equation*}
If we endow it with the seminorms
\begin{equation*}
|f|^{(\mu)}_{j,\alpha}=\sup_{y\in\gB M}\{\frac{1}{\mu(y)}\sum_I|D^\alpha f_I(y)|_j\},
\end{equation*}
then 
$\caB^\mu_E(M)$ becomes a complex Fr\'echet space, isomorphic to $\caB^\mu_E(\gB M)\otimes\bigwedge \gR^n$. If no confusion is possible, we will denote the norm $|f|^{(\mu)}_{j,\alpha}$ also by $|f|_{j,\alpha}$.

\item A function $\mu\in C^{\infty}(\gB M,\gR^\ast_+)$ is called a weight if $\mu\in\caB^\mu_\gC(M)$.

\item The space of Schwartz functions \cite{Schwartz:1966} is defined by:
\begin{multline*}
\caS_E(M)=\{f\in C^\infty(M,E\otimes\superA),\, \forall D^\alpha,\ \forall j,\ \forall I,\ \forall\nu\in\text{Pol}(\gB M),\ \exists C_{j,\alpha,I,\nu}>0,\\
\forall y\in \gB M,\ |\nu(y)D^\alpha f_I(y)|_j<C_{j,\alpha,I,\nu}\}.
\end{multline*}
If we endow it with the seminorms
\begin{equation*}
|f|_{\nu,j,\beta}=\sup_{y\in\gB M}\{|\nu(y)|\sum_I|D^\beta f_I(y)|_j\},
\end{equation*}
where $\nu$ is a polynomial function on $\gB M$, then $\caS_E(M)$ becomes a Fr\'echet space, isomorphic to $\caS_E(\gB M)\otimes\bigwedge \gR^n$. Note that we also could have used the seminorms:
\begin{equation*}
|f|'_{\nu,j,\beta}=\int_{\gB M}\dd y|\nu(y)|\sum_I|D^\beta f_I(y)|_j.
\end{equation*}
\end{itemize}
\end{definition}

\begin{definition}
We define (see \cite{Schwartz:1966}) $\caD_{L^p}(M)$, for $1\leq p\leq \infty$, to be the space of complex smooth functions whose derivatives all belong to $L^p(M)$. It is a Fr\'echet space when we endow it with the seminorms:
\begin{equation*}
|f|_{p,\alpha}=\sum_I \left(\int_{\gB M}\dd x\dd w|D^\alpha f_I(x,w)|^p\right)^{\frac1p}.
\end{equation*}
Note that we have the the property: $\caD_{L^p}(M)\subset\caD_{L^q}(M)$ for $p\leq q$.

\end{definition}

\section{Fine graded division algebras}
\label{sec-finealg}

In this appendix we recall the concepts associated to fine graded algebras (see \cite{deGoursac:2008bd,deGoursac:2009gh}). Let $\Gamma$ be an abelian group, $\gK$ a field, $\gK^\ast$ its multiplicative group, and $\algA$ a $\Gamma$-graded associative $\gK$-algebra. We denote $\algA=\bigoplus_{\alpha\in\Gamma}\algA_\alpha$.

\begin{definition}
\begin{itemize}
\item $\algA$ is said to be a graded division algebra if any non-zero homogeneous element\question{j'ai chang ``component'' to ``element''. Si je me trompe, il faut clarifier!} of $\algA$ is invertible.

\item $\algA$ is called fine-graded if $\dim_\gK(\algA_\alpha)\leq 1$ for all $\alpha\in\Gamma$. If that is the case, the support of the grading is defined as $\Supp(\algA)=\{\alpha\in\Gamma,\, \algA_\alpha\neq \algzero\}$.
\end{itemize}
\end{definition}

The notion of a (Schur) multiplier will be useful to characterize fine-graded algebras.

\begin{definition}
\begin{itemize}
\item A factor set is an application $\sigma:\Gamma\times\Gamma\to\gK^\ast$ satisfying: 
\begin{equation*}
\forall\alpha,\beta,\gamma\in\Gamma
\quad:\quad
\sigma(\alpha,\beta+\gamma)\sigma(\beta,\gamma)=\sigma(\alpha,\beta)\sigma(\alpha+\beta,\gamma).
\end{equation*}

\item Two factor sets $\sigma$ and $\sigma'$ are equivalent if there exists a map $\rho:\Gamma\to\gK^\ast$ such that 
\begin{equation*}
\forall\alpha,\beta\in\Gamma
\quad:\quad
\sigma'(\alpha,\beta)=\sigma(\alpha,\beta)\rho(\alpha+\beta)\rho(\alpha)^{-1}\rho(\beta)^{-1}.
\end{equation*}
The quotient of the set of factor sets by this equivalence relation is an abelian group, and each equivalence class $[\sigma]$ is called a multiplier.
\end{itemize}
\end{definition}

\begin{remark}
If $\Gamma$ is finitely generated and $\gK$ is algebraically closed, then a factor set $\sigma$ is equivalent to the trivial factor set $1$ if and only if $\sigma$ is symmetric.
\end{remark}

\begin{proposition}
\label{prop-finealg-equiv}
A fine $\Gamma$-graded division algebra $\algA$ is totally characterized by its support $\Supp(\algA)$ which is a subgroup of $\Gamma$, and its factor set $\sigma:\Supp(\algA)\times\Supp(\algA)\to\gK^\ast$ defined by:
\begin{equation*}
e_\alpha\fois e_\beta=\sigma(\alpha,\beta)e_{\alpha+\beta},
\end{equation*}
with $\alpha,\beta\in\Supp(\algA)$ and $(e_\alpha)$ a homogeneous basis of $\algA$. Moreover, two fine $\Gamma$-graded division algebras are isomorphic if and only if their supports coincide and their factor sets are equivalent.
\end{proposition}

\begin{example}
\label{ex-finealg-cliff}
The complex Clifford algebras $Cl(n,\gC)$ are fine $(\gZ_2)^n$-graded division algebras. If we denote the generators by $e_0=\gone, e_1, \dots, e_n$, then the relations are given by
\begin{equation*}
(e_p)^2=\gone, \quad 1\le p\le n
\qquad\mbox{and}\qquad e_p\fois e_q=-e_q\fois e_p, \quad 1\le p<q\le n.
\end{equation*}
The associated multiplier can be represented the factor set:
\begin{equation*}
\sigma_{Cl}(\alpha,\beta)=\prod_{1\leq p<q\leq n}(-1)^{\alpha_p\beta_q}
\end{equation*}
for $\alpha,\beta\in(\gZ_2)^n$.
\end{example}

\begin{proof}[of Theorem \ref{thm-product-form}]
In order to show the isomorphism $\caS(M)\approx Cl(n,\gC)$ (when $m=0$), it suffices to show that $\caS(M)$ is a fine $(\gZ_2)^n$-graded division algebra whose associated factor set is equivalent to the one of $Cl(n,\gC)$ as described in Example \ref{ex-finealg-cliff}. To do so, we first observe that we have
\begin{align*}
(f\star g)(\xi)&=\kappa\int\dd \xi_1\dd \xi_2\ f(\xi_1)g(\xi_2) e^{ic(\xi\fois\xi_1+ \xi_1\fois\xi_2+\xi_2\fois\xi)}\\
&= \kappa\int\dd \xi_1\dd \xi_2\sum_{I,J} f_I g_J\xi_1^I\xi_2^J\sum_{k=0}^n\frac{(ic)^k}{k!}(\xi\fois\xi_1+\xi_2\fois\xi+ \xi_1\fois\xi_2)^k,
\end{align*}
where $c=4a_0\lambda$, and $I,J$ are summed over subsets of $\{1,\dots,n\}$. We then use the identity
\begin{equation*}
(a+b+c)^k=\sum_{0\leq p+q\leq k}\frac{k!}{p!q!(k-p-q)!}a^p b^q c^{k-p-q},
\end{equation*}
valid for complex numbers $a,b,c$, to express the coefficient $c_{IJ}$ of $\kappa f_I g_J$ in the above expression, giving us
\begin{equation*}
c_{IJ}=\int\dd \xi_1\dd \xi_2\xi_1^I\xi_2^J\sum_{k=0}^n\ \sum_{0\leq p+q\leq k}\frac{(ic)^k}{p!q!(k-p-q)!}(\xi_1\fois\xi_2)^{k-p-q}(\xi\fois\xi_1)^p(\xi_2\fois\xi)^q.
\end{equation*}
Now note that we have
\begin{align*}
(\xi_1\fois\xi_2)^p
&=
\sum_{i_1,\dots,i_p}\xi_1^{i_1}\xi_2^{i_1}\dots\xi_1^{i_p}\xi_2^{i_p}
=
\sum_{L\,:\,|L|=p}\ \sum_\sigma \xi_1^{\sigma(1)}\xi_2^{\sigma(1)}\dots \xi_1^{\sigma(p)}\xi_2^{\sigma(p)}
\\
&= 
\sum_{L\,:\, |L|=p}\ \sum_\sigma (-1)^{\frac{p(p-1)}{2}+2|\sigma|}\xi_1^{L}\xi_2^{L} 
=
p! \sum_{L\,:\, |L|=p} (-1)^{\frac{p(p-1)}{2}} \xi_1^{L}\xi_2^{L},
\end{align*}
where $\sigma$ is summed over all bijections from $\{1,\dots,p\}$ to $L$. Using this, we get
\begin{multline*}
c_{IJ}
=
\int\dd \xi_1\dd \xi_2\sum_{k=0}^n\ \sum_{0\leq p+q\leq k}(ic)^k\sum_{K_i} (-1)^{\frac{k(k-1)}{2}+p+q+k+pk+(k-p+|J|)|I|}\eps(J,K_1\cup K_3)
\\
\eps(K_1,K_3)\eps(I,K_1\cup K_2)\eps(K_1,K_2)\eps(K_2,K_3)\xi_2^{J\cup K_1\cup K_3}\xi_1^{I\cup K_1\cup K_2}\xi^{K_2\cup K_3},
\end{multline*}
where the sum over the subsets $K_i$ are constrained by the conditions $|K_1|=k-p-q$, $|K_2|=p$ and $|K_3|=q$, and where we used the definition of the product $\xi^I\fois\xi^J=\eps(I,J)\xi^{I\cup J}$. When we perform the integration over $\xi_1$ and $\xi_2$, only the terms with $K_1=\complement(I\cup J)$, $K_2=J\setminus(I\cap J)$ and $K_3=I\setminus(I\cap J)$ contribute. If we denote $d=|I\cap J|$, then we finally obtain
\begin{equation*}
c_{IJ}=(ic)^{n-d}(-1)^{\frac{d(d+1)}{2}+\frac{n(n+1)}{2}+|I|d}\eps(I\setminus (I\cap J),J\setminus(I\cap J))\eps(I\cap J,(I\cup J)\setminus(I\cap J))\xi^{(I\cup J)\setminus(I\cap J)}.
\end{equation*}
On the other hand, the multiplier associated to the Clifford algebra $Cl(n,\gC)$ is given by
\begin{equation*}
\sigma_{Cl}(I,J)=(-1)^{\sharp\{(i,j)\in I\times J,\, i<j\}},
\end{equation*}
for $I,J\subset \{1,\dots,n\}$. Using the fact that $\eps(I,J)=(-1)^{\sharp\{(i,j)\in I\times J,\, i>j\}}$ whenever $I\cap J=\emptyset$, we deduce that
\begin{equation*}
\sigma_{Cl}(I,J)=(-1)^{\frac{d(d+1)}{2}+|I||J|}\eps(I\setminus (I\cap J),J\setminus(I\cap J))\eps(I\cap J,(I\cup J)\setminus(I\cap J)).
\end{equation*}
By choosing $\rho(I)=(ic)^{\frac{|I|}{2}-n}(-1)^{\frac{n(n+1)}{2}}$, we obtain 
\begin{equation*}
\forall I,J
\quad:\quad
c_{IJ}=\sigma_{Cl}(I,J)\rho((I\cup J)\setminus(I\cap J))\rho(I)^{-1}\rho(J)^{-1}\xi^{(I\cup J)\setminus(I\cap J)},
\end{equation*}
which shows, using Proposition \ref{prop-finealg-equiv}, that the $(\gZ_2)^n$-graded algebras $\caS(M)$ and $Cl(n,\gC)$ are isomorphic.
\end{proof}

\end{document}